\documentclass[12pt]{article}
\usepackage{pluripotential}
\usepackage{setspace}

\bibliography{Density_analytic_sing_types_12.01.21.bib}

\date{\vspace{-.7cm}}

\title{\vspace{-1.2cm} The closures of  test configurations and algebraic singularity types }
\author{Tam\'as Darvas and Mingchen Xia}
\begin{document}
\maketitle
\begin{abstract}
Given a Kähler manifold $X$ with an ample line bundle $L$, we consider the metric space of finite energy geodesic rays  associated to the Chern class $c_1(L)$. We characterize rays that can be approximated by ample test configurations. At the same time, we also characterize the closure of algebraic singularity types among all singularity types of quasi-plurisubharmonic  functions, pointing out the very close relationship between these two seemingly unrelated problems.

By Bonavero's holomorphic Morse inequalities, the arithmetic and non-pluripolar volumes of algebraic singularity types coincide. We show that in general the arithmetic volume dominates the non-pluripolar one, and equality holds exactly on the closure of algebraic singularity types. Analogously, we give an estimate for the Monge--Ampère energy of a general finite energy ray in terms of the arithmetic volumes along its Legendre transform. Equality holds exactly for rays approximable by test configurations.

Various other cohomological and potential theoretic characterizations are given in both settings. As applications, we give a concrete formula for the non-Archimedean Monge--Ampère energy in terms of asymptotic expansion, and show the continuity of the projection map from $L^1$ rays to non-Archimedean rays. 
\end{abstract}

\begin{spacing}{0.95}
\tableofcontents
\end{spacing}

\section{Introduction}

We fix notation and terminology for the entire paper. We consider $X$ a compact Kähler manifold of dimension $n$ with an ample line bundle $L$. We pick a positive smooth Hermitian metric $h$ on $L$ and let $\omega := \frac{\mathrm{i}}{2\pi} \Theta(h)>0$  be the background Kähler form of $X$, where $\Theta(h)$ denotes the curvature form of the Hermitian metric $h$. 

By an $\omega$-plurisubharmonic ($\omega$-psh) function $u \in \PSH(X,\omega)$, we understand a quasi-plurisubharmonic function on $X$, such that $\omega + \ddc u : = \omega + \frac{\mathrm{i}}{2\pi}\partial \bar \partial u \geq 0$, where $\mathrm{d} = \partial + \bar \partial$ and $\mathrm{d}^{\mathrm{c}} = \frac{\mathrm{i}}{4\pi}(-\partial + \bar \partial)$. We will often use $\omega_u$ shorthand for $\omega + \ddc u$.

Two potentials $u,v \in \PSH(X,\omega)$ have the same singularity type if $u - C \leq v \leq u+C$ for some $C >0$, inducing an equivalence relation $u \simeq v$, with equivalence classes $[u]=[v] \in \mathcal S$, called \emph{singularity types}. It turns out that this latter space has a natural pseudometric structure $(\mathcal S, d_\mathcal S)$, introduced in \cite{DDNL5} (recalled in Section~\ref{subsec:mecL1geod}).

For $u \in \PSH(X,\omega)$, by $H^0(X,L^k \otimes \mathcal{I}(ku)) \subseteq H^0(X, L^k)$ we denote the space of sections $s$ satisfying the $L^2$ integrability condition   $\int_X h^k(s,s)e^{-ku}\,\omega^n  <\infty$. We also denote 
\[
h^0(X,L^k \otimes \mathcal{I}(ku)):= \dim_{\mathbb{C}} H^0(X,L^k \otimes \mathcal{I}(ku))\,.
\]

A major theme in K\"ahler geometry is to relate algebraic objects to analytic ones. In this work we address two such problems. First, we give  cohomological and potential theoretic characterizations for $L^1$ geodesic rays in the space of Kähler metrics that lie in the closure of test configurations. Second, we  characterize the closure of algebraic singularity types in $\mathcal S$, with respect to the complete pseudometric  $d_\mathcal S$. Potentials with algebraic singularity types are among the nicest ones one could hope for in practice (see Definition~\ref{def: anal_alg_sing}).

According to our work, it is most natural to treat both of these seemingly different problems at the same time, with our final answers also paralleling each other on many different levels.  We now present one facet of this, with slight abuse of precision. 

Given $\varphi \in \PSH(X,\omega)$ with algebraic singularity type $[\varphi] \in \mathcal S$,  the arithmetic and non-pluripolar volumes coincide, according to the singular holomorphic Morse inequalities of Bonavero \cite{Bon98} (see Theorem~\ref{thm: analytic_sing_type_formula}):
\begin{equation}\label{eq: BonaId}
\lim_{k\to\infty} \frac{1}{k^n}h^0(X,L^k \otimes \mathcal{I}(k\varphi)) =\frac{1}{n!} \int_X \omega_{\varphi}^n.
\end{equation}
All volumes in this work, in particular the one on the right hand side above,  are interpreted in the non-pluripolar sense of Guedj--Zeriahi \cite{GZ07}, \cite{BEGZ10} (see \eqref{eq: non_pluripolar_def} below). Both the left and right hand sides only depend on the singularity type $[\varphi]$ (\cite[Theorem~1.1]{WN19}). 

In Theorem~\ref{main_thm: arith_plurip_vol_deficit} we show that in \eqref{eq: BonaId} the limit on the left exists for \emph{all} $\varphi \in \PSH(X, \omega)$ and dominates the right hand side in general. Moreover, $[\varphi] \in \mathcal S$ lies in the $d_\mathcal S$-closure of algebraic singularity types if and only if Bonavero's identity \eqref{eq: BonaId} holds for $\varphi$.

Paralleling the above, in \cite{DL20} the authors introduced a metric $d_1^c$ on the space of $L^1$ geodesic rays $\mathcal R^1$ (recalled in Section~\ref{subsec:mecL1geod}), making $(\mathcal R^1,d_1^c)$ a complete geodesic metric space. As is well known, to any ample test configuration one can associate a geodesic ray, a construction going back to Phong--Sturm \cite{PS07}. We show that a geodesic ray $\{r_t\}_t \in \mathcal R^1$ is in the $d_1^c$-completion of the space of rays induced by test configurations if and only if $\hat r_\tau \in \PSH(X,\omega)$ satisfies Bonavero's identity \eqref{eq: BonaId} for all $\tau \in \mathbb R$, where $\hat r_\tau := \inf_{t>0} (r_t - t \tau)$ is the Legendre transform of the ray $\{r_t\}_t$. In particular, the ray $\{r_t\}_t \in \mathcal R^1$ is approximable by test configurations if and only if the singularity types $[\hat r_\tau] \in \mathcal S$ are approximable by algebraic singularity types! This parallels previous characterizations of approximable rays in the non-Archimedean context from \cite{BBJ21}.

We refer to Theorem~\ref{mainthm: on_rays} and Theorem~\ref{main_thm: arith_plurip_vol_deficit} for additional cohomological and potential theoretic characterizations complementing each other in both settings.

In addition, in Theorem~\ref{main_thm: Non-Archimedean} we express the non-Archimedean Monge--Ampère energy of a ray as the first order term in the expansion of the non-Archimedean Donaldson's $\mathcal L$-functional, a result paralleling \cite[Theorem~A]{BB10}.

\paragraph{The closure of the rays induced by test configurations.} 

Let $d_1$ be the metric on the space of smooth Kähler potentials $\mathcal{H}_\omega:= \{u \in C^\infty(X) : \omega +\ddc u >0\}$ associated with the $L^1$ Finsler metric \cite{Da15}:
\[
\|\psi\|_1 := \frac{1}{V}\int_X |\psi|\,\omega_u^n\,, \quad u \in \mathcal{H}_\omega \textnormal{ and } \psi \in T_u \mathcal{H}_\omega\,,
\]
where $V = \int_X \omega^n$ is the total volume. 
By $\mathcal{E}^1$ we denote the $d_1$-completion of $\mathcal{H}_\omega$, that can be identified with a finite energy space studied by Guedj--Zeriahi \cite{GZ07}. Then $(\mathcal{E}^1,d_1)$ is a complete geodesic metric space. Its geodesics are limits of solutions to a degenerate complex Monge--Ampère equation \cite[Theorem~2]{Da15}.

By $\mathcal{R}^1$ we denote the space of $d_1$-geodesic rays $\{r_t\}_{t \geq 0}$ in $\mathcal{E}^1$ emanating from $r_0 = 0 \in \mathcal{H}_\omega$. As shown in \cite[Theorem~1.3, Theorem~1.4]{DL20}, $\mathcal{R}^1$ has a a natural chordal metric $d_1^c$ (see \eqref{eq: d_1^c_def}), compatible with $d_1$, making $(\mathcal{R}^1, d_1^c)$ a complete geodesic metric space.

Of special importance is the subspace $\mathcal{T} \subseteq \mathcal{R}^1$, composed of the rays induced by ample test configurations \cite{PS07}, \cite{PS10}, \cite{CT08}. Similarly, one can consider the bigger subspace $\mathcal{F} \subseteq \mathcal{R}^1$, the space of rays induced by filtrations \cite{RWN14}. In this work it is advantageous to think of ample test configurations as special kind of filtrations on the ring of sections of $(X,L)$, and we refer to Section~\ref{subsec:nonAchi} for details on this. Understanding the closures $\overline{\mathcal T}$ and $\overline{\mathcal F}$ is one of the main problems we take up in this work.

The well-known Monge--Ampère energy $\mathrm{I}(\cdot) : \mathcal{E}^1 \to \mathbb{R}$, whose Euler--Lagrange equation is simply the Monge--Ampère equation (see \eqref{eq: MA_en_def}), is linear along $d_1$-geodesics. One can simply consider its radial version $\Irad[\cdot]:\mathcal{R}^1 \to \mathbb{R}$ defined by the following slope
\begin{equation}\label{eq:defIrad}
\Irad[r_t] =\mathrm{I}(r_1) = \lim_{t \to \infty} \frac{\mathrm{I}(r_t)}{t}\,.
\end{equation}
Another quantity that is linear along a ray $\{r_t\}_t$ is the supremum of potentials (Lemma \ref{lem: suplinear}). For simplicity, we will often assume that $\sup_X r_t =0$ for $t \geq 0$, and such rays will be called \emph{sup-normalized}. Note that all our results hold in an appropriate form without normalization, even when these are not specifically mentioned.

First, in Theorem~\ref{thm: max_test_curve_ray_duality}, we develop ideas from \cite{RWN14} further, and show a precise formula for the radial Monge--Amp\`ere energy of  sup-normalized rays $\{r_t\}_t \in \mathcal{R}^1$:
\begin{equation}\label{eq: L^1 slope_dual_masses_intr}
\Irad[r_t] =  \int_{-\infty}^{0} \bigg(\frac{\int_X \omega_{{\hat r}_\tau}^n}{\int_X \omega^n}-1\bigg) \,\mathrm{d}\tau\,, \quad t > 0\,,
\end{equation}
where $\hat r_\tau \in \PSH(X,\omega)$ for $\tau < 0$ is the Legendre transform of the ray:
\[
\hat r_\tau := \inf_{t\geq 0} (r_t - t\tau)\,.
\] 

We attempt to approximate the non-pluripolar volumes in the integrand of \eqref{eq: L^1 slope_dual_masses_intr} using arithmetic volumes (in the spirit of Bonavero's identity \eqref{eq: BonaId}).  In our first main result we show that this fails in general, and it works exactly for rays in the $d_1^c$-closure of $\mathcal T$:

\begin{theorem}[Theorem~\ref{thm: closure_of_test_conf}, Corollary~\ref{cor: Thm1.1cor}]\label{mainthm: on_rays}
For $\{r_t\}_t \in \mathcal{R}^1$ with $\sup_X r_t = 0$ we have
\begin{equation}\label{eq: mainthm_on_rays}
\lim_{k\to\infty} \int_{-\infty}^0 \bigg( \frac{h^0(X,L^k \otimes \mathcal{I}(k {\hat r}_\tau))}{h^0(X,L^k)}  - 1\bigg)\,\mathrm{d}\tau \geq \Irad[r_t] \,.
\end{equation}
Moreover, equality holds in \eqref{eq: mainthm_on_rays} if and only if the following equivalent conditions hold:
\vspace{0.15cm}\\
\vspace{0.15cm}\Rom{1} $\displaystyle\lim_{k\to\infty} \frac{1}{k^n}h^0(X,L^k \otimes \mathcal{I}(k\hat r_\tau)) =\frac{1}{n!} \int_X \omega_{\hat r_\tau}^n$, $\tau <0 $.\\
\vspace{0.15cm}\Rom{2} $\PrIv[\hat r_\tau]=\hat r_\tau$, $\tau \leq 0$. \\
\vspace{0.15cm}\Rom{3} $\{r_t\}_t \in \overline{\mathcal{T}}$.\\
\vspace{0.15cm}\Rom{4} $\{r_t\}_t \in \overline{\mathcal{F}}$.
\end{theorem}
A ray satisfying condition~\Rom{3} is called a \emph{maximal} geodesic ray in the work \cite{BBJ21}, giving non-Archimedean characterizations of $\overline{\mathcal T}$ recalled below. Here we do not use this terminology, to avoid potential confusion with other notions of maximality.

In condition~\Rom{2} the operator $\PrIv[u]$ is the $\veq$-envelope of the singularity type $[u] \in \mathcal S$, used explicitly and implicitly in \cite{KS20} and  \cite{Cao14} respectively:
\begin{equation}\label{eq: PrI_def}
\PrIv[u] = \sup \left\{\, \chi \in \PSH(X,\omega) : \chi \leq 0 \textnormal{ and } \mathcal{I}(c \chi) = \mathcal{I}(c u) \textnormal{ for all } c > 0   \,\right\}\,,
\end{equation}
where $\mathcal{I}(u)$ is simply the multiplier ideal sheaf of a quasi-psh function $u$ on $X$.

As part of showing \eqref{eq: mainthm_on_rays}, we will argue that the limit on the left hand side exists. We think of  \Rom{1} as the \emph{cohomological characterization} of the closure of $\mathcal T$. On the other hand, we think of \Rom{2} as the \emph{potential theoretic} characterization. 

The equivalence between  \Rom{3}  and  \Rom{4}  indicates that uniform notions of K-stability with respect to test configurations and filtrations are very likely the same (c.f. \cite[Question~1.12]{CC3}). To fully confirm this, one needs to show that elements of $\mathcal F$ can be approximated by $\mathcal T$ while also preserving the slope of the K-energy functional, as predicted  by \cite[Conjecture~1.5]{Li20}. Of course, due to the examples of \cite{ACGTF08}, one might still expect that relative K-stability with respect to $\mathcal F$ and $\mathcal T$ are different notions. 

It is not hard to see that for many of the rays constructed in \cite{Da17} there is strict inequality in \eqref{eq: mainthm_on_rays}, implying $\overline{\mathcal T} \subsetneq \mathcal R^1$. This strict containment was noticed in \cite[Remark~5.9]{BBJ21} using non-Archimedean methods. However, as a result of condition \Rom{1} above and our Theorem~\ref{thm: max_test_curve_ray_duality} (iv) (that allows for flexible construction of $L^1$ rays using test curves)  the containment $\overline{\mathcal T} \subsetneq \mathcal R^1$ is seen to be nowhere $d_1^c$-dense (c.f. \cite[Question~1.10]{CC3}).

Finally, let us put our Theorem~\ref{mainthm: on_rays} in historical context, and discuss the possible connection with the general version of the Yau--Tian--Donaldson conjecture, seeking to characterize existence of constant scalar curvature Kähler (cscK) metrics cohomolgous to $\omega$ in terms of algebro-geometric properties of the bundle $(X,L)$. Despite the difficulties arising due to infinite dimensionality, and the underlying fourth order PDE, by now we have a comprehensive understanding on the analytic side (see \cite{BDL17}, \cite{CC1}, \cite{CC2}, \cite{CC3}, \cite{DL20}), allowing to characterize existence of cscK metrics in terms of uniform geodesic stability along $C^{1,\bar 1}$-rays of the space of Kähler metrics, yielding the essentially optimal version of what was conjectured by Donaldson \cite{Do99}.

Similarly, with the development of the non-Archimedean toolbox, we have a very good understanding of the algebraic side as well (see \cite{BBJ21}, \cite{BHJ17}, \cite{BJ18}, \cite{BHJ19}, \cite{De16}), allowing not only to embed test configurations into $\mathcal R^1$ (along with their invariants), but to also keeping track of algebraic invariants using non-Archimedean metrics, an intermediate notion lying between the algebraic and analytic data. 

The remaining step in the variational program for the Yau--Tian--Donaldson conjecture is to understand what $L^1$ rays are approximable by ample test configurations, while also preserving the slope of the radial K-energy in the limit. This is the connection point with our characterization theorem above, though here we completely ignored the behavior of the K-energy in the approximation process. 

During the final stages of writing up our work we learned of the preprint of C. Li \cite{Li20}, who proved that $L^1$ rays with bounded radial K-energy are in $\overline{\mathcal T}$. Though not a characterization of $\overline{\mathcal T}$, this result is more closely lined up with the variational program, and it is an intriguing prospect to examine the relationship between our results  and the ones in  \cite{Li20}.

\paragraph{Non-Archimedean interpretation.} The non-Archimedean approach to singularities in pluripotential theory developed in \cite{BFJ08}, \cite{BBJ21} will play a crucial role in our discussion (especially in the form of valuative criteria), and we mention here how Theorem~\ref{mainthm: on_rays} can be interpreted in this context.

In this approach $\mathcal{T}$ can be identified with $\mathcal{H}^\NA$, the space of Fubini--Study metrics on the Berkovich analytification  $(X^{\NA},L^{\NA})$  with respect to the trivial valuation on $\mathbb{C}$. 
On the other hand, the closure $\overline{\mathcal{T}}$ can naturally be identified with the space of finite energy metrics on $(X^{\NA},L^{\NA})$, leading to a characterization of $\overline{\mathcal T}  =\mathcal{E}^{1,\NA}$ in the non-Archimedean context \cite{BBJ21}, in addition to the ones given in  Theorem~\ref{mainthm: on_rays}.

Given an arbitrary ray $\{r_t\}_t \in \mathcal{R}^1$, in \cite{BBJ21} the authors introduce a natural projection 
\[\Pi: \mathcal R^1 \to  \overline{\mathcal{T}}=\mathcal{E}^{1,\NA} \subset \mathcal R^1\, ,\]
satisfying $r_t \leq \Pi(r)_t$ and one can think of $\{\Pi(r)_t\}_t$ as the closest ray to $\{r_t\}_t$ that is approximable by test configurations (see Section~\ref{subsec:rayfiltandapp} for more details). Using $\Pi$, one can conveniently introduce the non-Archimedean Monge--Ampère energy as follows:
\[
\mathrm{I}^{\NA}\{r_t\} := \Irad[\Pi(r)_t]\,.
\]
The original definition is given by means of the non-Archimedean Monge--Ampère measures introduced in \cite{CL06}, \cite{CLD12} that only depend on the non-Archimedean data $r^\NA$ (see \cite{BJ18} and references therein). Here we show that $\Pi$ is $d_1^c$-continuous, and give the following expansion interpretation for $\textrm{I}^{\NA}$: 
\begin{theorem}[Theorem~\ref{thm: Pi_cont}, Corollary~\ref{thm: Non-Archimedean}]\label{main_thm: Non-Archimedean} The map $\Pi: \mathcal R^1 \to  \overline{\mathcal{T}}$ is $d_1^c$-continuous. Moreover, for any sup-normalized $\{r_t\}_t \in \mathcal{R}^1$ we have
\begin{equation}\label{eq: INA_expand}
\mathrm{I}^{\NA}\{r_t\}=\mathrm{I}\{\Pi(r)_t\} = \lim_{k\to\infty} \frac{n!}{Vk^n}\int_{-\infty}^0 \big( h^0(X,L^k \otimes \mathcal{I}(k\hat{r}_\tau))-h^0(X,L^k)\big)\,\mathrm{d}\tau\,.
\end{equation}

\end{theorem}
The integral in \eqref{eq: INA_expand} can be interpreted as $\Lk^{\NA}\{r_t\}$, the non-Archimedean analogue of Donaldson's $\mathcal{L}$-functional (see \eqref{eq:LNADef} and Proposition~\ref{prop:Lkarctononarc}). 
Theorem~\ref{main_thm: Non-Archimedean} says that the leading order term in the expansion of $\Lk^{\NA}\{r_t\}$ is given by the non-Archimedean Monge--Ampère energy. This is the non-Archimedean analogue of \cite[Theorem~3.5]{BFiM14}, where based on \cite{Don05} and  \cite{BB10}, it is proved that Donaldson's $\mathcal{L}$-functional from \cite{Don05} admits an expansion whose leading order term is given by the usual Monge--Ampère energy of $\mathcal E^1$. 


Similar flavour results in the non-Archimedean setting were obtained in \cite[Theorem~A]{BE21} and \cite[Theorem~A]{BGGJKM20} under different assumptions on the ground field and for continuous metrics. It would be interesting to see if   one could extend their results to finite energy metrics in case of trivially valued base fields of characteristic $0$, using our Theorem~\ref{main_thm: Non-Archimedean}.

\paragraph{The closure of the space of algebraic singularity types.} Finally we discuss approximation in the space of singularity types $\mathcal S$. We start with precisely defining algebraic/analytic singularity types.

\begin{definition}\label{def: anal_alg_sing}
We say that $[\psi]$ is an \emph{algebraic singularity type} (notation: $[\psi] \in \mathcal{Z} \subseteq \mathcal{S}$), if there exists $c \in \mathbb Q_+$ and around every point of $X$ there exists a Zariski open set $U$ and $f_j \in \mathcal{O}(U)$ algebraic, such that
$\psi|_U - \frac{c}{2} \log \left( \sum_{j=1}^k |f_j|^2\right)$ is smooth. 

We say that $[\psi]$ is an \emph{analytic singularity type} (notation: $[\psi] \in \mathcal{A} \subseteq \mathcal{S}$), if there exists $c \in \mathbb R_+$ and around every point of $X$ there exists an open set $U$ (with respect to the analytic topology) and $f_j \in \mathcal{O}(U)$, such that
$\psi|_U - \frac{c}{2} \log \left( \sum_{j=1}^k |f_j|^2\right)$ is locally bounded. 
\end{definition}

Many different conventions are in place regarding the definition of analytic/algebraic singularity types in the literature (see \cite[Definition 1.10]{De12}, \cite[Definition 2.3.9]{MM07} or  \cite[(4)]{RWN17}). Out of all possible definitions, our choice of $\mathcal Z$ is the smallest family one can consider, and $\mathcal A$ is perhaps the biggest. As we will show below, for purposes of approximation, using $\mathcal A$ or $\mathcal Z$ does not make a difference.

The study of partial Bergman kernels in connection with approximation in K\"ahler geometry has a long history, having both potential theoretic and spectral theoretic applications (see \cite{DPS01}, \cite{Bo02}, \cite{DP04}, \cite{Ra13}, \cite{Cao14}, \cite{Dem15}, \cite{RWN17}, \cite{RoSi17}, \cite{ZeZh19} in a very fast expanding literature). Even in case of smooth potentials $u \in \mathcal H_\omega$ one can use approximation by Bergman kernels for various geometric purposes,  going back to questions of Yau \cite{Ya86}, early work of Tian \cite{Ti88, Ti90}, and many others \cite{Ca97}, \cite{Ze98}, \cite{Lu00}.

Since many of the important invariants involved only depend on the singularity type of the potentials, in \cite{DDNL5} the authors  introduced a pseudometric $d_\mathcal S$ on $\mathcal S$, to study  the effectiveness of the approximation methods in the literature. $d_\mathcal S$-convergence implies convergence of all mixed complex Monge--Ampere masses \cite[Lemma~3.7]{DDNL5}, together with semicontinuity of multiplier ideal sheaves \cite[Theorem~1.3]{DDNL5}. Being also complete in the presence of positive mass \cite[Theorem~1.1]{DDNL5}, $d_\mathcal S$ seems well suited for this purpose. 

We refer to Section~\ref{subsec:mecL1geod} for a detailed discussion on the $d_\mathcal S$ metric, as well as the paper \cite{DDNL5}. We only mention the following double inequality of  \cite[Lemma~3.4 and Proposition~3.5]{DDNL5}, giving intuition about what $d_\mathcal{S}$-convergence means:
\[
d_\mathcal{S}([u],[v]) \leq  \sum_{j=0}^n \Big(2 \int_X \omega^j \wedge \omega^{n-j}_{\max(u,v)} - \int_X \omega^j \wedge \omega^{n-j}_{u} - \int_X \omega^j \wedge \omega^{n-j}_{v} \Big)\leq C d_\mathcal{S}([u],[v])\,,
\]
where $C>1$ only depends on $\dim X$. 

The pseudometric $d_\mathcal{S}$ is slightly degenerate, however $d_{\mathcal{S}} ([\phi],[\psi])=0$ implies that the singularities of $\phi,\psi$ are essentially indistinguishable (for example all Lelong numbers, multiplier ideal sheaves, and mixed complex Monge--Ampère masses need to agree), so in many ways this is a blessing in disguise. 

In our last main result we prove the inequality between arithmetic and non-pluripolar volumes for general $\omega$-psh functions, complementing \eqref{eq: BonaId} (c.f. \cite[Theorem~1.2]{Bo02}, \cite[Proposition~1.1]{Cao14}). Moreover, in the presence of positive mass, we show that Bonavero's formula holds exactly for the $d_\mathcal S$-closure $\overline{\mathcal Z}$. 

We also give a potential theoretic characterization for elements of $\overline{\mathcal Z}$ in terms of the coincidence locus of $\PrIv[\cdot]$ (defined in \eqref{eq: PrI_def}) and its analytic counterpart $P[\cdot]$:
\[
P[u] := \usc\left(\sup \{\,v \in \PSH(X,\omega) : [v]=[u] \textup{ and } v \leq 0 \,\}\right)\,.
\]

\begin{theorem}[Theorem~\ref{thm: arith_plurip_vol_deficit}]\label{main_thm: arith_plurip_vol_deficit} For $u \in \PSH(X,\omega)$  we have
\begin{equation}\label{eq: Bon_ineq}
\lim_{k\to\infty} \frac{1}{k^n}h^0(X,L^k \otimes \mathcal{I}(ku))=\frac{1}{n!}\int_X \omega_{\PrIv[u]}^n \geq \frac{1}{n!}\int_X \omega_u^n\,.
\end{equation}
Assume that $\int_X \omega_u^n>0$. Then equaility holds in \eqref{eq: Bon_ineq} if and only if one the following equivalent conditions hold:\\
\Rom{1} $\displaystyle\lim_{k\to\infty} \frac{1}{k^n}h^0(X,L^k \otimes \mathcal{I}(ku)) = \frac{1}{n!}\int_X \omega_u^n$. \vspace{0.15cm}\\
\Rom{2} $P[u]=\PrIv[u]$.\vspace{0.15cm}\\
\Rom{3} $[u] \in \overline{\mathcal Z}$.\vspace{0.15cm}\\
\Rom{4} $[u] \in \overline{\mathcal A}$.
\end{theorem}

It is part of showing \eqref{eq: Bon_ineq} that the limit on the left hand side exists. The equality part of \eqref{eq: Bon_ineq} can be interpreted as singular version of the Riemann--Roch theorem.
There are many known examples of potentials $u \in \PSH(X,\omega)$ for which the inequality \eqref{eq: Bon_ineq} is strict. One can even construct potentials $u$ that have zero Lelong numbers but don't have full mass, i.e., $u \not\in \mathcal E$ \cite{GZ07}. In particular, $\overline{\mathcal Z} \subsetneq \mathcal S$. What is more, by taking convex combinations of this $u$ with a  potential of $\overline{\mathcal Z}$ (and checking failure of condition (i) above), one can see that the containment $\overline{\mathcal Z} \subsetneq \mathcal S$ is nowhere $d_\mathcal S$-dense.

That the equivalences of Theorem~\ref{main_thm: arith_plurip_vol_deficit} are only proved in the presence of positive mass is perhaps not surprising, in light of \cite[Theorem~1.1, Section~4.3]{DDNL5}, where it was shown that $d_\mathcal S$ is complete \emph{only} in the presence of such condition. Still, it remains to be seen if this condition is essential in Theorem~\ref{main_thm: arith_plurip_vol_deficit}.


With different motivation, Rashkovskii studied the approximability of local isolated psh singularties using isolated analytic singularities in \cite{Ra13}. It is an interesting prospect to find the local analog of the $d_\mathcal S$ metric, and to relate our findings to the ones in \cite{Ra13}. 

As we will see, in all of our main theorems one can allow an additional twisting Hermitian line bundle $(T,h_T)$ as well (see Theorem~\ref{thm: closure_of_test_conf}, Theorem~\ref{thm: on_rays}, Theorem~\ref{thm: arith_plurip_vol_deficit} and Corollary~\ref{thm: Non-Archimedean}).

\paragraph{Organization.} In Section~\ref{sec:prelim} we recall  previous results and adapt them to our context. In Section~\ref{sec:R1app} we extend the Ross--Witt Nystr\"om correspondence to finite energy $L^1$ geodesic rays. In Section~\ref{sec:closure} we prove Theorem \ref{mainthm: on_rays} and Theorem \ref{main_thm: Non-Archimedean}. In Section~\ref{sec:cloalgsingtype} we prove Theorem \ref{main_thm: arith_plurip_vol_deficit}.

\paragraph{Acknowledgments.} The first named author has been partially supported by NSF grant DMS-1846942(CAREER) and an Alfred P. Sloan Fellowship. We profited from discussions with B. Berndtsson, M. Jonsson and L. Lempert. 
We would like to thank F. Zheng for providing the proof of Lemma~\ref{lma:convintimppwconv}. Finally, we thank the anonymous referees for a careful reading and numerous suggestions that improved the presentation.

\section{Preliminaries}\label{sec:prelim}



\subsection{The metric space of \texorpdfstring{$L^1$}{L1} geodesic rays and singularity types}\label{subsec:mecL1geod}

\paragraph{The \texorpdfstring{$L^1$}{L1} metric on \texorpdfstring{$\mathcal{H}_\omega$}{Hw} and its completion.}
We recall the basics of the $L^1$ metric structure of $\mathcal{H}_\omega$, introduced in \cite{Da15}. For a  survey we refer to \cite[Chapter~3]{Dar19}, and perhaps \cite[Section~4]{DR17} is a convenient quick summary. For historical context, we refer to \cite{Ru20}.

The $d_1$ metric on $\mathcal{H}_\omega$ is simply the path length pseudometric associated with the following $L^1$ Finsler metric:
\[
\|\psi\|_1 := \frac{1}{V}\int_X |\psi|\,\omega_u^n\,, \quad u \in \mathcal{H}_\omega \textnormal{ and } \psi \in T_u \mathcal{H}_\omega\,,
\]
where $V = \int_X \omega^n$ is the total volume. One then shows that $d_1$ is non-degenerate, making $(\mathcal{H}_\omega, d_1)$ a \emph{bona fide} metric space \cite[Theorem~1]{Da15}. 

When trying to find the $d_1$-completion of $\mathcal{H}_\omega$, one encounters the space $\mathcal{E}^1 \subseteq \PSH(X,\omega)$ that is defined in the following manner. One first defines the space of full mass potentials $\mathcal{E} \subseteq \PSH(X,\omega)$. Potentials in this space are characterized by the property $\int_X \omega_u^n = \int_X \omega^n$. Here $\omega^n_u$ is the following limit of measures
\begin{equation}\label{eq: non_pluripolar_def}
\omega_u^n := \lim_{k \to \infty} \mathds{1}_{\{u > -k\}}\, \omega_{\max(u,-k)}^n\,,
\end{equation}
where $\omega_{\max(u,-k)}^n$ can be made sense of using Bedford--Taylor theory, since $\max(u,-k)$ is bounded \cite{BT76}. For a general $\omega$-psh potential $u$ we have $\int_X \omega_u^n \in [0,\int_X \omega^n]$, with all values taken up. For more on this we refer to the original papers \cite{GZ07} and \cite{BEGZ10} (for a minimalist survey see \cite[Chapter~2]{Dar19}).  

Then, $\mathcal{E}^1 \subseteq \mathcal{E}$ is the class of full mass potentials satisfying $\int_X |u|\, \omega_u^n <\infty$. By \cite[Theorem~2]{Da15}, one can extend the metric $d_1$ to $\mathcal{E}^1$. In addition,  $(\mathcal{E}^1,d_1)$ is a complete geodesic metric space whose geodesics are decreasing limits of $C^{1,1}$-solutions to a degenerate complex Monge--Ampère equation (\cite{Ch00}, \cite{CTW18}, \cite{Da15}). Unfortunately, such limits are not the only $d_1$-geodesics connecting points of $\mathcal{E}^1$ (see the comments after \cite[Theorem~4.17]{Da15}). However, when talking about $d_1$-geodesics, we will only consider this distinguished class of length minimizing segments.

We recall that the definition of the Monge--Ampère energy $\mathrm{I}: \mathcal{E}^1 \to \mathbb{R}$ (sometimes denoted Aubin--Yau, or Aubin--Mabuchi energy):
\begin{equation}\label{eq: MA_en_def}
\mathrm{I}(u) = \frac{1}{V(n+1)} \sum_{j=0}^n\int_X u \,\omega^j \wedge \omega_u^{n-j}\,.
\end{equation}
Using the Monge--Ampère energy one can give the following potential theoretic description of $d_1$ \cite[Corollary~4.14]{Da15}:
\[
d_1(u,v) = \mathrm{I}(u) + \mathrm{I}(v) - 2 \mathrm{I}(P(u,v))\,, \quad u,v \in \mathcal{E}^1\,,
\]
where $P(u,v) \in \mathcal{E}^1$ is the following \emph{rooftop} envelope:
\[
P(u,v) = \sup \left\{\,h \in \PSH(X,\omega) : h \leq u \textup{ and } h \leq v\,\right\}\,.
\]
To understand $d_1$-convergence from a purely analytical point of view, the following double estimate is often very useful \cite[Theorem~3]{Da15}: 
\[
\frac{1}{C} d_1(u,v) \leq \int_X |u-v|\, \omega_u^n + \int_X |u-v| \,\omega_v^n \leq C d_1(u,v)\,,
\]
where $C$ is a constant only dependent on $n = \dim X$.

\paragraph{The complete metric space of $L^1$ rays.} Building on the previous paragraph, we recall the basics of the $L^1$ metric structure of $\mathcal{R}^1$, the space of $d_1$-geodesic rays in $\mathcal{E}^1$ emanating from $0 \in \mathcal{H}_\omega$, explored in detail in \cite{DL20}.

To fix notation, a $d_1$-geodesic ray $[0,\infty) \ni t \mapsto u_t \in \mathcal{E}^1$ with $u_0 =0$ will simply be denoted $\{u_t\}_t \in \mathcal{R}^1$. The chordal metric $d_1^c$ on $\mathcal{R}^1$ is introduced in the following manner:
\begin{equation}\label{eq: d_1^c_def}
d^c_1(\{u_t\}_t, \{v_t\}_t) = \lim_{t \to \infty} \frac{d_1(u_t,v_t)}{t}\,.
\end{equation}
Since $t\mapsto d_1(u_t,v_t)$ is convex \cite[Proposition~5.1]{BDL17}, the limit on the right hand side exists, and one can show that $d_1^c$ is non-degenerate, satisfies the triangle inequality, moreover $(\mathcal{R}^1,d_1^c)$ is a complete geodesic metric space \cite[Theorem~1.3, Theorem~1.4]{DL20}.

The subspace $\mathcal{R}^\infty \subseteq \mathcal{R}^1$ is the space of \emph{bounded} geodesic rays $\{u_t\}_t$, satisfying the property $u_t \in L^\infty \cap \mathcal{E}^1, \ t \geq 0$. Such rays allow for an important approximation property \cite[Theorem~1.4]{DL20} that will be used in this work, as well as its proof:  
\begin{theorem}For any $\{u_t\}_t \in \mathcal{R}^1$ there exists $\{u^j_t\}_t \in \mathcal{R}^\infty$ such that $u^j_t \searrow u_t$, for $t \geq 0$ and $d_1^c(\{u_t\}_t, \{u^j_t\}_t) \to 0$.
\end{theorem}

\paragraph{The pseudo-metric space of singularity types.}
We recall the basics of the pseudo-metric structure on $\mathcal{S}$, the space of singularity types, first explored in \cite{DDNL5}. First one needs to construct a map from $\mathcal{S}$ to $\mathcal{R}^\infty \subseteq \mathcal{R}^1$, using ideas going back to \cite{Da17}. Starting with $[u] \in \mathcal{S}$, one constructs $d_1$-geodesic segments $[0,l] \ni t \mapsto s(u)^l_t \in \mathcal{E}^1 \cap L^\infty$ connecting $s(u)^l_0 = 0$ and $s(u)^l_l = \max(u,-l)$. Moreover, using the maximum principle one can show that $\{s(u)^l_t\}_{l \geq t}$ is an $l$-increasing sequence converging to $r[u]_t \in \mathcal{E}^1 \cap L^\infty$, yielding a geodesic ray $\{r[u]_t\}_t \in \mathcal{R}^\infty$ \cite[Lemma~4.2]{Da17}. 

Using the map $[u] \mapsto \{r[u]_t\}_t$ we define the following pseudometric \cite[Section~3]{DDNL5}:
\[
d_\mathcal{S}([u],[v]) = d_1^c(\{r[u]_t\},\{r[v]_t\})\,, \quad [u],[v] \in \mathcal{S}\,.
\]
Due to non-degeneracy of $d_1^c$, one immediately sees that $d_\mathcal{S}([u],[v]) =0$ if and only if $r[u]_t = r[v]_t$ for $t \geq 0$. As shown in \cite[Theorem~3.3]{DDNL5}, in the presence of non-vanishing mass ($\int_X \omega_u^n>0$ and $\int_X \omega_v^n>0$), this condition is equivalent to $P[u] = P[v]$, where $P[\chi]$ is the envelope of the singularity type $[\chi]$, first considered in \cite{RWN14} in the Kähler context:
\[
P[\chi] = \lim_{C \to \infty} P(0,\chi + C) = \usc\left(\sup\{\psi \in \PSH(X,\omega), \psi \leq 0 \textup{ and } [\psi]=[\chi]\} \right)\,.
\]
As pointed out in \cite{DDNL5}, if $P[u]=P[v]$ holds, then the singularity types $[u]$ and $[v]$ are indistinguishable using Lelong numbers, multiplier ideal sheaves and mixed masses. 

Unfortunately the pseudomertic space $(\mathcal{S},d_\mathcal{S})$ is incomplete \cite[Section~4.2]{DDNL5}. However when restricting to the subspaces  $\mathcal{S}_\delta := \left\{\,[u] \in \mathcal{S} : \int_X \omega_u \geq \delta>0\,\right\}$, one obtains complete metric spaces $(\mathcal{S}_\delta,d_\mathcal{S})$ \cite[Theorem~4.9]{DDNL5} that are able to govern the variation of singularity types in equations of complex Monge--Ampère type \cite[Theorem~1.4]{DDNL5}.

Lastly, we mention the following double inequality that often comes handy when discussing $d_\mathcal{S}$-convergence in practice \cite[Lemma~3.4 and Proposition~3.5]{DDNL5}:
\[
d_\mathcal{S}([u],[v]) \leq  \sum_{j=0}^n \Big(2 \int_X \omega^j \wedge \omega^{n-j}_{\max(u,v)} - \int_X \omega^j \wedge \omega^{n-j}_{u} - \int_X \omega^j \wedge \omega^{n-j}_{v} \Big)\leq C d_\mathcal{S}([u],[v])\,,
\]
where $C>1$ only depends on $n$. That the expression in the middle is non-negative and only dependent on $[u]$ and $[v]$ is a consequence of \cite[Theorem~1.1]{WN19}.

\subsection{Exponents and filtrations of a family of Hermitian metrics}\label{subsec:expfilt}

In this section we relate the log-slope of the volume of a one dimensional family of Hermitian metrics with the associated filtration. In many ways we simply tailor the arguments of \cite{Bern17} to our needs, and for more thorough treatment of related results we refer to \cite[Part~1]{BE21}.

Let $V$ be a finite dimensional complex vector space of dimension $N$.
By  $\Herm(V)$ we denote the set of positive Hermitian inner products on $V$. Throughout this section, $H_s\in \Herm(V)$ ($s\geq 0$) will denote a continuous family of Hermitian inner products, simply referred to as $s \mapsto H_s$.

We denote by $V^*$ the dual vector space of $V$. Recall that given any $H\in \Herm(V)$, it naturally induces a dual inner product $H^*\in \Herm(V^*)$. 

\begin{definition}\label{def:pmfamily}
Let $I\subseteq \mathbb{R}$ be an interval.
We say that a family $H_s\in \Herm(V)$ ($s\in I$) is \emph{negative} if its trivial complexification $z \mapsto H_{\RE z}$ is a Griffiths negative vector bundle on $V$ with base $\{\RE z\in I\}$. This is equivalent to $s \mapsto \log H_s(v,v)$ being convex on $I$ for any $v \in V \setminus \{0\}$ (\cite[Section~VII.6]{De12}). Analogously, $s \mapsto H_s$ is \emph{positive} if its  dual bundle $(H^*_s)_{s\in I}$ is negative. 
\end{definition}

Let $I = [0,\infty)$. We do not assume that  $s \mapsto H_s$ is positive or negative for the moment.
The \emph{exponent} $\lambda_H:V \to [-\infty,\infty]$ of $s \mapsto H_s$ is defined by
\begin{equation}\label{eq:Lyapdef}
    \lambda_H(v):=\varlimsup_{s\to\infty} \frac{1}{s}\log H_s(v,v)\,, \quad v \in V\,.
\end{equation}
Note that $\lambda_H(0)=-\infty$. Moreover, one sees that $\lambda_H(c v) =\lambda_H( v)$ for any $c \in \mathbb{C}^*$, and $\lambda_H(u+v) \leq \max(\lambda_H(u),\lambda_H(v))$. Thus for any $c\in [-\infty,\infty]$, the set $\{\lambda_H\leq c\}\subseteq V$ is a sub- vector space. Hence $\lambda_H$ takes up only a finite number of values. If $\infty$ is not one of them, then $\lambda_H$ is the exponent of the non-Archimedean pseudometric $e^{\lambda_H}$, motivating our terminology.

The above properties of the exponent $\lambda_H$ also allow to introduce the associated \emph{filtration} of $s \mapsto H_s$:
\begin{equation}\label{eq: Filt_def}
    \mathcal{F}_{\lambda}^H:=\{v\in V: \lambda_H(v)\leq \lambda\}\,.
\end{equation}
Notice that $\mathcal{F}_{\lambda}^H$ defines an increasing right-continuous filtration of $V$ by linear subspaces. This filtration is bounded from above (in the sense that $\mathcal{F}^H_{\lambda}=V$ for some $\lambda\in \mathbb{R}$) if and only if $\lambda_H<\infty$. We call a number $\lambda\in \mathbb{R}$ a \emph{jumping number} of the filtration $\mathcal{F}^{H}$ if $\mathcal{F}^{H}_{\lambda}\neq \mathcal{F}^{H}_{\lambda-\epsilon}$ for any $\epsilon>0$.

Given $U_0,U_1 \in \Herm(V)$, one can diagonalize $U_1$ with respect to $U_0$ to find eigenvalues $e^{\lambda_1}, \ldots, e^{\lambda_N}$ counting multiplicity. Then one can introduce the following metric:
\begin{equation}\label{eq: d_1^V_def}
d_1^V(U_0,U_1) := \frac{1}{\dim V} \sum_{j=1}^N |\lambda_j|\,.
\end{equation}
This metric, along with its $L^p$-counterparts, was studied extensively in \cite{DLR20}, where it was shown that $d_1^V$ quantizes $d_1$ in the appropriate context. 

In particular, (in the appropriate diagonalizing basis) the curve $[0,1] \ni t \mapsto U_t: = \diag(e^{t \lambda_1}, \ldots, e^{t \lambda_N}) \in \Herm(V)$ provides a $d_1^V$-geodesic joining $U_0$ and $U_1$ (\cite[Theorem~1.1]{DLR20}, \cite[Theorem~3.7]{BE21}). There are other $d_1^V$-geodesics joining $U_0,U_1$, but we will only consider the above type of length minimizing curves.

We emphasize the following formula, pointing out that the dualization map $U \mapsto U^*$ between $\Herm(V)$ and $\Herm(V^*)$ is an isometry:
\begin{equation}\label{eq: dual_isom}
d_1^V(U_0,U_1) = d_1^{V^*}(U_0^*,U_1^*)\,, \quad U_0,U_1 \in \Herm(V)\,.
\end{equation}
This can be verified by picking an appropriate diagonalizing basis of $V$. 

In studying the growth of the volume of the unit ball with respect to $H_s$ as $s \to \infty$, we start with the following lemma that one can justify simply by diagonalizing:
\begin{lemma}\label{lma:intfiltr}
Suppose $s \mapsto H_s$ is a $d_1^V$-geodesic ray and let $\lambda_1\leq \ldots\leq \lambda_m$ be the jumping numbers of $\mathcal{F}_{\lambda}^H$. Then
\begin{equation}\label{eq:limforgeod}
    \lim_{s\to\infty}\frac{1}{s}\log \left(\frac{\det H_s}{\det H_0}\right)=  \sum_{j=0}^{m-1}\lambda_{j+1}(\dim \mathcal{F}_{\lambda_{j+1}}^H-\dim \mathcal{F}_{\lambda_{j}}^H)=\int_{-\infty}^{\infty} \lambda \,\mathrm{d}(\dim\mathcal{F}_{\lambda}A)\,,
\end{equation}
where the integral on the right is interpreted in the Stieltjes sense. Moreover, $\dim\mathcal{F}_{\lambda_0}^H=\dim \bigcap_{\lambda\in \mathbb{R}}\mathcal{F}_{\lambda}^H=0$ in the middle sum, by convention.
\end{lemma}

By $\det H$ we mean the determinant of a matrix representative of the sesquilinear form $H \in \Herm(V)$ with respect to a fixed basis, making ${\det H_s}/{\det H_0}$ in  \eqref{eq:limforgeod} well-defined. Note that our convention is different from that in \cite{BE21} by a square.

Using Hadamard's inequality, for $s \mapsto H_s$ only satisfying $\lambda_H < \infty$, one can show that in general the left hand side is dominated by the right hand side in \eqref{eq:limforgeod}. 

As we will see, equality holds in \eqref{eq:limforgeod} when $s \mapsto H_s$ is only positive, satisfying a mild decay condition. Before we prove this, we will construct a geodesic ray $s \mapsto \tilde H_s$ asymptotic to any $s \mapsto H_s$, closely following \cite[Proposition~2.2]{Bern17}.

\begin{lemma}\label{lma:geoderay}
Assume that $[0,\infty) \ni s \mapsto H_s$ is positive and $\lambda_{H^*}<\infty$. Then there exists a $d_1^V$-geodesic ray $s \mapsto \tilde H_s$  such that\\
\Rom{1} $H_0 = \tilde H_0$.\\
\Rom{2} $H_s \geq \tilde H_s$.\\
\Rom{3} $\lambda_H=\lambda_{\tilde H}$. In particular, $\lambda_H < \infty$.\\
\Rom{4} $\displaystyle \lim_{s\to \infty}\frac{1}{s}\log \bigg( \frac{\det H_s}{\det H_0} \bigg)=\lim_{s\to \infty}\frac{1}{s}\log \bigg( \frac{\det \tilde H_s}{\det\tilde H_0} \bigg)$.
\end{lemma}
Recall the following comparison principle that will be used multiple times in the argument below: if $[a,b] \ni s \mapsto U_s,W_s \in \Herm(V)$ are such that $s \mapsto U_s$ is positive and $s \mapsto W_s$ is a geodesic with $W_a \leq U_a$ and $W_b \leq U_b$ then  $W_s \leq U_s$ for $s \in [a,b]$ \cite[Lemma~8.11]{BK12}.
Note that $s \mapsto \log \det W_s$ is linear  and also $ \log \det W_s \leq \log \det U_s $. Varying the endpoints $a,b$ we obtain that $s \mapsto \log \det H_s$ is concave, whenever $s \mapsto H_s$ is positive. As a result, the limit on the left of  \Rom{4}  exists.

\begin{proof} First we interpret the condition $\lambda_{H^*} < \infty$. Since $s \mapsto H_s^*$ is negative,  $s \mapsto \log H^*_s(v,v)$ is convex for any $v \in V^*\setminus \{0\}$, hence   $H_s^*(v,v) \leq e^{s \lambda_{H^*}(v)} H_0^*(v,v)$, for $s \geq 0$. Dualizing, we arrive at
\begin{equation}\label{eq: H_s_lower bound}
H_s \geq e^{-s \lambda_{H^*}} H_0\,, \quad s \geq 0\,.
\end{equation} 
Now we construct $s \mapsto \tilde H_s$. For each $t\geq 0$, we define $[0,\infty) \ni s \mapsto H^t_s\in \Herm(A)$ as follows: for $[0,t] \ni s \mapsto H^t_s$ is the geodesic connecting $H_0$ and $H_t$ and $H^t_s = H_s$ for $s>t$.

By the comparison principle, we get that $H_s^t$ is $t$-decreasing for any $s\geq 0$ (in fact $s \mapsto H_s^t$ is positive for any $t$, but this will not be needed). Due to \eqref{eq: H_s_lower bound} we can take the decreasing $t$-limit to obtain 
\[
\tilde H_s(v,v) := \lim_{t\to\infty}H_s^t(v,v)\,, \quad v \in V\,.
\]
It is immediate that $s \mapsto H_s$ is a $d_1^V$-geodesic ray satisfying  \Rom{1}  and  \Rom{2} .

Recall that $s \mapsto \log \det H_s$ is concave (due to positivity) and of course $s \mapsto \log \det \tilde H_s$ is linear (since $s \mapsto \tilde H_s$ is a geodesic). Using this, due to the construction of $s \mapsto \tilde H_s$, one immediately sees that $\displaystyle\lim_{s \to \infty} s^{-1}(\log \det H_s - \log \det \tilde H_s) =0$, proving  \Rom{4}. 

Since $H_s \geq \tilde H_s$, comparing with \eqref{eq: d_1^V_def} we arrive at
\[
\lim_{s \to \infty} \frac{d_1^V(\tilde H_s, H_s)}{s} = \lim_{s \to \infty} \frac{\log \det H_s - \log \det \tilde H_s}{s}  =0\,.
\]
Because of this, by Lemma~\ref{lma:comparenorms} below, for any $\epsilon >0$ there exists $s_0$ such that 
$e^{-\epsilon s}\tilde H_s \leq H_s \leq e^{\epsilon s}\tilde H_s$, for $s \geq s_0$. This is immediately seen to imply  \Rom{3} .  
\end{proof}
\begin{lemma}\label{lma:comparenorms}
Let $U_1,U_2 \in \Herm(V)$. Assume that $d_1^V(U_1,U_2)\leq \epsilon$ for some $\epsilon>0$. Then
\[
e^{-\epsilon \dim V}U_2 \leq U_1 \leq e^{\epsilon \dim V}U_2\,.
\]
\end{lemma}
\begin{proof}
We fix a basis $(e_1,\ldots,e_{\dim V})$ that is orthonormal with respect to $U_1$ and orthogonal with respect to $U_2$, with eigenvalues $e^{\lambda_1},\ldots, e^{\lambda_{{\dim V}}}$. 
Then by definition, 
$d_1(U_1,U_2)= \frac{1}{\dim V}\sum_{j=1}^{{\dim V}} |\lambda_j|$.
Hence, $|\lambda_j|\leq \epsilon \dim V$, so $e^{-\epsilon \dim V}U_2\leq U_1 \leq e^{\epsilon \dim V}U_2$.
\end{proof}

\begin{theorem}\label{thm: detequality}
Assume that  $[0,\infty) \ni s \mapsto H_s$ is positive  with $\lambda_{H^*}<\infty$, and $\lambda_1\leq\ldots\leq\lambda_m$ are the jumping numbers of the filtration $\mathcal{F}_\lambda^H$. Then
\begin{equation}\label{eq:limsupleqintfil1}
    \lim_{s\to\infty}\frac{1}{s}\log \left(\frac{\det H_s}{\det H_0}\right)=\sum_{j=0}^{m-1}\lambda_{j+1}(\dim \mathcal{F}_{\lambda_{j+1}}^H-\dim \mathcal{F}_{\lambda_{j}}^H)= \int_{-\infty}^{\infty} \lambda \,\mathrm{d}\left(\dim\mathcal{F}_{\lambda}^H\right)\,,
\end{equation}
where $\dim \mathcal F^H_{\lambda_0} =0$ by convention.
\end{theorem}
\begin{proof}
As discussed below the statement of Lemma~\ref{lma:geoderay}, the limit on the left hand side of  \eqref{eq:limsupleqintfil1} exists and is finite. In fact, for the 
ray $s \mapsto \tilde H_s$ constructed in Lemma~\ref{lma:geoderay} we have that 
\[
\lim_{s\to\infty}\frac{1}{s}\log\bigg(\frac{\det H_s}{\det H_0}\bigg)=\lim_{s\to \infty} \frac{1}{s}\log\bigg(\frac{\det \tilde H_s}{\det \tilde H_0}\bigg)\,.
\]
Since $\lambda_H = \lambda_{\tilde H}$ implies $\mathcal{F}^H_\lambda = \mathcal{F}^{\tilde H}_\lambda$, the conclusion follows from Lemma~\ref{lma:intfiltr}.
\end{proof}

\subsection{Quantization of the Monge--Ampère energy}\label{subsec:quantization}

Recall that we have a positive Hermitian line bundle $(L,h)$ inducing a background Kähler metric $\omega = \frac{\mathrm{i}}{2\pi} \Theta(h)>0$ with class $[\omega]\in c_1(L)$. For any $k\geq 1$, the metric $h$ induces a Hermitian metric $h^k$ on $L^k$.

For the rest of the paper we also fix a holomorphic (twisting) line bundle $T$ on $X$ together with a  smooth Hermitian metric $h_T$. By slight abuse of notation, we also denote the induced metric on $T\otimes L^k$ by $h^k$.

Given $\varphi \in \PSH(X,\omega)$, by $\HTLm \subseteq \HTL$, we denote the space of holomorphic sections of $L^k$ over $X$ that are $L^2$-integrable with respect to the weight $e^{-k\varphi}$, i.e.,   $\int_X h^k(s,s)e^{-k\varphi}\,\omega^n  <\infty$. Also $h^0(X,T\otimes L^k \otimes \mathcal{I}(k\varphi))= \dim \HTLm$.

For each $k\geq 1$, define the Hilbert map $\Hilb_k:\mathcal{E}^1 \rightarrow \Herm(\HTL)$ as follows:
\begin{equation}\label{eq:Hilbkdef}
\Hilb_k(\varphi)(f,g):=\int_X h^k(f,g)e^{-k\varphi}\,\omega^n \,,\quad \varphi\in \mathcal{E}^1 \textnormal{ and } f,g\in \HTL\,.
\end{equation}

Define the quantum Monge--Ampère energy $\mathrm{I}_k:\Herm(\HTL)\rightarrow \mathbb{R}$ by the formula
\begin{equation}\label{eq:defIkqu}
\mathrm{I}_k(U):=-\frac{1}{kV}\log \left(\frac{\det U}{\det \Hilb_k(0)}\right)\,.
\end{equation}
The expression $\mathrm{I}_k(U)-\mathrm{I}_k(V)$ is nothing but Donaldson's original $\mathcal L$-functional from \cite{Don05}. As we will see, the $\mathrm{I}_k$ quantizes the usual Monge--Ampère energy $\mathrm{I}$, motivating our notation.

Now we define $\Lk:\mathcal{E}^1\rightarrow \mathbb{R}$ by
\begin{equation}\label{eq:defDonLk}
\Lk(\varphi):=\mathrm{I}_k \circ \Hilb_k(\varphi)\,.
\end{equation}
\begin{remark}
When $(T,h_T)$ is trivial and $\varphi$ is equal to $P(\phi)$ for some continuous function $\phi$ on $X$, the functional $\Lk(\varphi)$ is defined and studied in \cite{BB10}. Note that our $\Lk(\varphi)$ corresponds to $h^0(X,L^k)\Lk(X,\phi/2)$ in their paper, with the extra $1/2$  due to the difference in conventions.
\end{remark}

Let $V_k := \HTL$. As recalled in Section~\ref{subsec:expfilt}, there is a natural metric $d_1^{V_k}$ on $\Herm(V_k)$. With the focus of this section on quantization, we define a scaled version of this metric: 
\[
\dok(H_1,H_2):=\frac{1}{k}d_1^{V_k}(H_1,H_2)\,, \quad H_1,H_2 \in V_k
\,.
\]
This convention coincides with the one used in \cite{DLR20}.

Let $S = \{0< \textup{Re }z < 1\} \subset \mathbb C$ be the unit strip and $\pi:S \times X \rightarrow X$ be the natural projection. We say that $(0,1) \ni t \mapsto \varphi_t \in \mathcal{E}^1$ is a \emph{subgeodesic} if its complexification $\Phi$ satisfies $
\pi^*\omega+\ddc \Phi\geq 0$ on $S \times X$ in the sense of currents.
Let us recall the following version of Berndtsson's convexity theorem \cite[Theorem~1.2]{Bern09}.
\begin{theorem}\label{thm:Bernconvex}
Let $[0,1] \ni t \mapsto \varphi_t \in \mathcal{E}^1$ be a subgeodesic connecting $\varphi_0,\varphi_1 \in \mathcal{E}^1$. Let $\Phi$ be the complexification of $t \mapsto \varphi_t$ and assume that
\[
\pi^*\omega+\ddc \Phi\geq \epsilon\pi^*\omega
\]
for some $\epsilon>0$. Let $[0,1] \ni t \mapsto H_t\in \Herm(\HTL)$  be the geodesic connecting $\Hilb_k(\varphi_0)$ with $\Hilb_k(\varphi_1)$. Then there exists $k_0(\varepsilon)>0$, so that for any $k\geq k_0$ we have 
\[
H_t\leq \Hilb_k(\varphi_t)\,,\quad t\in [0,1]\,.
\]
Moreover, $t \mapsto \Hilb_k(\varphi_t)$ is positive (see Definition~\ref{def:pmfamily}).

If additionally $t \mapsto \varphi_t$ is $t$-increasing, then for any $s\in \HTL$,
\begin{equation}\label{eq:derHtlower}
-k\int_X \dot{\varphi}_1 h^k(s,s)e^{-k\varphi_1}\,\omega^n\leq \ddtz[1]H_t(s,s)\,.
\end{equation}
\end{theorem}
The proof follows line by line from that of \cite[Corollary~2.13, Lemma~2.14]{DLR20}.

\begin{lemma}\label{lma:Iklip}
For $\varphi_0,\varphi_1\in \mathcal{E}^1$ we have
\[
\left| \Lk(\varphi_0)-\Lk(\varphi_1) \right|\leq C k^{n} \dok({\Hilb_k(\varphi_0)},{\Hilb_k(\varphi_1)})\,,
\]
where $C$ depends only on $X$. 
\end{lemma}
\begin{proof}
Notice that
\[
\Lk(\varphi_1)-\Lk(\varphi_0)=-\frac{1}{kV}\log \frac{\det {\Hilb_k(\varphi_1)}}{\det {\Hilb_k(\varphi_0)}}\,.
\]
Take a basis $(e_1,\ldots,e_{N_k})$ of $H^0(X,T\otimes L^k)$ which is orthonormal with respect to ${\Hilb_k(\varphi_0)}$ and is orthogonal with respect to ${\Hilb_k(\varphi_1)}$. Let $\lambda_j:=\log \Hilb_k(\varphi_1)(e_j,e_j)$ for $j=1,\ldots,N_k$. Then 
\[
\Lk(\varphi_1)-\Lk(\varphi_0)=-\frac{1}{kV}\sum_{j=1}^{N_k} \lambda_j\leq \frac{N_k}{V} \dok\left({\Hilb_k(\varphi_0)},{\Hilb_k(\varphi_1)}\right)\,.
\]
By the Riemann--Roch theorem, $N_k$ is dominated by $V k^n$, and the result follows.
\end{proof}
\begin{lemma}\label{lma:d1quant}
For $\varphi_0,\varphi_1\in \mathcal{E}^1$ we have
\[
\lim_{k\to\infty} \dok({\Hilb_k(\varphi_0)},{\Hilb_k(\varphi_1)})=d_1(\varphi_0,\varphi_1)\,.
\]
\end{lemma}
This result is the twisted version of \cite[Theorem~1.2\Rom{2}]{DLR20}. We reproduce the proof for convenience of the reader.
\begin{proof}
Without loss of generality, let us assume that $\varphi_0,\varphi_1\leq -1$. Due to \cite[Theorem~3.3]{Bern13a} the results is known for $\varphi_0,\varphi_1 \in \mathcal{H}_\omega$. 

By \cite{BK07} we can find $\varphi_0^j$, $\varphi_1^j\in \mathcal{H}_{\omega}$, sequences decreasing to $\varphi_0,\varphi_1$,  respectively. We may assume without loss of generality that $\varphi_0^1,\varphi_1^1\leq 0$.  By our assumption, for any $j\geq 1$ we have,
\[
\lim_{k\to\infty} \dok({\Hilb_k(\varphi_0^j)},{\Hilb_k(\varphi_1^j)})=d_1(\varphi_0^j,\varphi_1^j)\,.
\]
Hence it is enough to show that for any $\epsilon>0$, we can find $j_0>0$ such that
\[
\varlimsup_{k\to\infty}\dok({\Hilb_k(\varphi_0^j)},{\Hilb_k(\varphi_0)})<\epsilon\,,\quad \varlimsup_{k\to\infty} \dok({\Hilb_k(\varphi_1^j)},{\Hilb_k(\varphi_1)})<\epsilon\,,
\]
for any $j\geq j_0$. By symmetry, we only prove the former. We fix some real number $\delta>1$ for now.
Let $[0,1] \ni t \mapsto \ell_t \in \mathcal{E}^1$ be the geodesic from $P(\delta \varphi_0)$ to $\varphi_0^j$. For $t\in [0,1]$, let 
\[
\ell'_t:=\frac{1}{\delta}\ell_t+\left(1-\frac{1}{\delta}\right)\varphi_0^{j}\,.
\]
Notice that $\ell'_0\leq \varphi_0\leq \varphi_0^j=\ell'_1$. As a result,  ${\Hilb_k(\ell'_1)}\leq {\Hilb_k(\varphi_0)}\leq {\Hilb_k(\ell'_0)}$.
Hence, by comparison of the tangent vectors at $\ell'_1$, we conclude
\begin{equation}\label{eq: d_1^k_est}
\dok({\Hilb_k(\varphi_0)},{\Hilb_k(\varphi_0^j)})\leq \dok({\Hilb_k(\ell'_0)},{\Hilb_k(\ell'_1)})\,.
\end{equation}
Let $t \to G^k_t\in \Herm(\HTL)$ ($t\in [0,1]$) be the geodesic from ${\Hilb_k(\ell'_0)}$ to ${\Hilb_k(\ell'_1)}$.
Observe that the conditions of Theorem~\ref{thm:Bernconvex} are satisfied by $t \to \ell'_t$. Hence for $k \geq k_0(\delta)$,
\[
G^k_t\leq {\Hilb_k(\ell'_t)}
\]
for all $t\in [0,1]$. By \eqref{eq:derHtlower}, for any $f\in \HTL$, we have
\begin{equation}\label{eq: interm_est}
-\frac{1}{\delta}\int_X\dot{\ell}_1 h^k(f,f)e^{-k\varphi_0^j}\,\omega^n\leq \frac{1}{k}\ddtz[1] G^k_t(f,f)\leq 0\,.
\end{equation}
By \cite[Lemma~4.5]{DLR20}, the left hand side is finite. Now we find a basis $e_1,\ldots,e_N$ of $\HTL$ that is orthonormal with respect to $\Hilb_k(\ell'_1)$ and such that the quadratic form
\[
f\mapsto -\frac{1}{\delta}\int_X\dot{\ell}_1 h^k(f,f)e^{-k\varphi_0^j}\,\omega^n
\]
is orthogonal with eigenvalues $\lambda_1,\ldots,\lambda_N$. Then, using \eqref{eq: d_1^k_est} and \eqref{eq: interm_est}, we get
\[
\dok({\Hilb_k(\varphi_0)},{\Hilb_k(\varphi_0^j)}) \leq \dok({\Hilb_k(\ell'_0)},{\Hilb_k(\ell'_1)})\leq \frac{1}{N}\sum_{a=1}^N |\lambda_a|\leq \frac{1}{N}\sum_{a=1}^N\int_X |\dot{\ell}_1|h^k(e_a,e_a)e^{-k\varphi_0^j}\,\omega^n\,.
\]
Letting $k\to\infty$, it follows from the classical Bergman kernel expansion  that (see the elementary calculations following \cite[(35)]{DLR20})
\[
\varlimsup_{k\to\infty}\dok({\Hilb_k(\varphi_0)},{\Hilb_k(\varphi_0^j)})\leq \frac{1}{V\delta}\int_X |\dot{\ell}_1|\,\omega_{\varphi_0^j}^n=\frac{1}{\delta}d_1(P(\delta \varphi_0),\varphi_0^j)\,,
\]
where the last equality follows from \cite[Lemma~4.5]{DLR20}. Letting $\delta \searrow 1$, we find
\[
\varlimsup_{k\to\infty}\dok({\Hilb_k(\varphi_0)},{\Hilb_k(\varphi_0^j)})\leq d_1(P(\varphi_0,\varphi_0^j)\,,
\]
finishing the proof of the claim, and the argument.
\end{proof}

Next we quantize the Monge--Amp\`ere energy (see \eqref{eq: MA_en_def})  on the space $\mathcal E^1$, extending the corresponding result for smooth metrics \cite{Don05}, continuous metrics  \cite[Theorem~A]{BB10}, and the case of $K_X$-twisting \cite[Theorem~3.5]{BFiM14}:
\begin{theorem}\label{thm:QuantI}
For any $\varphi\in \mathcal{E}^1$, we have
\begin{equation}\label{eq:limLkIE1}
\lim_{k\to\infty}\frac{n!}{k^{n}}\Lk(\varphi)=\mathrm{I}(\varphi)\,.
\end{equation}
\end{theorem}
\begin{proof}
Assume that this result is true for $\varphi\in \mathcal{H}_{\omega}$. For a general $\varphi\in \mathcal{E}^1$, take a decreasing sequence $\varphi_j\in \mathcal{H}_{\omega}$ that converges to $\varphi$. Then by Lemma~\ref{lma:Iklip} and Lemma~\ref{lma:d1quant}, for any $j\geq 1$,
\[
\varlimsup_{k\to\infty} \left|\frac{n!}{k^{n}}\Lk(\varphi_j)-\frac{n!}{k^{n}}\Lk(\varphi)\right|\leq C \varlimsup_{k\to\infty} d_1^k(\varphi_j,\varphi)=C d_1(\varphi_j,\varphi)\,.
\]
By our assumption, $\lim_{k\to\infty}\frac{n!}{k^{n}}\Lk(\varphi_j)=\mathrm{I}(\varphi_j)$. This implies that
\[
\varlimsup_{k\to\infty} \left|\mathrm{I}(\varphi_j)-\frac{n!}{k^{n}}\Lk(\varphi)\right|\leq C d_1(\varphi_j,\varphi)\,.
\]
Letting $j\to\infty$, we conclude. 

It remains to prove \eqref{eq:limLkIE1} when $\varphi \in \mathcal{H}_\omega$. When $(T,h_T)$ is trivial, this was carried out in \cite{Don05}. Indeed, it suffices to observe that
\[
\frac{n!}{k^{n}}\cdot \ddt \Lk(t\varphi)=\int_X \varphi B_k(t\varphi)\,\omega^n\,,\quad t\in [0,1]\,,
\]
where $B_k(t\varphi)$ denotes the $k$-th $T$-twisted Bergman kernel at $t\varphi\in \PSH(X,\omega)$. The well-known Bergman kernel expansion \cite[Theorem~4.1.1]{MM07} implies that  $B_k(t\varphi)\,\omega^n \to V^{-1}\omega_{t\varphi}^n$ uniformly. Consequently, 
\[
\frac{n!}{k^{n}}\cdot\ddt \Lk(t\varphi) \to \ddt \mathrm{I}(t\varphi)
\]
uniformly. Taking integral with respect to $t\in [0,1]$, we conclude \eqref{eq:limLkIE1} in this case.
\end{proof}

\subsection{An algebraic notion of singularity type}\label{subsec:algsingtype}

\paragraph{Detecting singularities using algebraic tools.} 

In this section, let $(X,\omega)$ be a compact K\"ahler manifold of dimension $n$. Let $\theta$ be a smooth real $(1,1)$-form on $X$ representing a pseudo-effective cohomology class.

Given $u \in \PSH(X,\theta)$, as pointed out in the literature (see for example \cite{BFJ08}, \cite{Kim15}), one can not characterize the singularity type $[u]$ using ``mainstream'' algebraic data, like multiplier ideal sheaves $\mathcal{I}(c u), c >0$ or Lelong numbers. Instead, one can introduce an algebraic notion that is coarser than equivalence up to singularity types, considered in \cite[Section~2.1]{KS20}:
\begin{definition}
Let $\varphi,\psi\in \PSH(X,\theta)$. We put $\varphi \preceq_{\veq} \psi$ in case $\mathcal{I}(a\varphi) \subseteq \mathcal{I}(a\psi)$ for all $a > 0$. Then $\preceq_{\veq}$ is a preorder, with equivalence relation $\varphi\simeq_{\veq} \psi$ characterized by $\mathcal{I}(a\varphi)=\mathcal{I}(a\psi)$ for all $a > 0$. The corresponding classes are called $\veq$-\emph{singularity types}, and are denoted by $[\chi]_{\veq}$, where $\chi \in \PSH(X,\theta)$ is a representative of the class.
\end{definition}

In the language of \cite[Section~2.1]{KS20} the relation $\simeq_{\veq}$ is called \emph{v-equivalence}. Obviously $[\varphi]=[\psi]$ implies $[\varphi]_{\veq}  =[\psi]_{\veq}$. However the reverse direction does not hold in general, and this phenomenon is at the center of the discussion in this subsection. Before we dive deeper into this, let us recall the following characterization of $\veq$-equivalence via Lelong numbers, a direct consequence of \cite[Theorem~A]{BFJ08}:

\begin{theorem}\label{thm: Lelong_characterization}
Let $\varphi,\psi\in \PSH(X,\theta)$. Then the following are equivalent: \vspace{0.15cm}\\
\Rom{1} $\varphi\simeq_{\veq} \psi$.\vspace{0.15cm}\\
\Rom{2} $\nu(\varphi ,y)=\nu(\psi ,y)$ for any projective modification $\pi:Y\rightarrow X$, with $Y$ smooth, and $y\in Y$.
\end{theorem}
In the above statement $\nu(\varphi , y)$ is the Lelong number of $\varphi \circ \pi$ at $y$, in local coordinates defined by 
\[
\nu(\varphi , y)=\nu^{\pi}(\varphi , y):= \sup \left\{\,c \geq 0 : \varphi \circ \pi(z) \leq c \log \|z - y\| + O(1) \textup{ near } y\,\right\}\,.
\]

Given a prime divisor $Z$ of $Y$, the \emph{generic Lelong number} of $\varphi$ along $Z$ is defined as:
\[
\nu(\varphi , Z) = \inf_{z \in Z}\nu(\varphi , z)\,.
\]

Due to Siu's semicontinuity theorem, for a set $S \subseteq Z$ of measure zero,  we have that $\nu(\varphi , z) = \nu(\varphi , Z)$ for $z \in Z \setminus S$, motivating the terminology.

Since we work with smooth models $Y$, for a coherent ideal $\mathcal{J} \subseteq \mathcal{O}_X$ one can talk about $\nu(\mathcal{J}, y)$ ($\nu(\mathcal{J}, Z)$) as the minimum vanishing order of $f_j \circ \pi$ at $y$ (along $Z$) for a finite set of generators $\{f_j\}_j$ of $\mathcal{J}_y$ ($\mathcal{J}_z$ for some  $z \in Z$). Moreover, one can see that $\nu(\mathcal{J}, y) := \nu(\mathcal{J}, E_y)$, where  $E_y$ is the exceptional divisor of $p_y: \Bl_{\{y\}}Y \to Y$, the blowing up of $Y$ at $y$.

The following result is implicit in \cite{BFJ08}, and clarifies the relationship between multiplier ideal sheaves and Lelong numbers in Theorem~\ref{thm: Lelong_characterization}. We give a detailed sketch of the argument for the convenience of the reader:

\begin{prop}\label{rmk:mistoLelong}
Let $\pi:Y\rightarrow X$ be a projective modification with $Y$ smooth, and $y\in Y$. For $\varphi\in \PSH(X,\theta)$ we have
\begin{equation}\label{eq: Lelong_L^2_formula}
\nu(\varphi ,y)=\lim_{k\to\infty}\frac{1}{k}\nu(\mathcal{I}(k\varphi),y)\,.
\end{equation} 
\end{prop}
\begin{proof}
That $\nu(\varphi ,y)\geq \frac{1}{k}\nu(\mathcal{I}(k\varphi),y)$ follows from the fact that the local potential with singularity governed by $\frac{1}{k} \mathcal{I}(k \varphi)$ is always less singular than $\varphi$ (by the Ohsawa--Takegoshi theorem). Taking $\varlimsup$, we get that the left hand side is greater than $\nu(\varphi ,y) \geq \varlimsup_{k\to\infty}\frac{1}{k}\nu(\mathcal{I}(k\varphi))$.

For the reverse inequality, we start with noticing that $\nu(\varphi  \circ p_y, z) \geq \nu (\varphi ,y)$ for any $z \in E_y$. Indeed, Lelong numbers can only increase under pullbacks. In particular, $\nu(\varphi,{E_y}) \geq \nu(\varphi , y)$. That $\varliminf_{k\to\infty} \frac{1}{k} \nu(\mathcal{I}(k \varphi),{E_y}) \geq \nu(\varphi,{E_y})$, follows from an application of Fubini--Tonelli's theorem, as elaborated in the proof of \cite[(5.3)]{BFJ08}.
\end{proof}

\begin{remark} That \Rom{1} implies \Rom{2} in Theorem~\ref{thm: Lelong_characterization} is seen to follow from \eqref{eq: Lelong_L^2_formula}. The reverse direction now follows from the local result \cite[Theorem~A]{BFJ08}. More broadly, the reverse direction is the consequence of the valuative criteria for integrability (see \cite[Theorem~10.12]{Bo17} and its proof).
\end{remark}

\begin{coro}\label{cor:charrelmislelong}
Let $\varphi,\psi\in \PSH(X,\theta)$. Then the following are equivalent: \vspace{0.15cm}\\
\Rom{1} $\varphi\preceq_{\veq} \psi$. \vspace{0.15cm}\\
\Rom{2} $\nu(\varphi,y)\geq\nu(\psi,y)$ for any projective modification $\pi:Y\rightarrow X$, with $Y$ smooth, and $y\in Y$.
\end{coro}
\begin{proof}
Assume that \Rom{1} holds. Then \Rom{2} holds by Proposition~\ref{rmk:mistoLelong}. Conversely, assume that \Rom{2} holds. Then $\max\{\varphi,\psi\}\simeq_{\veq} \psi$ by Theorem~\ref{thm: Lelong_characterization}, and the  fact  that $\nu(\max(\varphi,\psi),y) = \min(\nu(\psi,y),\nu(\varphi,y))$ \cite[Corollary~2.10]{Bo17}. Hence
$\varphi\preceq_{\veq}\max(\varphi,\psi)\simeq_{\veq} \psi$ as desired.
\end{proof}

As we saw in the above argument, the class of potentials $\chi$ satisfying $\chi \preceq_{\veq} \varphi$ are stable under taking $\max$, hence we can introduce the notion of an envelope with respect to $\veq$-singularity:
\begin{equation}\label{eq: PI_def}
\begin{split}
\PrIv[\varphi]:=&\usc \left(\sup \left\{ \,\psi\in \PSH(X,\theta):  \psi\leq 0\,, \psi\preceq_{\veq} \varphi
\,\right\} \right)\\
=&\usc \left(\sup \left\{\,\max\{\psi, \varphi - \sup_X \varphi\}: \psi\in \PSH(X,\theta)\,, \psi\leq 0\,, \psi\preceq_{\veq} \varphi \,\right\} \right)\\
=&\usc \left(\sup \left\{\,\psi\in \PSH(X,\theta): \psi\leq 0\,, \psi\simeq_{\veq} \varphi
\,\right\} \right)\,.
\end{split}
\end{equation}
The above envelope should be compared with the well-known envelope with respect to singularity type (going back to \cite{RWN14} and \cite{RS05} in the local case):
\[
P[\varphi] := \usc\left(\sup \left\{\,v \in \PSH(X,\theta) : [v]=[\varphi] \textup{ and } v \leq 0 \,\right\}\right)\,.
\]
We refer the reader to \cite{RWN14}, \cite{DDNL2}, \cite{DDNL5} for basic properties of $P[u]$. 


\begin{definition}Recall that $u \in \PSH(X,\theta)$ is a \emph{model potential} if $u = P[u]$. Also, the singularity $[u]$ of a model potential $u$ is called a \emph{model singularity type}. Analogously, a potential $\varphi\in \PSH(X,\theta)$ is called \emph{$\veq$-model} if $\varphi=\PrIv[\varphi]$. The singularity type $[\varphi]$ of an $\veq$-model potential $\varphi$ is called an \emph{$\veq$-model singularity type}. 
\end{definition}

We begin to discuss the parallel between the above notions:

\begin{prop}\label{prop:Qprojection}
 Let $\varphi\in \PSH(X,\theta)$. Then \vspace{0.15cm}\\
\Rom{1} $\PrIv[\varphi]\in \PSH(X,\theta)$ is a model potential ($P[\PrIv[\varphi]]=\PrIv[\varphi]$), moreover $\PrIv[\varphi]\geq P[\varphi]$.
In particular, all $\mathcal{I}$-model potentials are also model potentials.
\vspace{0.15cm}\\
\Rom{2} $\varphi\simeq_{\veq} \PrIv[\varphi]$. In particular, $P[\PrIv[\varphi]]_\mathcal I=\PrIv[\varphi]$, and the $\usc$ is unnecessary in \eqref{eq: PI_def}.
\end{prop}

According to the above result $\PrIv[u] = u$ implies $P[u] = u$. As a result, if $u$ is $\veq$-model then it is automatically model, but not vice versa.

\begin{proof}
\Rom{1} $\PrIv[\varphi]=P[\PrIv[\varphi]]$ because $a\PrIv[\varphi]$ and $aP[\PrIv[u]]= a\lim_{C \to \infty}P(\PrIv[\varphi]+C,0)$ have that same multiplier ideal sheaves for any $a\geq 0$. Indeed, multiplier ideal sheaves are stable under taking increasing limits \cite{GZ15}.
Since $\varphi\simeq_{\veq} P[\varphi]$, we get $\PrIv[\varphi]\geq P[\varphi]$.

\Rom{2} By Choquet's lemma we can take $\psi_j\in \PSH(X,\theta)$ $(j\geq 0)$, such that $\psi_j\leq 0$, $\psi_j\sim_{\veq} \varphi$ and that $\psi_j$ increases to $\PrIv[\varphi]$ a.e.. It follows from Guan--Zhou's strong openness theorem \cite{GZ15} that $\varphi\sim_{\veq} \PrIv[\varphi]$.
\end{proof}

\begin{example}
Following \cite[Example~6.10]{BBJ21}, we give an example showing that not all model potentials are $\mathcal{I}$-model. Consider $X=\mathbb{P}^1$ and let $\omega$ be the Fubini--Study form on $X$. Let $K\subseteq \mathbb{P}^1$ be a Cantor set. Then $K$ carries an atom-free probability measure, whose potential $v$ has zero Lelong numbers. Then the pull-back of $v$ to any proper modification of $X$ has zero Lelong numbers as well \cite[Corollary~2.4]{Kim15}. Hence $\PrIv[v]=0$.
But $v$ does not have full mass. Hence $P[v]\neq 0$, i.e., $P[v]$ is model but not $\mathcal I$-model.
\end{example}

\begin{prop}\label{prop: anal_sing_type_envelope} Assume that $\psi \in \PSH(X,\theta) \in \mathcal A$ (See Definition~\ref{def: anal_alg_sing}). Then
 \[
 \PrIv[\psi]=P[\psi]\,.
 \]
 In particular, $P[\psi]$ is $\veq$-model, and has the same singularity type as $\psi$.
\end{prop}

\begin{proof}
First one notices that $[\PrIv[\psi]]=[\psi]$ for $[\psi]$ analytic. This is a consequence of \cite[Theorem~4.3]{Kim15}. It can also be seen after an analysis on the pullback $\pi:Y \to X$, where $\pi$ is the  normalized blowup of the ideal of $\psi$, precomposed with a log resolution. 

Since $[\PrIv[\psi]]=[\psi]$, we get $\PrIv[\psi] \leq P[\psi]$, with the reverse being true by Proposition~\ref{prop:Qprojection}\Rom{1}. 
\end{proof}

\begin{lemma}\label{lma:decvmodel}
Suppose that $\{\varphi_j\}_j\in \PSH(X,\theta)$ and $\varphi \in \PSH(X,\theta)$ are model potentials. \vspace{0.15cm}\\
\Rom{1} If $\varphi_j \searrow \varphi$ and $\varphi_j$ are $\veq$-model, then $\varphi$ is $\veq$-model as well.\vspace{0.15cm}\\
\Rom{2} If $\varphi_j \searrow \varphi$ and $\int_X \theta_{\varphi}^n > 0$, then $\PrIv[\varphi_j] \searrow \PrIv[\varphi]$.\vspace{0.15cm}\\
\Rom{3} If $\varphi_j \nearrow \varphi$ a.e. and $\int_X \theta_{\varphi}^n > 0$, then $\PrIv[\varphi_j] \nearrow \PrIv[\varphi]$ a.e. as well. In particular, if  the $\varphi_j$ are additionally $\veq$-model, then $\varphi$ is  $\veq$-model as well.
\end{lemma}
\begin{proof} 
First we prove \Rom{1}. Note that $\PrIv[\varphi]\simeq_{\veq} \varphi \preceq_{\veq} \varphi_j$ for any $j\geq 1$. Hence by Proposition~\ref{prop:Qprojection}, $\PrIv[\varphi]\leq \PrIv[\varphi_j]=\varphi_j$.
Letting $j\to\infty$, we obtain $\PrIv[\varphi]\leq \varphi$. 
Since $\varphi\leq \varphi_j\leq 0$, we know that $\varphi\leq \PrIv[\varphi]$, hence $\varphi$ is $\veq$-model.

We deal with \Rom{2}. Since $\int_X \theta_{\varphi_j}^n \searrow \int_X \theta^n_\varphi>0$ \cite[Proposition 4.8]{DDNL5}, by \cite[Lemma 4.3]{DDNL5}  there exists $\alpha_j \searrow 0$ and $v_j := P(\frac{1}{\alpha_j}\varphi + (1 - \frac{1}{\alpha_j}) \varphi_j) \in \PSH(X,\theta)$ satisfying $(1-\alpha_j)\varphi_j + \alpha_j v_j \leq \varphi$. Taking $\PrIv[\cdot]$ we arrive at 
\[
(1-\alpha_j)\PrIv[\varphi_j] + \alpha_j \PrIv[v_j] \leq \PrIv[(1-\alpha_j)\varphi_j + \alpha_j v_j] \leq \PrIv[\varphi]\,, 
\]
where in the first inequality we have used $\PrIv[\psi] \sim_\veq \psi$, Theorem~\ref{thm: Lelong_characterization}, and additivity of Lelong numbers. 
Since $\{\varphi_j\}_j$ is decreasing, so is $\{\PrIv[\varphi_j]\}_j$, hence $w:= \lim_j \PrIv[\varphi_j] \geq \PrIv[\varphi]$ exists. Since $\alpha_j \to 0$ and $\sup_X \PrIv[v_j]=0$, comparison with the above  gives $w = \PrIv[\varphi]$.

Dealing with \Rom{3} is similar. Since $\int_X \theta_{\varphi_j}^n \nearrow \int_X \theta^n_\varphi>0$ \cite[Theorem 2.3]{DDNL2}, by \cite[Lemma 4.3]{DDNL5} there exists $\alpha_j \searrow 0$ and $v_j := P(\frac{1}{\alpha_j}\varphi_j + (1 - \frac{1}{\alpha_j}) \varphi) \in \PSH(X,\theta)$ satisfying $(1-\alpha_j)\varphi + \alpha_j v_j \leq \varphi_j$. Taking $\PrIv[\cdot]$ we arrive at 
$$(1-\alpha_j)\PrIv[\varphi] + \alpha_j \PrIv[v_j]  \leq  \PrIv[(1-\alpha_j) \varphi + \alpha_j v_j] \leq \PrIv[\varphi_j]\, ,$$
where in the first inequality we have used that $\PrIv[\psi] \sim_\veq \psi$, Theorem~\ref{thm: Lelong_characterization}, and additivity of Lelong numbers. 
Since $\{\varphi_j\}_j$ is increasing, so is $\{\PrIv[\varphi_j]\}_j$, hence $w:= \lim_j \PrIv[\varphi_j] \leq \PrIv[\varphi]$ exists. Since $\alpha_j \to 0$ and $\sup_X \PrIv[v_j]=0$, comparison with the above  yields $w = \PrIv[\varphi]$.
\end{proof}

\begin{remark}\label{rmk:naprecsim}
The condition $\varphi\simeq_{\veq} \psi$ is strictly stronger than requiring $\varphi$ and $\psi$ have the same Lelong number everywhere on $X$ (See \cite[Example~2.5]{Kim15}).
As we will see in the next section, in terms of valuations, $\varphi \preceq_{\veq} \psi$ means exactly that the induced non-Archimedean functions on the space of divisorial valuations $X^{\mathrm{div}}_{\mathbb{Q}}$ satisfy $\varphi^{\NA} \leq \psi^{\NA}$.
In particular, $\varphi^{\NA} = \psi^{\NA}$ is equivalent to $\varphi\simeq_{\veq} \psi$.
See \cite{BFJ08} for further details.
\end{remark}

\paragraph{Algebraic approximation of $\veq$-model potentials.} 
For the remained of this subsection, we return to the context of an ample line bundle $L \to X$, with hermitian metric $h$, whose first Chern form is equal to the K\"ahler form $\omega$. Let us recall the following well known result, originated from \cite[Theorem~2.2.1]{DPS01}:

\begin{theorem}\label{thm:qeqsingapp}
Let $u\in \PSH(X,\omega)$. Let $u_k\in \PSH(X,\omega)$ be the partial Bergman kernel of $V^u_k:=H^0(X,L^{2^k + k_0} \otimes \mathcal{I}(2^k u))$: 
\begin{equation}\label{eq: quasi_eq_sing_def}u_k = \frac{1}{2^{k} + k_0} \sup_{\substack{s \in V^u_k\,,\\ \Hilb_{2^k}(u)(s,s) \leq 1}} \log h^{2^{k} + k_0}(s,s)\,,
\end{equation}
where $\Hilb_{2^k}(u)$ is the Hilbert map of $L^{2^k}$ with twisting $T = L^{k_0}$ (see Section~\ref{subsec:quantization}). Then $u_k$ has algebraic singularity type (See Definition~\ref{def: anal_alg_sing}), and for some $k_0= k_0(X,L,\omega)$ the following hold:\\
\Rom{1} $u_k$ converges to $u$ in $L^1$ as $k\to\infty$.\\
\Rom{2} $[u_{k+1}] \preceq [u_k]$.\\
\Rom{3} for all $m > 0$ and $k > m$ one can find $\delta_{k,m}>1$ such that  $ \mathcal{I}(m \delta_{k,m} u_k)\subseteq \mathcal{I}(mu)$ and $\delta_{k,m} \searrow 1$ as $k \to \infty$.
\end{theorem}
\begin{proof}[Sketch of proof] \Rom{1} follows from Step 1 in the proof of \cite[Theorem~7.1]{GZ05}. By Step 2 in the proof of \cite[Theorem~7.1]{GZ05} we get that for sufficiently high $k_0$ the sequence $[u_k]$ is decreasing, satisfying \Rom{2}. Condition~\Rom{3} is a consequence of the \emph{comparison of integrals method} of \cite[Theorem~2.2.1]{DPS01} (see \cite[Lemma~3.2]{Cao14} where this is written out explicitly).
\end{proof}
Following terminology of \cite{Cao14}, approximations $\{u_j\}_j$ of $u\in \PSH(X,\omega)$ of the type \eqref{eq: quasi_eq_sing_def}, satisfying all three conditions in Theorem~\ref{thm:qeqsingapp} will be referred to as \emph{quasi-equisingular approximations} of $u$. 

We arrive at the following result, characterizing the difference between model and $\veq$-model potentials in terms of $d_\mathcal S$-approximability via quasi-equisingular sequences:
\begin{theorem}\label{thm:anaappvmodel}
Let $\varphi \in \PSH(X,\omega)$ be a model potential $(P[\varphi]=\varphi)$ with $\int_X \omega_{\varphi}^n > 0$. Then $\varphi$ is $\veq$-model ($\PrIv[\varphi]=\varphi$) if and only if $[\varphi]$ is the $d_\mathcal{S}$-limit of a quasi-equisingular approximation $[\varphi_j] \in \mathcal Z$.
\end{theorem}

\begin{remark}\label{rem: anal_sing_d_S_approx} Due to this theorem and Proposition~\ref{prop: anal_sing_type_envelope}, the class of analytic singularity types $\mathcal A$ are $d_\mathcal S$-approximable by algebraic singularity types of $\mathcal Z$ (in the presence of positive mass), already proving the (easy) equivalences between  \Rom{4}  and  \Rom{5}  in Theorem~\ref{main_thm: arith_plurip_vol_deficit}.
\end{remark}

\begin{proof} 
Let $\varphi \in \PSH(X,\omega)$ be an $\veq$-model potential with $\int_X \omega_\varphi^n >0$. Let $\varphi_k$ be the corresponding quasi-equisingular approximation of $\varphi$. By \cite[Theorem~1.1]{DDNL2} the sum $\sum_{j=0}^n \int_X \omega^{n-j}\wedge \omega_{\varphi_k}^j$ is decreasing in $k$, hence converges. By \cite[Lemma~3.4]{DDNL5}, $\{[\varphi_k]\}_k \subseteq \mathcal{S}$ forms a $d_\mathcal{S}$-Cauchy sequence with $\int_X \omega_{\varphi_j}^n \geq \int_X \omega_{\varphi}^n>0$. Hence by \cite[Theorem~1]{DDNL5}, the sequence  $d_\mathcal S$-converges to $[\varphi'] \in \mathcal{S}$, where $\varphi' = P[\varphi]'=\lim_j P[\varphi_j] \geq P[\varphi]=\varphi$ \cite[Corollary~4.7]{DDNL5}. 
We claim that 
\begin{equation}\label{eq: varphi_eq_varphi'}
\varphi=\varphi'\,.
\end{equation}
By Lemma~\ref{prop: anal_sing_type_envelope} and Lemma~\ref{lma:decvmodel}\Rom{1}, both $\varphi$ and $\varphi'$ are $\veq$-model, so it suffices to show that  $\mathcal{I}(m\varphi)=\mathcal{I}(m\varphi')$
for any $m > 0$. Since $\varphi \leq \varphi'$, the non-trivial inclusion is $\mathcal{I}(m\varphi)\supseteq \mathcal{I}(m\varphi')$. 
To prove this, by the last statement of the above theorem, we notice that
\[
\mathcal{I}(m\varphi)\supseteq \mathcal{I}(m \delta_{k,m} \varphi_k)=\mathcal{I}(m \delta_{k,m} P[\varphi_k]) \supseteq \mathcal{I}(m \delta_{k,m} \varphi')\,.
\]
By the strong openness theorem \cite{GZ15}, we can let $k \to \infty$ to arrive at $\mathcal{I}(m\varphi)\supseteq \mathcal{I}(m\varphi')$, as desired. 

Conversely, let $\varphi_j$ be the quasi-equisingular approximation of $\varphi$ such that $d_\mathcal{S}([\varphi_j],[\varphi]) \to 0$. Taking envelopes, by Lemma~\ref{prop: anal_sing_type_envelope} we conclude that $\varphi_j' := P[\varphi_j] = \PrIv[\varphi_j]$ is pointwise decreasing, and $d_\mathcal{S}([\varphi_j],[\varphi]) = d_\mathcal{S}([\varphi'_j],[\varphi]) \to 0$. \cite[Lemma~3.6]{DDNL5} now gives that $\int_X \omega_{\varphi'}^n \searrow \int_X \omega_{\varphi}^n$. Since $\varphi = P[\varphi]$, $\lim_j \varphi'_j \geq \varphi$, and $\int_X \omega_\varphi^n >0$, by \cite[Theorem~3.12]{DDNL2} we obtain that $\lim_j \varphi'_j = \varphi$. 
Finally, Lemma~\ref{lma:decvmodel}\Rom{1} implies that $\varphi$ is $\veq$-model. 
\end{proof}

Before we proceed further, we recall that the conventions set at the beginning of Section~\ref{sec:prelim} guarantee that the leading order Riemann--Roch expansion takes the form $\hTL= \frac{V}{n!}k^n+\mathcal{O}(k^{n-1})$, where $T$ is an arbitrary line bundle. We recall the following result of Bonavero:
\begin{theorem}\label{thm: analytic_sing_type_formula}
Assume that $\varphi\in \PSH(X,\omega)$ has algebraic singularity type ($[\varphi] \in \mathcal Z$). Then
\[
\lim_{k\to\infty} \frac{\hTLm}{k^n}=\frac{1}{n!}\int_X \omega_{\varphi}^n\,.
\]
\end{theorem}
This is indeed a special case of the singular holomorphic Morse inequalities proved by Bonavero \cite{Bon98}, surveyed in \cite[Theorem~2.3.18]{MM07}.
\begin{proof}
According to \cite[Théorème~1.1]{Bon98} for $q=0$, we have $k^{-n}\hTLm\leq \frac{1}{n!}\int_X \omega_{\varphi}^n+o(1)$, hence
\[
\varlimsup_{k\to\infty}\frac{\hTLm}{k^n} \leq \frac{1}{n!}\int_X \omega_{\varphi}^n\,.
\]

Applying \cite[Théorème~1.1]{Bon98}  with $q=1$, we get
\[
-\frac{\hTLm}{k^n}+\frac{h^1(X,T\otimes L^k\otimes \mathcal{I}(k\varphi))}{k^n}\leq -\frac{1}{n!}\int_X \omega_{\varphi}^n+o(1)\,.
\]
But according to \cite[Corollaire~2.2]{Bon98}, $h^1(X,T\otimes L^k\otimes \mathcal{I}(k\varphi))=o(k^n)$,  hence
\[
\varliminf_{k\to\infty} \frac{\hTLm}{k^n} \geq \frac{1}{n!}\int_X \omega_{\varphi}^n\,.
\]
Hence equality indeed holds.
\end{proof}
\begin{remark}
Note that our convention for the multiplier ideal sheaves is different from that of Bonavero's. In fact, Bonavero's definition of $\mathcal{I}(\varphi/2)$ corresponds to our $\mathcal{I}(\varphi)$. But the volume of $\varphi/2$ in the sense of Bonavero is exactly the same as $\int_X \omega_{\varphi}^n$ in our sense, hence the holomorphic Morse inequalities take exactly the same form, despite the difference in conventions. 

We additionally note that Bonavero proved the above result for potentials with \emph{analytic} singularity type, however his definition of this notion is less general than ours in Definition~\ref{def: anal_alg_sing}, this being the reason for our more conservative statement above. 
\end{remark}

\begin{theorem}\label{thm:genBon}
Let $\varphi\in \PSH(X,\omega)$ be an $\veq$-model potential.
Then
\[
\varlimsup_{k\to\infty}\frac{1}{k^n}\hTLm\leq \frac{1}{n!}\int_X \omega_{\varphi}^n\,.
\]
\end{theorem}
\begin{proof}
If $\int_X \omega_{\varphi}^n>0$,
the estimate follows directly from Theorem~\ref{thm: analytic_sing_type_formula}, Theorem~\ref{thm:anaappvmodel} and \cite[Lemma~3.6]{DDNL5}.

Now let
$\int_X \omega_{\varphi}^n=0$ and $\varphi_j$ be a quasi-equisingular approximation of $\varphi$. Let $\epsilon\in (0,1)\cap \mathbb{Q}$. Let $P_{\epsilon}[\cdot]$ denote the envelope with respect to singularity type with respect to  $\omega+\epsilon\omega$. Note that the potentials $P_{\epsilon}[\varphi_j]\in \PSH(X,\omega+\epsilon\omega)$ have positive masses bounded away from zero, for each $\epsilon >0$ fixed. Moreover, $P_{\epsilon}[\varphi_j]$ has the same singularity type as $\varphi_j$. Let $\psi^{\epsilon}$ be the decreasing limit of $P_{\epsilon}[\varphi_j]$. By Lemma~\ref{lma:decvmodel}\Rom{1}, $\psi^{\epsilon}$ is an $\veq$-model potential and 
\[
\int_X (\omega+\epsilon\omega+\ddc\psi^{\epsilon})^n=\lim_{j\to\infty}\int_X (\omega+\epsilon\omega+\ddc \varphi_j)^n\,.
\]
We have  $\psi_\epsilon \searrow \psi\in \PSH(X,\omega)$ as $\epsilon \searrow 0$. From the condition $\mathcal{I}(m \delta_{k,m}\varphi_k) \subseteq \mathcal{I}(m \varphi)$, we see that $\mathcal{I}(m\psi_{\epsilon}) = \lim_k \mathcal{I}(m \delta_{k,m} \psi_\epsilon) \subseteq \lim_k \mathcal{I}(m \delta_{k,m} \varphi_k) \subseteq \mathcal{I}(m\varphi)$, for any $m\geq 0$. Hence it follows that $\mathcal{I}(m\psi)\subseteq \mathcal{I}(m\varphi)$. Hence $\PrIv[\psi]\leq \varphi$. But as $\psi\leq 0$, we know that $\PrIv[\psi]\geq \psi\geq \varphi$, hence we conclude that $\varphi=\psi$.

Now we claim that 
\begin{equation}\label{eq:claimmass0}
\lim_{\epsilon\to 0+} \int_X (\omega+\epsilon\omega+\ddc\psi^{\epsilon})^n=0\,.
\end{equation}
Indeed, for $c>0$ let $\psi^\epsilon_c : = \max (\psi^\epsilon,-c)$ and $\varphi_c : = \max(\varphi,-c)$. For any $b >0$ We have that
\[\varlimsup_{\epsilon\to 0+}\int_X (\omega+\epsilon\omega+\ddc\psi^{\epsilon})^n \leq \varlimsup_{\epsilon\to 0+}\int_X e^{b \psi^\epsilon}(\omega+\epsilon\omega+\ddc\psi^{\epsilon}_c)^n = \int_X e^{b \varphi}(\omega+\ddc\varphi_c)^n,\]
where in the first estimate we used that $(\omega+\epsilon\omega+\ddc\psi^{\epsilon})^n$ is supported on $\{\psi^{\epsilon} =0\}$ \cite[Theorem~3.8]{DDNL2}, and in the last equality we used that $e^{b\psi^\epsilon}$ is bounded and quasi-continuous, converging to $e^{b\varphi}$ in capacity. Similarly, $(\omega+\ddc\varphi)^n$ is supported on $\{\varphi =0\}$, hence letting $b \to \infty$ we arrive at our claim 
\[
\varlimsup_{\epsilon\to 0+}\int_X (\omega+\epsilon\omega+\ddc\psi^{\epsilon})^n \leq \int_X \omega_\varphi^n=0\,.
\]

Let $\delta >0$ be arbitrary, and take $\epsilon=p/q \in \mathbb Q_+$ such that $\frac{1}{n!}\int_X (\omega+\epsilon\omega+\ddc\psi^{\epsilon})^n<\delta$. By the positive mass case of this theorem,
\begin{flalign}\label{eq: h^0_est}
\nonumber \varlimsup_{k\to\infty, q|k}\frac{1}{k^n}\hTLm&\leq \varlimsup_{k\to\infty, q|k}\frac{1}{k^n}h^0(X,T\otimes L^k\otimes L^{k\epsilon}\otimes \mathcal{I}(k\psi^{\epsilon}))\\
&\leq \frac{1}{n!}\int_X (\omega+\epsilon\omega+\ddc\psi^{\epsilon})^n<\delta\,.
\end{flalign}
For a general $k$ (possibly not divisible by $q$) write $k= d q+r$ with $d\in \mathbb{Z}_{\geq 0}$, $r=0,\ldots,q-1$. Then
\[
\frac{1}{k^n}\hTLm\leq \frac{1}{(dq)^n} h^0(T\otimes L^r\otimes L^{dq}\otimes \mathcal{I}(dq\varphi))
\]
for $q$ large enough.
Thus, replacing $T$ with $T\otimes L^r$ as the twisting line bundle, we are reduced to the case $r=0$, dealt with in \eqref{eq: h^0_est}. Letting $\delta \to 0$, the proof is finished.
\end{proof}

\subsection{Filtrations, flag ideals and the non-Archimedean formalism}\label{subsec:nonAchi}

\paragraph{Filtrations of the ring of sections.} Let us recall the basics of filtrations in the context of canonical Kähler metrics, going back to work of Székelyhidi \cite{Sze15}. We refer to \cite[Section~1, Section~5]{BHJ17} and \cite[Section~3]{BJ18} for a much more detailed description. In the sequel, we will focus on the point of view advocated by Ross--Witt Nyström \cite{RWN14}. For $r > 0$ we will consider  
\[R(X,L^r):=\bigoplus_{k=0}^\infty H^0(X,L^{kr})\] 
the graded ring associated to $(X,L^r)$. When dealing with filtrations, we always assume that $r$ is big enough so that  $R(X,L^r)$ is generated in degree 1.

A \emph{filtration} $\big(\{\mathcal{F}^{\lambda}_k\}_{\lambda\in \mathbb{R}, k\in \mathbb{N}},r\big)$ of $R(X,L^r)$ is a collection of decreasing left-continuous filtrations $\{\mathcal{F}_k^{\lambda}\}_{\lambda}$ on each vector space $H^0(X,L^{kr})$ , that is \emph{multiplicative} ($\mathcal{F}^{\lambda}_k \cdot \mathcal{F}^{\lambda'}_{k'} \subseteq \mathcal{F}^{\lambda + \lambda'}_{k + k'}$ for any $k,k'\in \mathbb{N}$, $\lambda,\lambda'\in \mathbb{R}$) and \emph{linearly bounded} (there exists $\tilde \lambda >0$ big enough such that $\mathcal{F}^{-\tilde \lambda k}_k = H^0(X,L^{kr})$ and $\mathcal{F}^{\tilde \lambda k}_k = \{0\}$, for $k\geq 0$.)

Let $(\mathbb{C},\|\cdot\|)$ be the trivially normed complex line. A \emph{non-Archimedean graded norm} on $R(X,L^r)$ is a norm on $R(X,L^r)$ considered as a $(\mathbb{C},\|\cdot\|)$-algebra satisfying exponential boundedness (there exists $\tilde \lambda>0$, such that for any $k\in \mathbb{N}$ and any non-zero $s\in H^0(X,L^{kr})$, $e^{-\tilde \lambda k} \leq \| s \|_k \leq e^{\tilde \lambda k})$ and sub-multiplicativity ($\| s \cdot s'\|_{k + k'} \leq \| s\|_k \|s'\|_{k'}$ for any $s \in H^0(X,L^{kr})$ and $s' \in H^0(X,L^{k'r})$).

It is elementary to verify that there is a bijection between filtrations $\big(\{\mathcal{F}^{\lambda}_k\},r\big)$ and non-Archimedean graded norms $\{\|\cdot\|_k\}_{k\in \mathbb{N}}$ on $R(X,L^r)$ given by
 \[
 \| s\|_k \leq e^{-\lambda} \Leftrightarrow s \in \mathcal{F}^\lambda_k\,,\quad k\in \mathbb{N}\,, \lambda \in \mathbb{R}\,, s\in H^0(X,L^{kr})\,.
 \]
Due to this, we will use the terms filtrations and non-Archimedean norms interchangeably. 

\paragraph{Filtrations induced by test configurations and flag ideals.} A filtration $\big(\{\mathcal{F}^\lambda_k\}, r \big)$ is a $\mathbb Z$-\emph{filtration} if the jumping numbers/points of discontinuity  of $ \lambda \mapsto \mathcal{F}^\lambda_k$ are integers for all $k \geq 0$. 

Due to the fact that $R(X,L^r)$ is generated in degree $1$, we have a surjective map 
\begin{equation}\label{eq: fin_gen_filt}
\left(H^0(X,L^r )\right)^{\otimes k} \to H^0(X,L^{kr})\,.
\end{equation} 

Naturally, $\| \cdot\|_1$ induces a non-Archimedean norm on $(H^0(X,L^r))^{\otimes k}$, as well as on any quotient $(H^0(X,L^r))^{\otimes k}/W$, where $W \subseteq (H^0(X,L^r))^{\otimes k}$ is  a subspace. 

As a result, given a filtration $\big(\{\mathcal{F}^\lambda_k\}, r \big)$, it is possible
 to define a non-Archimedean graded norm ${\| \cdot \|^T_k}$ on each $H^0(X,L^{kr})$ only using $\| \cdot \|_1$ and the maps \eqref{eq: fin_gen_filt}. 

We say that $\big(\{\mathcal{F}^\lambda_k\}, r \big)$ is \emph{induced by an (ample) test configuration} if it is a $\mathbb Z$-filtration, and the map \eqref{eq: fin_gen_filt} induces  an isometry between the graded non-Archimedean norms $\| \cdot\|^T_k$ and $\| \cdot\|_k $ for any $k \geq 0$.

This of course is not the usual definition of (ample) test configurations. However, as pointed out in \cite[Proposition~2.15]{BHJ17}, this construction is in a one-to-one correspondence with the usual one going back to Tian \cite{Ti97} and Donaldson \cite{Don01}. 

Flag ideals yield an important (and in many ways exhaustive) class of filtrations induced by test configurations, going back to Odaka \cite{Od13}. A \emph{flag ideal}  $\mathfrak a$  is a $\mathbb{C}^*$-invariant coherent ideal of $\mathcal{O}_{X \times \mathbb{C}}$, cosupported in $X \times \{0\}$. Such an ideal is always of the form
\[
\mathfrak a = \sum_{j=0}^{d-1} \tau ^j \mathfrak a_j + \tau^d \mathcal{O}_X\,,
\]
where $\mathfrak{a}_j$ is an increasing sequence of coherent ideals of $\mathcal{O}_X$ and $\tau$ is the variable in $\mathbb{C}$. As a convention, we write $\mathfrak a_j = \mathcal{O}_X$, when $j \geq d$ and $\mathfrak a_j = 0$, when $j  <  0$.

If for some $r>0$, the sheaves $L^r \otimes \mathfrak a_i$ are globally generated for every $i \geq 0$, then we associate a filtration to $\mathfrak a$ in the following way. First we define $\mathcal{F}^\lambda_0 := H^0(X,L^r \otimes \mathfrak a_{\lfloor -\lambda \rfloor})$. As $\{\mathcal{F}^\lambda_0\}_\lambda$ is decreasing and left-continuous, it induces a non-Archimedean norm $\| \cdot \|_1$ on $H^0(X,L^r)$, which further introduces a non-Archimedean graded norm $\| \cdot \|$ on $R(X,L^r)$ via the surjections \eqref{eq: fin_gen_filt}.

By construction, the underlying $\mathbb Z$-filtration is clearly induced by a test configuration, with the jumping numbers of $\{\mathcal{F}^\lambda_0\}_\lambda$ being exactly the integers $j$ such that $\mathfrak a_{j} \subsetneq \mathfrak a_{j+1}$. As pointed out by Odaka \cite{Od13}, essentially all test configurations arise via this construction.

\paragraph{The non-Archimedean formalism.} Here we recall some of the formalism developed in  \cite{BHJ17, BHJ19, BBJ21}, and later tailor some of their results to our context.

By $X^{\Div}_\mathbb{Q}$ we denote the set of rational divisorial valuations on $X$, i.e., valuations $v: \mathbb{C}(X) \to \mathbb{Q}$ of the form $v = c \ord_D$, with $D$ being a prime divisor on some smooth variety $Y$, mapping to $X$ via a projective modification, and $c \in \mathbb{Q}_+$. By convention, we also take the trivial valuation $v_\triv$ to be part of $X^{\Div}_\mathbb{Q}$.

To any $v \in X^{\Div}_\mathbb{Q}$ one associates $\sigma(v) \in (X \times \mathbb{C})^{\Div}_\mathbb{Q}$, the \emph{Gauss extension} of $v$. The construction is described in detail in \cite[Section~4.1]{BHJ17}. 

The Gauss extension is defined  as $\sigma(v)(\sum_j f_j \tau^j) := \min_j(v(f_j)+j)$, where $f_j \in \mathbb{C}(X)$ and $\tau$ is the coordinate of $\mathbb{C}$. 
It can be immediately verified that $\sigma(v)$ thus defined is a valuation, moreover as shown in \cite[Lemma~4.5, Theorem~4.6]{BHJ17},
$\sigma(v)$ is also divisorial on $X \times \mathbb C$.

\begin{remark} In \cite[Section~3]{BBJ21} the authors define $X^{\Div}_\mathbb{Q}$ as divisorial valuations on normal models of $X$. However, due to Hironaka's theorem, one can always take log-resolutions of normal models, hence no information is lost if one considers only prime divisors of smooth models, as we do in this work.
\end{remark}

\paragraph{The non-Archimedean data of potentials, rays and flag ideals.} In the non-Archime\-dean approach to canonical Kähler metrics one converts both analytic and algebraic data into \emph{non-Archimedean data}, i.e., various functions on $X^{\Div}_\mathbb{Q}$. We describe how this is carried out with $\omega$-psh functions, geodesic rays and flag ideals. 

Given $u \in \PSH(X,\omega)$, one defines $u^{\NA}: X^{\Div}_\mathbb{Q} \to \mathbb{R}$ by 
\begin{equation}\label{eq:uNAdef}
u^{\NA}(V):= -c\nu(u, D)\,, \quad V= c \ord_D \in X^{\Div}_\mathbb{Q}\,.
\end{equation}
Recall that  $\nu(u, D)$ is the generic Lelong number of $u$ along $D$ (see Section~\ref{subsec:algsingtype}). In accordance with the literature, sometimes we will also write $V(u): = c\nu(u, D)$.

Before proceeding, we note the following result, which corresponds to Theorem~\ref{thm: Lelong_characterization} in the non-Archimedean dictionary. Indeed, for any projective modification $\pi': Y \to X$ with $Y$ smooth, the Lelong number for any $u \in \PSH(X,\omega)$ at $y \in Y$ is the same as the generic Lelong number along the exceptional divisor of the blowup $\Bl_{\{y\}}Y$.
\begin{prop} \label{prop: I-equiv_NA}Suppose that $u,w \in \PSH(X,\omega)$. Then the following are equivalent:\vspace{0.15cm}\\
\Rom{1} $u \simeq_{\veq} w$.\vspace{0.15cm}\\
\Rom{2} $u^{\NA} = w^{\NA}$.
\end{prop}

Given a ray $\{r_t\}_t \in \mathcal{R}^1$, it is known that $t \mapsto \sup_X r_t$ is linear, in particular, there exists $l \in \mathbb{R}$ such that $\Phi(s,x) = r_{-\log|s|} +  l \log|s| \in \PSH(X \times \mathbb D, \pi^*\omega)$, where $\mathbb D$ is the unit disk and $\pi: X \times \mathbb D \to X$ is the usual projection.

We define $r^{\NA}: X^{\Div}_\mathbb{Q} \to \mathbb{R}$ using the Gauss extension in the following manner:
\begin{equation}\label{eq: r^NA_def}
r^{\NA}(v) := - \sigma(v)(\Phi) + l\,,
\end{equation}
where $\sigma(v) \in (X \times \mathbb{C})^{\Div}_\mathbb{Q}$ is the Gauss extension of $v \in X^{\Div}_\mathbb{Q}$ and $\sigma(v)(\Phi)$ is to be interpreted as a suitable multiple of the generic Lelong number along the center divisor $E'$ of $\sigma(v)$. It can be seen that this definition does not depend on $l \in \mathbb{R}$, nor on the choice of smooth model hosting $E'$. 

Lastly, by the same construction $u^\NA$ can be defined for sublinear subgeodesic rays $\{u_t\}_t$ (as defined in Section~\ref{subsec:rwncorre} below).\medskip

Given a \emph{flag ideal} $\mathfrak a = \sum_{j=0}^{d-1} \tau ^j \mathfrak a_j + \tau^d \mathcal{O}_X$, such that $L^m \otimes \mathfrak a_j$ are globally generated for all $j$, we define the corresponding function non-Archimedean function $\varphi^{\NA}_\mathfrak a : X^{\Div}_\mathbb{Q} \to \mathbb{R}$ as follows:
\[
\varphi^{\NA}_\mathfrak a(v) := -\min_j(v(\mathfrak a_j) + j)\,,
\]
where $v \in X^{\Div}_\mathbb{Q}$, and $v(\mathfrak a_j)$ is the valuation of  $\mathfrak a_j$ given by $v(\mathfrak a_j) := \inf \{\,v(a): a \in \mathfrak a_j\,\}$.

\paragraph{Approximation by flag ideals.} Given  a ray $\{r_t\}_t \in \mathcal{R}^1$, in \cite{BBJ21} the authors define an approximation scheme by flag ideals $\mathfrak a^m$ such that  $\varphi^{\NA}_{\mathfrak a^m} \searrow r^{\NA}$. We describe the main point of this procedure, as it will be important in the sequel.

For simplicity, let us assume that $\{r_t\}_t \in \mathcal{R}^1$ satisfies $\sup_X r_t \leq 0$ for $t \geq 0$. The general case, can easily be reduced to this case, but one needs to slightly extend the definition of flag ideals (to allow for fractional ideals). We have $\Phi(s,x) = r_{-\log|s|} \in \PSH(X \times \mathbb D, \pi^*\omega)$, and we simply define \cite[Section~5.3]{BBJ21}: 
\[
\mathfrak{a}^m := \mathcal{I}(2^m \Phi)\,.
\]
As pointed out in \cite[Lemma~5.6]{BBJ21} $L^{2^m + m_0} \otimes \mathfrak{a}^m_j$ is globally generated for some $m_0>0$ and all $m,j$. In addition, by the proof of \cite[Theorem~6.2]{BBJ21}, the subbaditivity of multiplier ideals implies that $\varphi_{\mathfrak{a}^m}^{\NA}$ is $m$-decreasing, moreover $\varphi_{\mathfrak{a}^m}^{\NA}(v) \searrow r^{\NA}(v)$ for $v \in X^{\Div}_\mathbb{Q}$.  

\section{The structure of \texorpdfstring{$\mathcal{R}^1$}{R1} and approximability}\label{sec:R1app}

\subsection{The extended Ross--Witt Nyström correspondence}\label{subsec:rwncorre}

The results of this subsection hold for an arbitrary Kähler manifold $(X,\omega)$. The goal is to give a duality between the finite energy geodesic rays of $\mathcal{R}^1_\omega$ and certain maximal test curves, reminiscent of \cite{RWN14} and \cite{DDNL3}, but to also give a formula the Monge--Ampère slope of $L^1$ rays in terms of their Legendre transforms. To do this we consider a wider context and generalize the discussion  going back to \cite{RWN14},  revisited in \cite{DDNL3}.

A \emph{sublinear subgeodesic ray} is a subgeodesic ray $(0,+\infty) \ni t \mapsto u_t \in \PSH(X,\omega)$ (notation $\{u_t\}_{t >0}$) such that $u_t \to_{L^1} u_0: = 0$ as $t \to 0$, and there exists $C\in \mathbb{R}$ such that $u_t(x) \leq C t$ for all $t\geq 0$, $x \in X$.

Due to $t$-convexity, we obtain some immediate properties of sublinear subgeodesic rays:
\begin{lemma}\label{lem: -inf_est_subgeod} Suppose that $\{u_t\}_t$ is a sublinear subgeodesic ray. Then the set $\{u_t >-\infty \}$ is the same for any $t >0$. In particular, for any $x \in X$ the curve $t \mapsto u_t(x)$ is either finite and convex on $(0,\infty)$, or equal to $-\infty$ on this interval.
\end{lemma}

A \emph{psh geodesic ray} is a sublinear subgeodesic ray that additionally satisfies the following maximality property: for any $0 < a < b$, the subgeodesic $(0,1) \ni t \mapsto v^{a,b}_t:=u_{a(1-t) + bt} \in \PSH(X,\omega)$ can be recovered in the following manner:
\begin{equation}\label{eq: vabt_eq}
v^{a,b}_t:=\sup_{h \in \mathcal{S}}{h_t}\,,\quad t \in [0,1]\,,
\end{equation}
where $\mathcal{S}$ is the set of subgeodesics $(0, 1) \ni  t \to h_t \in \PSH(X,\omega)$ such that
\[
\lim_{t \searrow 0}h_t\leq  u_a\,,\quad \lim_{t \nearrow 1}h_t\leq u_b\,.
\]

We note the following properties of the map $v \mapsto \sup_X v$ along rays:
\begin{lemma}\label{lem: suplinear} For any psh geodesic ray $\{u_t\}_t$, the map $t \mapsto \sup_X u_t$ is linear. For sublinear subgeodesics, the map $t \mapsto \sup_X u_t$ is only convex.
\end{lemma}

The statement for subgeodesics is a consequence of $t$-convexity.  To argue the statement for psh geodesic rays, one can simply use \cite[Theorem~1]{Da17} together with approximation by bounded geodesics, and the continuity of $u \mapsto \sup_X u$ in the weak $L^1$-topology of $\PSH(X,\omega)$.

Making small tweaks to \cite[Definition 5.1]{RWN14}, we are ready to give the definition of test curves:

\begin{definition}
A map $\mathbb{R}\ni \tau \mapsto \psi_\tau\in \PSH(X, \omega)$ is a \emph{psh test curve}, denoted $\{\psi_\tau\}_{\tau \in \mathbb R}$, if \vspace{0.15cm}\\
\Rom{1} $\tau\mapsto \psi_\tau(x)$ is concave, decreasing and usc for any $x\in X$. \vspace{0.15cm}\\
\Rom{2} $\psi_\tau\equiv -\infty$ for all $\tau$ big enough, and $\psi_\tau$ increases a.e. to $0$ as $\tau \to -\infty$.
\end{definition}
Note that this definition is more general than the one in \cite{RWN14} (where the authors only considered potentials with small unbounded locus), even more general than the one in \cite{DDNL3} (where the authors considered only bounded test  curves). Moreover, condition~\Rom{2} allows for the introduction of the following constant:
\begin{equation}\label{eq: psi+_tau_def}
\tau_\psi^+ := \inf\{\tau \in \mathbb{R} : \psi_\tau \equiv -\infty\}\,.
\end{equation}

\begin{remark} We adopt the following notational convention: psh test curves will always be parametrized by $\tau$, whereas rays will be parametrized by $t$. Hence $\{\psi_t\}_t$ will always refer to some kind of ray, whereas  $\{\phi_\tau\}_\tau$ will refer to some type of test curve. As we prove below, rays and test curves are dual to each other, so one should think of the  parameters $t$ and $\tau$ to be dual to each other as well.
\end{remark}

\begin{definition}\label{def: test_curves}
\leavevmode A  psh test curve $\{\psi_\tau \}_\tau$  can have the following properties: \vspace{0.1cm}\\
\Rom{1} $\{\psi_{\tau}\}_\tau$ is \emph{maximal} if $P[\psi_\tau] =\psi_\tau$ for any $\tau \in \mathbb{R}$.\vspace{0.1cm}\\
\Rom{2} $\{\psi_{\tau}\}_\tau$ is \emph{$\veq$-maximal} if $\PrIv[\psi_\tau] =\psi_\tau$ for any $\tau \in \mathbb{R}$.\vspace{0.1cm}\\
\Rom{3} $\{\psi_{\tau}\}_\tau$ is a \emph{finite energy test curve} if
    \begin{equation}\label{eq: fetestcurve_def}
        \int_{-\infty}^{\tau^+_\psi} \left( \int_X \omega_{\psi_\tau}^n-\int_X \omega^n \right) \,\mathrm{d}\tau >-\infty\,.
    \end{equation}
\Rom{4} We say $(\psi_{\tau})$ is \emph{bounded} if $\psi_\tau = 0$ for all $\tau$ small enough. In this case, one can introduce the following constant, complementing \eqref{eq: psi+_tau_def}:
 \begin{equation}\label{eq: psi-_tau_def}
\tau_\psi^- := \sup\left\{\,\tau \in \mathbb{R} : \psi_\tau \equiv 0\,\right\}\,.
\end{equation}
\end{definition}
In the above definition, we followed the convention $P[-\infty]=\PrIv[-\infty]=-\infty$.
Note that bounded test curves are clearly of finite energy.

 We recall the \emph{Legendre transform}, that will help establish the duality between various types of maximal test curves and geodesic rays. Given a convex function $f:[0, +\infty)\rightarrow \mathbb{R}$, its Legendre transform is defined as 
 \begin{equation}
 \hat f(\tau):= \inf_{t \geq 0} (f(t)-t\tau)=\inf_{t > 0} (f(t)-t\tau)\,,\quad \tau\in \mathbb{R}\,.
 \end{equation}
The \emph{(inverse) Legendre transform} of a decreasing concave function $g:\mathbb{R}\rightarrow \mathbb{R}\cup \{-\infty\}$ is
\begin{equation}\label{eq: inverse_Lag_tran_def}
\check{g}(t):=\sup_{\tau \in \mathbb{R}} (g(\tau)+t\tau)\,, \quad t \geq 0\,.
\end{equation}
We point out that there is a sign difference in our choice of Legendre transform compared to the convex analysis literature, however this choice will be more suitable for us.

As it is well known, for every $\tau \in \mathbb{R}$ we have that $\hat{\check{g}}(\tau) \geq g(\tau)$ with equality if and only if $g$ is additionally $\tau$-usc. Similarly, $\check{\hat{f}}(t) \leq f(t)$ for all $t\geq 0$ with equality if and only if $f$ is $t$-lsc.   In general, $\hat{\check{g}}$ is the $\tau$-usc envelope of $g$, and $\check{\hat{f}}$ is the $t$-lsc envelope of $f$. We will refer to these facts commonly as the involution property of the Legendre transform. 

Starting with a psh test curve $\{\psi_\tau\}_\tau$, our goal will be to construct a geodesic/subgeodesic ray by taking the $\tau$-inverse Legendre transform. The first step is the next proposition which was essentially proved in \cite{Da17}:

\begin{prop} \label{lem: Legendre_usc} Suppose $\{\psi_\tau\}_\tau$ is a psh test curve. Then $\sup_{\tau}(\psi_\tau(x) + t\tau)$ is usc with respect to $(t,x)\in(0,\infty)\times X$.

\end{prop}

Since $\tau^+_\psi <\infty$ and $\psi_\tau \leq 0, \tau \in \mathbb R$, we note that $ \sup_{\tau} (\psi_\tau + t\tau) \leq t \tau^+_\psi$ for $t \geq 0$. Also, for $t=0$ the above proposition may fail.

\begin{proof} 

Let $S = \{ \RE s > 0\}$. 
In the proof, $\usc(\cdot)$ will denote the usc regularization on $S \times X$. We consider the usual complexification of the inverse Legendre transform: 
\[
u(s,z) := \sup_{\tau} (\psi_\tau(z) +  \tau\RE s)\,,\quad (s,z) \in S \times X\,.
\]
Also, $u_t(x): = u(t,x) \leq  t\tau^+_\psi, t >0$. Clearly, $\usc u \in \PSH(S\times X,\pi^*\omega)$, where $\usc u$ is the usc regularization of $u$ on $S \times X$. Let $\pi:S\times X\rightarrow X$ be the natural projection.
It will be enough to show that  $\usc u=u$. 

We introduce $E = \{ u < \usc u\} \subseteq S \times X$. As both $u$ and $\usc u$ are $\mathbb{R}$-invariant in the imaginary direction of $S$, it follows that $E$ is also $\mathbb{R}$-invariant, i.e., there exists $B \subseteq (0,\infty) \times X$ such that $E = B \times \mathbb R$.

As $E$ has Monge--Ampère capacity zero, it follows that $E$ has Lebesgue measure zero. By Fubini's theorem $B \subseteq (0,\infty) \times X$ has Lebesgue measure zero as well. For $z \in X$, we introduce the slices:
\[
B_z = B \cap \left((0,\infty) \times \{z\}\right)\,.
\]
By Fubini's theorem again, we have that $B_z$ has Lebesgue measure $0$ for all $z \in X \setminus F$, where $F \subseteq X$ is some set of Lebesgue measure $0$. 

By slightly increasing $F$, but not its zero Lebesgue measure(!), we can additionally assume that $u_t(z)> -\infty$ for all $t >0$ and $z \in X \setminus F$ (indeed, at least one potential $\psi_\tau$ is not identically equal to $-\infty$). 

Let $z \in X \setminus F$. We argue that $B_z$ is in fact empty. By our assumptions on $F$, both maps $t \mapsto u_t(z)$ and $t \mapsto (\usc u)(t,z)$ are locally bounded and convex (hence continuous) on $(0,\infty)$. As they agree on the dense set $(0,\infty) \setminus B_z$, it follows that they have to be the same, hence $B_z=\emptyset$. This allows to conclude that 
\begin{equation}\label{eq: a.e._id1}
\inf_{ t> 0} [u_t(x) - \tau t] = \chi_\tau :=\inf_{ t> 0} \left[{(\usc u)}(t,x) - \tau t\right]\,, \quad \tau \in \mathbb{R}  \textup{ and } z \in X \setminus F\,.
\end{equation}
By duality of the Legendre transform $\psi_\tau(x) = \inf_{t > 0} [u_t(x) - t\tau]$ for all  $x \in X$ and $\tau \in \mathbb R$ (here is where the $\tau$-usc property of $\tau \mapsto \psi_\tau$ is used). From this and \eqref{eq: a.e._id1} it follows that $\psi_\tau=\chi_\tau$ a.e. on $X$, for all $\tau \in \mathbb{R}$. Since both $\psi_\tau$ and $\chi_\tau$ are $\omega$-psh (the former by definition, the latter by Kiselman's minimum principle \cite[Theorem~I.7.5]{De12}), it follows that in fact $\psi_\tau\equiv\chi_\tau$ for all $\tau \in \mathbb{R}$. 

Consequently, applying the $\tau$-Legendre transform to the $\tau$-usc and $\tau$-concave curves $\tau \mapsto \psi_\tau$ and $\tau \mapsto \chi_\tau$, we obtain that $u_t(x) = \usc u(t,x)$ for all $(t,x)\in (0, \infty) \times X$.
\end{proof}

Given a sublinear subgeodesic ray $\{\phi_t\}_t$ (psh test curve $\{\psi_\tau\}_\tau$), we can associate its (inverse) Legendre transform at $x \in X$ as 
\begin{equation}\label{eq: Leg_transf_def_ray_test_curve}
    \begin{aligned}
\hat \phi_\tau(x) := \inf_{t>0}(\phi_t(x) - t\tau)\,,& \quad \tau \in \mathbb R\,,\\
\check \psi_t(x) := \sup_{\tau \in \mathbb R}(\psi_\tau(x) + t\tau)\,,& \quad t> 0\,.
\end{aligned}
\end{equation}

Our main theorem describes a duality between various types of rays and maximal test curves, extending various particular cases from \cite{RWN14}, \cite{DDNL3}:
\begin{theorem}\label{thm: max_test_curve_ray_duality}
The Legendre transform $\{\psi_\tau\}_\tau \mapsto \{\check \psi_t\}_t$ gives a bijective map with inverse $\{\phi_t\}_t \mapsto \{\hat \phi_\tau\}_\tau$ between:\vspace{0.1cm}\\
\Rom{1} psh test curves and sublinear subgeodesic rays,\vspace{0.1cm}\\
\Rom{2} maximal psh test curves and psh geodesic rays,\vspace{0.1cm}\\
\Rom{3}\cite{RWN14}, \cite{DDNL3} maximal bounded test curves and bounded geodesic rays. In this case, we additionally have that
        \[
        \tau_\psi^- t \leq {\check \psi_t}\leq \tau_\psi^+t\,, \quad t \geq 0\,.
        \]
\Rom{4} maximal finite energy test curves and finite energy geodesic rays. In this case, we additionally have that
        \begin{equation}\label{eq: I_RWN_form}
            \Irad[\check \psi]=\frac{1}{V}\int_{-\infty}^{\tau^+_\psi} \left(\int_X \omega_{\psi_\tau}^n-\int_X \omega^n \right) \,\mathrm{d}\tau+\tau_{\psi}^+\,.
        \end{equation}
        Recall that the functional $I$ is defined in \eqref{eq:defIrad}.
\end{theorem}
\begin{proof} 
We prove \Rom{1}. This is essentially \cite[Proposition~4.4]{DDNL3}, where an important particular case was addressed. Let $\{\psi_\tau\}_\tau$ be a psh test curve. Then $\check \psi_t \in \PSH(X,\omega)$ for all $t>0$ due to  Proposition~\ref{lem: Legendre_usc}. We also see that $\sup_X \check \psi_t \leq t \tau^+_\psi$, and $\check \psi_t \to_{L^1} 0$ as $t \to 0$, proving that $\{\check{\psi}_t\}_t$ is a subgeodesic.

For the reverse direction, let $\{\phi_t\}_t$ be a sublinear subgeodesic ray. Then $\hat \phi_\tau \in \PSH(X,\omega)$ or $\hat \phi_\tau \equiv \infty$ for any $\tau \in \mathbb R$ due to Kiselman's minimum principle. By properties of Legendre transforms and Lemma~\ref{lem: -inf_est_subgeod}, we get that $\tau \mapsto \hat \phi_\tau(x)$ is $\tau$-usc, $\tau$-concave and decreasing. Due to sublinearity of $\{\phi_t\}_t$ we get that $\hat \phi_\tau \equiv -\infty$ for $\tau$ big enough. Lastly $\psi_\tau \nearrow 0$ a.e. as $\tau \to -\infty$, since $\phi_t \to_{L^1} 0$ as $t \to 0$.

We prove \Rom{2}. From \cite[Propisition 5.1]{Da17} (that only uses the maximum principle \eqref{eq: vabt_eq}) we obtain that for any psh geodesic ray $\{u_t\}_t$, the curve $\{ \hat u_\tau\}_\tau$ is a maximal psh test curve.

Let $\{\psi_\tau\}_\tau$ be a maximal psh test curve. We will show that the sublinear subgeodesic $\{\check \psi_t\}_t$ is a psh geodesic ray. By elementary properties of the Legendre transform we can assume that $\tau^+_\psi =0$, in particular $\{\check \psi_t\}_t$ is $t$-decreasing. 

Now assume by contradiction that $\{\check \psi_t\}_t$ is not a psh geodesic ray. Comparing with \eqref{eq: vabt_eq}, there exists $0 < a < b$ such that 
\[
\check \psi_{(1-t)a + tb} \lneq \chi_t:=\sup_{h \in \mathcal{S}} h_t\,,\quad t \in [0,1]\,,
\]
where $\mathcal{S}$ is the set of subgeodesics $(a, b) \ni  t \mapsto h_t \in \PSH(X,\omega)$  satisfying $\displaystyle\lim_{t \to a+} h_t \leq  \check \psi_a$ and $\displaystyle\lim_{t \to b-} h_t \leq  \check \psi_b$. Now let $\{\phi_t\}_t$ be the sublinear subgeodesic such that $\phi_t := \check \psi_t$ for $t \not\in (a,b)$ and $\phi_{a(1-t) + bt} := \chi_t$ otherwise.

Trivially, $\check \psi_t \leq \phi_t \leq 0$, hence by duality, $\psi_\tau \leq \hat \phi_\tau$ and $\tau^+_\psi = \tau^+_{\hat \phi}=0$.  However, comparing with \eqref{eq: Leg_transf_def_ray_test_curve}, we claim that $ \hat \phi_\tau \leq  \psi_\tau + \tau(a-b)$ for any $\tau \in \mathbb R$. Since $\tau^+_\psi = \tau^+_{\hat \phi}=0$, we only need to show this for $\tau \leq 0$. For such $\tau$ we indeed have 
\[
\inf_{t \in [a,b]}(\phi_t - t \tau) \leq \phi_b - b \tau = \check \psi_b - b \tau  \leq \inf_{t \in [a,b]}(\check \psi_t - t \tau) +  (a-b) \tau\,,
\]
where in the last inequality we used that $t \mapsto \check \psi_t$ is decreasing.

By the maximality of $\{\psi_\tau\}_\tau$, we obtain that $\psi_\tau = \hat \phi_\tau$. An application of the Legendre transform now gives that $\check \psi_t=\phi_t$, a contradiction. Hence $\{\psi_t\}_t$ is a psh geodesic ray.

The duality of  \Rom{3}  is simply \cite[Theorem~1.3]{DDNL3}, closely following \cite{RWN14}.

We deal with  \Rom{4}. As before, we may assume that $\tau_{\psi}^+=0$. As a preliminary result, in Proposition~\ref{prop: RWN_I_formula} below we prove \eqref{eq: I_RWN_form} for bounded maximal test curves.

Given a finite energy maximal test curve $\{\psi_\tau\}_\tau$, we know that $\{\check \psi\}_t$ is a psh geodesic ray. 
By \cite[Theorem~4.5]{DL20} and its proof there exists bounded geodesic rays $\{\check \psi^k_t\}_t$ such that $\check \psi^k_t \searrow \check \psi_t$ for any $t 
\geq 0$, and $\int_X \omega^n_{\psi^k_\tau } \searrow \int_X \omega^n_{\psi_\tau }$ for any $\tau < \tau^+_\psi = \tau^+_{\psi^k}=0$ (see especially the last displayed equation of \cite[pp. 17]{DL20}). Indeed, the arguments of \cite[Theorem~4.5, Lemma~4.6]{DL20} work for general psh rays, without change.

By Proposition~\ref{prop: RWN_I_formula} below
\[
\Irad[\check \psi^k_t]=\frac{1}{V}\int_{-\infty}^{0} \left( \int_X \omega_{\psi^k_\tau}^n-\int_X \omega^n\right)\,\mathrm{d}\tau\,.
\]

The right hand side is bounded from below, since $\{\psi_\tau\}_\tau$ is a finite energy test curve.
Since $\int_X \omega^n_{\psi^k_\tau } \searrow \int_X \omega^n_{\psi_\tau }$, we can take the limit on both the left and right hand side, to arrive at \eqref{eq: I_RWN_form}, also implying that $\{\check \psi_t\}_t$ is a finite energy geodesic ray. 

Conversely, assume that $\{\phi_t\}_t$ is a finite energy geodesic ray, with decreasing approximating sequence of bounded rays $\{\phi_t^k\}_t$, as detailed above. For similar reasons we have
$
\Irad[\phi^k_t]=\frac{1}{V}\int_{-\infty}^{0} \left( \int_X \omega_{\hat \phi^k_\tau}^n-\int_X \omega^n\right)\,\mathrm{d}\tau\,.
$
Since $\textbf{I}\{\phi_t^k\} \searrow \Irad[\phi_t]$,  the monotone convergence theorem gives that \eqref{eq: I_RWN_form} holds for $\{\hat \phi_\tau\}_{\tau}$, finishing the proof.
\end{proof}

As promised, to complete the argument of Theorem \ref{thm: max_test_curve_ray_duality}, we prove the next proposition, whose argument can be extracted from \cite[Section~6]{RWN14} with additional references to \cite{DDNL2}. We recall the precise details here as the results of \cite{RWN14} were proved in the context of potentials with small unbounded locus.
\begin{prop}\label{prop: RWN_I_formula} Suppose that $\{\psi_\tau\}_\tau$ is a bounded maximal test curve with $\tau^+_\psi = 0$. Then
\begin{equation}\label{eq: I_bounded_formula_RWN}
\frac{\mathrm{I}(\check \psi_t)}{t}=\Irad[\check \psi_t]= \frac{1}{V}\int_{-\infty}^{0} \left( \int_X \omega_{\psi_\tau}^n-\int_X \omega^n\right)\,\mathrm{d}\tau\,, \quad t >0\,. 
\end{equation} 
\end{prop}
\begin{proof} Without loss of generality we assume that $V=1$. For $N\in \mathbb{Z}_+,M \in \mathbb{Z}$ and $t >0$, we introduce the following:
\[
\check{\psi}^{N,M}_t := \max_{\substack{k \in \mathbb{Z} \\ k \leq M}} \left(\psi_{k/2^N} +  tk/2^N\right)\,.
\]
It is clear that $\check \psi_t^{N,M} \in \PSH(X,\omega) \cap L^\infty(X)$, since it is a maximum of a finite number of $\omega$-psh potentials (here we also used that $\{\psi_\tau\}_\tau$ is a bounded test curve). Moreover, we now argue that
\begin{equation}\label{eq: diff_eq_I}
 \frac{t}{2^N} \int_X \omega^n_{\psi_{(M+1)/2^N}}\leq \mathrm{I}(\check \psi_t^{N,M+1})-\mathrm{I}(\check \psi_t^{N,M})\leq \frac{t}{2^N} \int_X \omega^n_{\psi_{M/2^N}}\,.
\end{equation}
Indeed, for elementary reasons:
\begin{equation}\label{eq: first_I_ineq}
\int_X \left(\check{\psi}_t^{N,M+1}-\check{\psi}_t^{N,M}\right)\,\omega^n_{\check \psi_t^{N,M+1}}\leq \mathrm{I}(\check{\psi}_t^{N,M+1})-\mathrm{I}(\check{\psi}_t^{N,M})\leq \int_X \left(\check{\psi}_t^{N,M+1}-\check{\psi}_t^{N,M}\right)\,\omega^n_{\check \psi_t^{N,M}}\,.
\end{equation}
Clearly $\check{\psi}_t^{N,M+1} \geq \check{\psi}_t^{N,M}$, and using $\tau$-concavity we notice that 
\[
 U_t := \left\{\,\check \psi^{N,M+1}_t-\check \psi^{N,M}_t>0\,\right\} = \left\{\,\psi_{(M+1)/2^N} +  2^{-N}t -\psi_{M/2^N}>0\,\right\}\,.
\]
Moreover, on $U_t$ we have 
\[
\check \psi_t^{N,M+1} = \psi_{(M+1)/2^N} + t (M+1)/2^{N}\,,\quad \check \psi_t^{N,M} = \psi_{M/2^N} + t M /2^{N}\,. 
\]
 We also note that $U_t$ is an open set in the plurifine topology, implying that $\omega^n_{\psi_{(M+1)/2^N}}\big|_{U_t}=\omega^n_{\check \psi^{N,M+1}_t}\big|_{U_t}$ and $\omega^n_{\psi_{M/2^N}}\big|_{U_t}=\omega^n_{\check \psi^{N,M}_t}\big|_{U_t}$. Recall that $\omega^n_{\psi_{M/2^N}}$ and $\omega^n_{\psi_{(M+1)/2^N}}$ are supported on the sets $\{\psi_{M/2^N} =0\}$ and $\{\psi_{(M+1)/2^N} =0\}$ respectively \cite[Theorem~3.8]{DDNL2}.  Since  $\{\psi_{(M+1)/2^N} =0\} \subseteq U_t$ and $\{\psi_{(M+1)/2^N} =0\} \subset \{\psi_{M/2^N} =0\}$, applying the above to \eqref{eq: first_I_ineq}, we arrive at \eqref{eq: diff_eq_I}.

Fixing $N$, let $M$ be the biggest integer to the left of $2^{N}\tau^{-}_\psi$. Then repeated application of \eqref{eq: diff_eq_I} yields
\[
\sum_{M+1 \leq j \leq 0} \frac{t}{2^N} \int_X \omega^n_{\psi_{j/2^N}}\leq \mathrm{I}(\check \psi_t^{N,0})-\mathrm{I}(\check \psi_t^{N,M})\leq \sum_{M \leq j \leq -1} \frac{t}{2^N} \int_X \omega^n_{\psi_{j/2^N}}\,.
\]
Since $M \leq 2^{N}\tau^-_\psi$ we have that $\check \psi^{N,M}_t = \psi_{M/2^N} + t M/2^N=t M/2^N$, we can continue to write 
\[
\sum_{j=M+1}^0 \frac{t}{2^N} \left(\int_X \omega^n_{\psi_{j/2^N}} - \int_X \omega^n\right)\leq \mathrm{I}(\check \psi_t^{N,0})\leq \sum_{j=M}^{-1} \frac{t}{2^N} \left(\int_X \omega^n_{\psi_{j/2^N}} - \int_X \omega^n\right)\,.
\]
We now notice that we have Riemann sums on both the left and right of the above inequality. Using Lemma~\ref{lma:contimass} below, it is possible to let $N \to \infty$ and obtain \eqref{eq: I_bounded_formula_RWN}, as desired.
\end{proof}
\begin{lemma}\label{lma:contimass}
Suppose that $\{\psi_\tau \}_{\tau}$ is a psh test curve. Then $\tau \mapsto \int_X \omega_{\psi_\tau}^n>0$ is a continuous function for $\tau \in (-\infty, \tau^+_\psi)$.
\end{lemma}

By working harder, using \cite[Theorem~B]{DDNL5}, one can show that $\tau \mapsto \big(\int_X \omega_{\psi_\tau}^n\big)^{1/n}$ is concave, however we will not need this in the sequel.

\begin{proof} First we argue positivity. Since $\psi_\tau \nearrow 0$ a.e, as $\tau \to -\infty$, \cite[Theorem~2.3]{DDNL2} gives $\int_X \omega_{\psi_\tau}^n \nearrow \int_X \omega^n > 0$ as $\tau \to -\infty$. Let $\tau \in (-\infty,\tau^+_\psi)$ be arbitary. Pick $\tau_1 \in (\tau,\tau^+_\psi)$ and $\tau_0 < \tau$ such that $\int_X \omega_{\psi_{\tau_0}^n} > 0$. Let $\alpha \in (0,1)$ such that $\tau = \alpha \tau_0 + (1-\alpha) \tau_{1}$. By $\tau$-concavity and \cite[Theorem~1.1]{WN19} we obtain that $\int_X \omega^n_{\psi_\tau} \geq \int_X \omega^n_{\psi_{\alpha \tau_0 + (1-\alpha) \tau_{1}}} \geq \alpha^n \int_X \omega^n_{\psi_{\tau_0}}>0$, as desired.

Next, we argue continuity. We know that $\tau \mapsto \psi_\tau$ is $\tau$-decreasing. Fix $\tau_0 \in (-\infty, \tau^+_\psi)$ then $\int_X \omega_{\psi_\tau}^n \nearrow \int_X \omega_{\psi_{\tau_0}}^n$ as $\tau \searrow \tau_0$ by \cite[Theorem~2.3]{DDNL2}. Now we argue that $\int_X \omega_{\psi_\tau}^n \searrow \int_X \omega_{\psi_{\tau_0}}^n$ as $\tau \nearrow \tau_0$. For $\epsilon >0$ small, using the $\tau$-convexity of $\tau \mapsto \psi_\tau$ together with monotonicity and multi-linearity of the non-pluripolar measure, we have
\[
\frac{1}{2^n} \int_X \omega^n_{\psi_{\tau_0 - \epsilon}}  + \frac{2^n - 1}{2^n} \int_X \omega^n_{\psi_{\tau_0 + \epsilon}}\leq \int_X \omega^n_{\frac{1}{2}\psi_{\tau_0 - \epsilon} + \frac{1}{2}\psi_{\tau_0 + \epsilon}} \leq \int_X \omega^n_{\psi_{\tau_0}}\,.
\]
Letting $\epsilon \searrow 0$, we know that $\int_X \omega^n_{\psi_{\tau_0 + \epsilon}} \to \int_X \omega^n_{\psi_{\tau_0}}$, hence  $\int_X \omega_{\psi_{\tau_0 - \epsilon}}^n \searrow \int_X \omega_{\psi_{\tau_0}}^n$.
\end{proof}

The technique in the proof of Proposition~\ref{prop: RWN_I_formula} can also be applied to other energy functionals. We refer to \cite{Xia20} for details.

\subsection{Rays induced by filtrations and approximability}\label{subsec:rayfiltandapp}

We fix a filtration $(\{\mathcal{F}^{\lambda}_k\},r )$ on $R(X,L^r)$. Following \cite[Section~7]{RWN14}, one can associate a maximal test curve to this filtration in the following manner. The corresponding construction for test configurations is due to Phong--Sturm \cite{PS07}, \cite{PS10}. For $\tau \in \mathbb{R}$, let 
\begin{equation}\label{eq: filtration_to_test_curve_def1}
\hat u_{\tau}^k := \sup_{ \substack{s \in \mathcal{F}^{\tau k}_k\,,  h^k(s,s) \leq 1}} \frac{1}{kr} \log h^{kr}(s,s) \leq 0\,.
\end{equation}
Since each $\mathcal{F}^{\tau  k}_k$ is finite dimensional, one notices that $\hat u^k_\tau \in \PSH(X,\omega)$ has analytic singularity type. Moreover, by the multiplicativity of the filtration we have that 
\begin{equation}\label{eq: mult_filt}
k \hat u^k_\tau+ k' \hat u^{k'}_\tau\leq (k + k') \hat u^{k + k'}_\tau \leq 0\,.
\end{equation}
As a result, Fekete's lemma implies that $\hat u_\tau := \lim_{k} \hat u^k_\tau = \sup_k \hat u^k_\tau \in \PSH(X,\omega)$ exists, and the curve $\tau \mapsto \hat u_\tau$ has a number of special properties.

\begin{theorem}\cite[Proposition~7.7, Proposition~7.11]{RWN14} For  any filtration $(\{\mathcal{F}^\lambda_k\},r)$ the potentials $\{\hat u_\tau\}_\tau$ form a maximal bounded test curve. In particular, by Theorem~\ref{thm: max_test_curve_ray_duality} they induce a ray of bounded potentials $\{ u_t\}_t \in \mathcal{R}^\infty$.
\end{theorem}
We give a very brief sketch of the argument. As elaborated in \cite[Section~7]{RWN14}, that $\{\hat u_\tau\}_\tau$ is $\tau$-concave and $\tau$-decreasing is a consequence of the multiplicativity of the filtration. Boundedness follows from linear boundedness of the filtration. To make sure that $\{\hat u_\tau\}_\tau$ is $\tau$-usc we take $\hat u_{\tau_u^+} := \lim_{\tau \searrow {\tau_u^+}}\hat u_\tau$ \cite[Lemma~4.3]{DDNL3}. Regarding maximality, $P[\hat u_\tau] = \hat u_\tau$, for $\tau < \tau_u^+$ is a consequence of the Skoda division theorem. That $P[\hat u_{\tau_u^+}] = \hat u_{\tau_u^+}$ follows since $\hat u_\tau \searrow \hat u_{\tau_u^+}$ as ${\tau \searrow {\tau_u^+}}$ \cite[Corollary~4.7]{DDNL5}.

Recall the notion of $\mathcal I$-maximal test curves from Definition~\ref{def: test_curves}. As $\veq$-maximal test curves are maximal (Lemma~\ref{prop:Qprojection}), we give the following improvement to the above result:
\begin{theorem} \label{thm: filtration_I_max_testcurve}For  any filtration $(\{\mathcal{F}^\lambda_k\},r)$ the curve $\{\hat u_\tau\}_\tau$ is a bounded $\veq$-maximal test curve.
\end{theorem}
\begin{proof} Using the previous result, we only have to show that $\PrIv[\hat u_\tau]= \hat u_\tau$ for $\tau \leq \tau_u^+$. Due to \eqref{eq: mult_filt} we have that $\hat u^{2^j}_\tau \leq \hat u^{2^{j+1}}_\tau$, moreover $\hat u_\tau = \lim_j \hat u^{2^j}_\tau$. By maximality of $\hat u_\tau$, we have that  $\hat u^{2^j}_\tau \leq P[\hat u^{2^j}_\tau]\leq P[\hat u_\tau] =\hat u_\tau$, in particular, $ P[\hat u^{2^j}_\tau] \nearrow \hat u_\tau$.

Let us assume momentarily that $\tau < \tau^+_u$. Then $\int_X \omega_{\hat u_\tau} > 0$ by Lemma~\ref{lma:contimass}. By Lemma~\ref{lma:decvmodel}\Rom{3}, to conclude that $\PrIv[\hat u_\tau] = \hat u_\tau$, we only need to argue that $P[\hat u^{2^j}_\tau]$ is $\veq$-model, which follows from Proposition~\ref{prop: anal_sing_type_envelope}.

In case when $\tau = \tau^+_u$, notice that $\PrIv[\hat u_{\tau-\epsilon}]= \hat u_{\tau-\epsilon} \searrow \hat u_{\tau}$ as $\epsilon \searrow 0$. Hence by Lemma~\ref{lma:decvmodel}\Rom{1}, and what we just proved, we get that $\PrIv[\hat u_{\tau}]= \hat u_{\tau}$, as desired.
\end{proof}

\paragraph{Test configurations and approximable rays.}
We introduce some preliminary terminology, aiding our discussion in this paragraph.
\begin{definition}
We say that a ray $\{r_t\}_t \in \mathcal{R}^1$ is \emph{approximable}  if 
there exists rays $\{r^j_t\}_t \in \mathcal{R}^1$ induced by test configurations such that $r^j_t \searrow r_t$ for $t \geq 0$. 
\end{definition}
In the terminology of \cite{BBJ21}, approximable rays  $\{r_t\}_t$ are called \emph{maximal}. By \cite[Lemma~4.3]{DL20} we obtain that $d_1^c (\{r^j_t\}_t,\{r_t\}_t) \to 0$, in particular $\{r_t\}_t \in \overline{\mathcal{T}}$, where $\mathcal{T}$ is the set of rays induced by ample test configurations. 

Due to the completeness of $\mathcal{E}^{1,\NA}$ proved by Boucksom--Jonsson (see \cite[Example~3.3]{Xi19}, or Theorem \ref{thm: Pi_cont} below), $\{v_t\}_t \in \overline{\mathcal{T}}$ if and only if it is approximable, but we will not rely on this property in the present section.

We recall the following potential theoretic interpretation of $r^{\NA}$ from \cite[Section~4.3]{BBJ21} in terms of the Legendre transform:
\begin{prop}\label{prop: NA_Legendre_formula} Let $\{r_t\}_t \in \mathcal{R}^1$ such that $ \tau^+_{\hat r}=\sup_X r_t \leq 0$, for $t \geq 0$. Let $v \in X^{\Div}_\mathbb{Q}$. Then
\begin{flalign}\label{eq: NA_Gauss_no_Gauss}
r^{\NA}(v)&=- \inf_{\tau < \tau^+_{\hat r}}(v(\hat{r}_{\tau})-\tau)= -\inf_{\tau \in \mathbb R}(v(\hat{r}_{\tau})-\tau)\\
&=\sup_{\tau < \tau^+_{\hat r}}(\hat{r}^\NA_{\tau}(v)+\tau)=\sup_{\tau \in \mathbb R}(\hat{r}^\NA_{\tau}(v)+\tau)\,.  \nonumber
\end{flalign}
where $\{\hat r_\tau\}_\tau$ is the maximal finite energy test curve of $\{r_t\}_t$. By convention, $v(-\infty)=\infty$.
\end{prop}

\begin{proof} We only need to prove the very first equality. Recall from Theorem~\ref{thm: max_test_curve_ray_duality} the following duality between $\{r_t\}_t$ and its maximal test curve:
\begin{equation}\label{eq: dual_Legendre}
r_{t}=\sup_{\tau < \tau^+_{\hat r}} (\hat r_{\tau}+t\tau)\,,\quad t\geq 0\,.
\end{equation}

Let $v\in X^{\Div}_\mathbb{Q}$, and $\sigma(v) \in (X \times \mathbb{C})^{\Div}_\mathbb{Q}$ be the corresponding Gauss extension (see Section~\ref{subsec:nonAchi}). Since $ \tau^+_{\hat r} \leq 0$, by  Lemma~\ref{lma:Lelonginc} below we conclude that $-r^{\NA}(v)  =\sigma(v)(\Phi)=\inf_{\tau < \tau^+_{\hat r}} (v(\hat{r}_{\tau})-\tau)$.
\end{proof}
\begin{lemma}\label{lma:Lelonginc}
Let $\Omega$ be a complex manifold. Let $\mathcal{F}$ be a non-empty family of non-positive psh functions on $\Omega$ and $\displaystyle
\psi:=\usc\left(\sup_{\varphi\in \mathcal{F}}\varphi\right)$.
Then for any $x\in \Omega$,
\begin{equation}\label{eq:Lelonginc}
\nu(\psi,x)=\inf_{\varphi\in \mathcal{F}} \nu(\varphi,x)\,.
\end{equation}
\end{lemma}
\begin{proof}
By Choquet's lemma, we may assume that $\mathcal{F}$ consists of only countably many functions $\varphi_j$ $(j\in \mathbb{N})$ and $\psi=\usc \left(\sup_{j\in \mathbb{N}}\varphi_j\right)$.

By upper semicontinuity of Lelong numbers,
$\nu(\psi,x)\geq \varlimsup_{j\in \mathbb{N}} \nu( \max\{\varphi_0,\ldots,\varphi_j\},x)$.
In addition, by monotonicity of Lelong numbers, $
\nu(\psi,x)= \lim_{j\in \mathbb{N}} \nu (\max\{\varphi_0,\ldots,\varphi_j\},x)= \inf_j \nu (\varphi_j,x)$, where in the last step we used \cite[Corollary~2.10]{Bo17}.
\end{proof}

Given a ray $\{r_t\}_t \in \mathcal{R}^1$, we define two associated envelopes, based on the ideas of \cite{BBJ21}. For $t \geq 0$ let
\[
\Pi(r_t):=\inf \left\{\, r'_t:  \{r'_t\}_t \in \mathcal{R}^1 \textup{ is induced by a test configuration and } r'_t \geq r_t\,\right\}\,,
\]
A priori, it is not even clear that $\{\Pi(r_t)\}_t$ is a geodesic ray. On the other hand, following the argument of \cite[Theorem~6.6]{BBJ21}, for $t \geq 0$ we can also consider 
\[
\pi(r_t):=\sup\left\{\,  \{r''_t\}_t \in \mathcal{R}^1 : {r''}^{\NA}=r^{\NA}\,\right\}\,.
\]
As we prove now, these two projections coincide to give a ray, whose maximal test curve can be described concretely:
\begin{theorem} \label{thm:BBJ21projnaequal}
Let $\{r_t\}_t\in \mathcal{R}^1$. Then $\{\Pi(r)_t\}_t$ is an approximable geodesic ray. Moreover the following hold: \vspace{0.1cm}\\
\Rom{1} $\widehat{\Pi(r)}_\tau = \widehat{\pi(r)}_\tau = \PrIv[\hat r_\tau]$, $\tau \neq  \tau^+_{\hat r}$.\vspace{0.1cm}\\
\Rom{2} $\Pi(r)_t = \pi(r)_t$ for $t\geq 0$. Moreover $\widehat{\Pi(r)}_\tau  \simeq_{\veq} \hat r_\tau$, $\tau \neq  \tau^+_{\hat r}$.\vspace{0.15cm}\\
\Rom{3} $\Pi(r)^{\NA}=\pi(r)^{\NA} = r^{\NA}$.
\end{theorem}

Since $\{\Pi(r)_t\}_t$ is always approximable, we get that $\Pi \circ \Pi = \Pi$, i.e., $\Pi$ is a projection. It is not clear if \Rom{1} and \Rom{2} hold for $\tau = \tau^+_{\hat r}$, though this is not essential to our discussion.

\begin{proof} First we observe that no generality is lost if we assume the condition $\sup_X r_t \leq 0$ for $t \geq 0$, after possibly replacing $r_t$ with $r_t - m t$ for some $m \in \mathbb N$ big enough.

We note that  \Rom{1}  immediately implies  \Rom{2}, as $\PrIv[\hat r_\tau] \simeq_{\veq} \hat r_\tau$ (Proposition~\ref{prop:Qprojection}\Rom{2}). On the other hand, due to Theorem~\ref{thm: Lelong_characterization} and \eqref{eq: NA_Gauss_no_Gauss}, we have that  \Rom{2}  implies  \Rom{3}  as well. Lastly,  \Rom{1}  together with Lemma~\ref{lem: BBJ21_approx} below imply that $\{\Pi(r)_t\}_t$ is approximable.

Hence we only need to argue  \Rom{1}. In fact, from Lemma~\ref{lem: BBJ21_approx} below we know that $\{\pi(r_t)\}_t$ is an approximable ray, so $\Pi(r)_t\leq \pi(r)_t$, since $r_t \leq \pi(r)_t$. To show that $\Pi(r)_t\geq \pi(r)_t$, it suffices to show that for any Phong--Sturm ray $\{w_t\}_t$  satisfying $w_t \geq r_t$, we have $w_t \geq \pi(r)_t$.
We have that ${w}^{\NA} \geq r^{\NA}={r''}^{\NA}$ for any candidate $\{r''_t\}_t$ of $\{\pi(r_t)\}_t$. Hence,  \cite[Lemma~4.6]{BBJ21} implies that $w_t\geq r''_t$. Taking supremum over $\{r''_t\}_t$ we obtain  $w_t \geq \pi(r)_t$, finishing the proof of $\Pi(r)_t = \pi(r)_t$. 

Due to how $\Pi(r_t)$ is defined, we immediately obtain $\tau^+_{\widehat{\Pi(r)}} = \tau^+_{\widehat{\pi(r)}} = \tau^+_{\hat{r}}$, giving 
\[
\widehat{\Pi(r)}_\tau = \widehat{\pi(r)}_\tau = \PrIv[\hat r_\tau] = -\infty\,, \quad \tau > \tau^+_{\hat{r}}\,.
\]

To finish the proof of  \Rom{1}, we need to show that $\widehat{\pi(r)}_\tau = \PrIv[\hat r_\tau]$, for $\tau < \tau^+_{\hat r}$. Due to the lemma below, $\{\pi[r]_t\}_t$ is approximable.  Hence, by  Theorem~\ref{thm: filtration_I_max_testcurve} and Lemma~\ref{lma:decvmodel}\Rom{1}, $\widehat{\pi(r)}_\tau$ is $\veq$-maximal for any $\tau \in \mathbb R$. In particular, to show that $\widehat{\pi(r)}_\tau = \PrIv[\hat r_\tau]$, it is enough to argue that $\widehat{\pi(r)}_\tau \simeq_{\veq} \hat r_\tau$ for $\tau < \tau^+_{\hat r}$. To show this, due to Proposition~\ref{prop: I-equiv_NA} it is enough to argue that 
\begin{equation}\label{eq: I_equiv_to_show}
\widehat{\pi(r)}_\tau^\NA(v)= \hat r_\tau^\NA(v)\,,
\end{equation}
for all $v = c\ord_E\in X^{\Div}_{\mathbb{Q}}$ and $\tau < \tau^+_{\hat r}$. 

Since test curves are $\tau$-concave, both sides in \eqref{eq: I_equiv_to_show} are $\tau$-concave on $\mathbb R$, with the $\tau$-usc property failing at most at $\tau = \tau^+_{\hat r}$, the point of discontinuity. 

Due to the comments following \eqref{eq: inverse_Lag_tran_def}, to argue \eqref{eq: I_equiv_to_show} it is enough to show that both sides have the same  Legendre transform on positive rational values, namely
\[
\sup_{\tau < \tau^+_{\hat r}}(\widehat{\pi(r)}_\tau^\NA(v)+ t\tau)=\sup_{\tau < \tau^+_{\hat r}}(\hat r_\tau^\NA(v)+ t\tau)\,, 
\]
for any $t\in \mathbb{Q}_{>0}$. We may assume $t=1$ by considering the valuation $t^{-1}v$ instead. From \eqref{eq: NA_Gauss_no_Gauss}, this is equivalent to $r^{\NA}(v)=\pi(r)^{\NA}(v)$, which is known to hold by the lemma below.
\end{proof}

\begin{lemma}[\cite{BBJ21}] \label{lem: BBJ21_approx} For any $\{r_t\}_t \in \mathcal{R}^1$ the ray $\{\pi(r)_t\}_t \in \mathcal{R}^1$ is approximable and $\pi(r)^\NA = r^\NA$.
\end{lemma}

\begin{proof} As before, we can assume that $\sup_X r_t \leq 0$, for $t \geq 0$, after possibly replacing $r_t$ with $r_t - m t$ for some $m \in \mathbb N$ big.
Recall that in the proof of \cite[Theorem~6.2]{BBJ21}, one constructs a sequence of globally generated flag ideals $\mathfrak a^k$, such that $\varphi_{\mathfrak a^k}\in \mathcal{H}^{\NA}$ and $\varphi_{\mathfrak a^k} \searrow r^{\NA}$. Let $\{r^k_t\}_t$ be the Phong--Sturm geodesic ray induced by the fractional ideals $\mathfrak a^k$. 

Due to  \cite[Lemma~4.6]{BBJ21}, $r^k_t\geq \rho_t$ for any ray $\{\rho_t\}_t \in \mathcal{R}^1$ that is a candidate for $\{\pi(r)_t\}_t \in \mathcal{R}^1$. Let $\{r'_t\}_t$ be the decreasing limit of $\{r^k_t\}_t$. Since $r^k_t\geq r'_t\geq r_t$, and ${r^k}^{\NA}= \varphi_{\mathfrak a^k} \to r^{\NA}$, we get that  $r'^{\NA}=r^{\NA}$. This implies that $\{r'_t\}_t$ is a candidate for $\{\pi(r)_t\}_t$ implying that $\{r'_t\}_t = \{\pi(r)_t\}_t$.
\end{proof}

Finally, we arrive at one of the main results of this section: 

\begin{theorem}\label{thm: bijection_I_max_test_curves}     There is a bijective correspondence between $\veq$-maximal finite energy test curves and the approximable geodesic rays of $\mathcal R^1$.
\end{theorem}
\begin{proof}Let $\{r_t\}_t \in \mathcal{R}^1$ be an approximable geodesic ray. 
Let $\{r^k_t\}_t$ be a sequence of Phong--Sturm geodesic rays decreasing to $\{r_t\}_t$. Since test configurations induce filtrations, that in turn induce geodesic rays (see Section~\ref{subsec:nonAchi}), we can use Theorem~\ref{thm: filtration_I_max_testcurve} to conclude that  $\{\widehat{r}^k_\tau\}_\tau$ is $\veq$-maximal. So $\{\hat{r}_\tau\}_\tau$ is $\veq$-maximal by Lemma~\ref{lma:decvmodel}\Rom{1}.

Assume now that $\{\psi_{\tau}\}_\tau$ is an $\veq$-maximal finite energy test curve. Due to $\veq$-maximality, by Theorem~\ref{thm:BBJ21projnaequal}\Rom{1} we have that 
\[
\widehat{\Pi(\check{\psi})}_\tau = \PrIv[ \psi_\tau] =  \psi_\tau\,, \quad \tau \neq \tau^+_\psi\,. 
\]
By duality, $\{\Pi(\check{\psi})_t\}_t = \{ \psi_t\}_t \in \mathcal R^1$. Finally, by Theorem~\ref{thm:BBJ21projnaequal} the ray $\{\Pi(\check{\psi})_t\}_t$ is approximable, finishing the proof.
\end{proof}

In addition to the above characterization, we show below that the projection $\Pi$ is continuous. First, we recall radial analogs of some known properties of $(\mathcal E^1, d_1)$. 

Given $\{u_t\}_t,\{v_t\}_t \in \mathcal R^1$, it is possible to construct $\{\max_\mathcal R(u,v)_t\}_t, \{P_\mathcal R(u,v)_t\}_t \in \mathcal R^1$ the smallest/biggest ray that is above/below $\{u_t\}$ and $\{v_t\}$ respectively. The ray $\{P_\mathcal R(u,v)_t\}_t$ was constructed in \cite[Example 3.2]{Xi19}, and $\{\max_R(u,v)_t\}_t$ was constructed above \cite[Proposition 2.15]{DDNL5}. These two rays satisfy the following metric estimates/identities for some $C(n) >1$, as argued in  \cite[Proposition 2.15]{DDNL5} and \cite[Example 3.2]{Xi19}:
\begin{gather}
d_1^c(\{u_t\}_t,\{v_t\}_t) \leq  d_1^c(\{u_t\}_t,\{{\max}_\mathcal R(u,v)_t\}_t)) + d_1^c(\{{\max}_\mathcal R(u,v)_t\}_t,\{v_t\}_t) \leq C d_1^c(\{u_t\}_t,\{v_t\}_t) , \nonumber \\
d_1^c(\{u_t\}_t,\{v_t\}_t) =  d_1^c(\{u_t\}_t,\{P_\mathcal R(u,v)_t\}_t)) + d_1^c(\{P_\mathcal R(u,v)_t\}_t,\{v_t\}_t)\, .\label{eq: d_p_decomp}
\end{gather}

\begin{theorem}\label{thm: Pi_cont} The projection map $\Pi: \mathcal R^1 \to  \overline{\mathcal{T}}$ is $d_1^c$-continuous. In particular, the set of approximable rays is $d_1^c$-closed.
\end{theorem}
The last sentence also follows from the completeness of $\mathcal{E}^{1,\NA}$ proved by Boucksom--Jonsson (\cite[Theorem~9.8]{BJ18b}). 
This theorem proves the first part of Theorem~\ref{main_thm: Non-Archimedean}.

\begin{proof} Let $\{u^j_t\}_t,\{u_t\}_t \in \mathcal R^1$ with $d_1^c(\{u^j_t\},\{u_t\}) \to 0$. To derive a contradiction, we can suppose that  $d_1^c(\Pi\{u^j_t\},\Pi\{u_t\}) \geq \delta > 0$.

After possibly taking a subsequence of $\{u^j_t\}_t$, the radial version of \cite[Proposition~2.6]{BDL17} (whose proof is the same, and only depends on the estimates \eqref{eq: d_p_decomp}) gives existence of two sequences $\{v^j_t\}_t, \{w^j_t\}_t \in \mathcal R^1$ that are decreasing and increasing respectively, satisfying $w^j_t \leq u^j_t \leq v^j_t$, $w^j_t \leq u_t \leq v^j_t$ together with $d_1^c(\{v^j_t\},\{u_t\}) \to 0$, $d_1^c(\{w^j_t\},\{u_t\}) \to 0$. For closely related arguments, see \cite[Proposition~3.1]{Xi19} and \cite[Proposition~4.2]{DDNL5}.

Naturally we also get $\Pi\{w^j\}_t \leq \Pi\{u^j_t\} \leq \Pi\{v^j_t\}$, $\Pi\{w^j\}_t \leq \Pi\{u_t\} \leq \Pi\{v^j_t\}$, hence to conclude it is enough to show that  $\Pi\{w^j_t\} \nearrow \Pi\{u_t\}$ a.e., and $\Pi\{v^j_t\} \searrow \Pi\{u_t\}$ for all $t \geq 0$ \cite[Lemma 4.3]{DL20}. 

 Note that we have $\sup_X v^j_1 = \tau^+_{\hat v^j} \searrow \tau^+_{\hat u} =\sup_X u_1$ and $\sup_X w^j_1 = \tau^+_{\hat w^j} \nearrow \tau^+_{\hat u} =\sup_X u_1$ by the Hartogs lemma for $L^1$-convergence of quasi-psh functions. Because of this, by the duality of Theorem \ref{thm: bijection_I_max_test_curves}, we only need to show that $ \PrIv[\hat {w}^j_\tau]  \nearrow \PrIv[\hat {u}_\tau]$ a.e. and $\PrIv[\hat {v}^j_\tau]  \searrow \PrIv[\hat {u}_\tau]$ for $\tau < \tau^+_{\hat u}$. But this follows from Lemma \ref{lma:decvmodel}, since $\hat {w}^j_\tau \nearrow \hat {u}_\tau$ a.e. and $\hat {v}^j_\tau  \searrow \hat {u}_\tau$ for $\tau < \tau^+_{\hat u}$.
\end{proof}

\subsection{Approximation of rays from below via subgeodesics of Kähler currents.} We prove the following result, which is the radial version of \cite[Proposition~2.15]{DLR20}:

\begin{theorem}\label{thm: approx_ray_below} Let $\{u_t\}_t \in \mathcal{R}^1$. Then for any $\epsilon > 0$ there exists subgeodesics  $[0,\infty) \ni t \mapsto u^\epsilon_t \in \mathcal{E}^1$ such that, $u^\epsilon_0 = 0$, $\omega_{u_t} \geq \epsilon \omega$, $u^\epsilon_t \leq u_t$ and 
\[
\Irad[u_t] - \Irad[u^\epsilon_t] \leq \epsilon \left|\Irad[u]\right|\,.
\]
In addition, $u^\epsilon_t \nearrow u_t$ a.e, as $\epsilon \to 0$ for any $t \geq 0$.
\end{theorem}

\begin{proof}
We can assume without loss of generality that $\sup_X u_t = 0$.
By \cite[Proposition~2.15]{DLR20} we have that $P((1+ \epsilon)u_t) \in \mathcal{E}^1$ for any $\epsilon > 0$.

We fix $\epsilon > 0$ and $t \geq 0$ momentarily. Let $[0,t] \ni l \mapsto v^{\epsilon,t}_l \in \mathcal{E}^1$ be the  geodesic connecting $0$ and $P((1+\epsilon)u_t)$. Then $l \mapsto  \frac{1}{1+\epsilon} v^{\epsilon,t}_l$ is a subgeodesic connecting $0$ and $\frac{1}{1+\epsilon}P((1+ \epsilon) u_t) \leq u_t$. Hence by the maximum principle, $\frac{1}{1+\epsilon} v^{\epsilon,t}_l \leq u_l$, i.e., $v^{\epsilon,t}_l \leq (1+\epsilon) u_l$, i.e., $v^{\epsilon,t}_l \leq P((1+\epsilon) u_l)$ for all $l \in [0,t]$. In particular, another application of the maximum principle gives that $\{v^{\epsilon,t}_l\}_{t \geq l}$ is a decreasing sequence for any $\epsilon > 0$ and $l \geq 0$ fixed. 

Next we notice the following: for any $t\geq 0$ and $l \in [0,t]$,
\begin{flalign*}
\frac{t}{l} (\mathrm{I}(u_l) - \mathrm{I}(v^{\epsilon,t}_l)) &= \mathrm{I}(u_t) - \mathrm{I}(v^{\epsilon,t}_t) = \mathrm{I}(u_t) - \mathrm{I}(P((1+\epsilon) u_t)) \\
& \leq \frac{1}{V} \int_X (u_t - P((1+\epsilon) u_t)) \,\omega_{P((1+\epsilon) u_t)}^n  \\
& \leq -\frac{\epsilon (1+\epsilon)^n}{V} \int_X u_t  \,\omega_{u_t}^n \leq \epsilon (1+\epsilon)^n (n+1)   |\mathrm{I}(u_t)|\,,
\end{flalign*}
where in the third inequality we have used \cite[Lemma~4.4]{DDNL5}, and in the very last inequality we have used that $\sup_X u_t = 0$.

Now linearity of $I$ along geodesic segments gives that  $v^\epsilon_l: = \lim_{t \to \infty} v^{\epsilon,t}_l \in \mathcal{E}^1$. Moreover, endpoint stability of geodesics gives that $\{v^\epsilon_l\}_l \in \mathcal{R}^1$ \cite[Proposition~4.3]{BDL17}. Lastly, the sequence of rays $\{v^\epsilon_t\}_t$ is increasing to $\{u_t\}_t$. In addition, by the maximum principle, $v^\epsilon_l \leq P((1+\epsilon)u_l)$. 

Finally, we introduce the subgeodesics $u^\epsilon_l := \frac{1}{1 + \epsilon} v^\epsilon_l \geq v^\epsilon_l$, for $l\geq 0$.
We immediately obtain that $\omega_{u^\epsilon_l} \geq \frac{\epsilon}{1 + \epsilon} \omega$. Since $v^\epsilon_l \leq P((1+\epsilon) u_l)$, we get that $v^\epsilon_l \leq u^\epsilon_l \leq u_l$. Lastly, 
\[
\Irad[u_t] - \Irad[u^{\epsilon}_t]\leq \Irad[u_t] - \Irad[v^{\epsilon}_t] \leq \epsilon (1+\epsilon)^n(n+1)  |\Irad[u_t]|\,.
\]
After re-scaling $\epsilon> 0 $, the result follows.
\end{proof}

\section{The closure of rays induced by test configurations}\label{sec:closure}
For this section, let $(T,h_T)$ be a fixed Hermitian line bundle on $X$ with smooth metric $h_T$. 

To start, we notice that a sublinear subgeodesic ray $\{r_t\}_t$ satisfies
\[
\lim_{t \to\infty}\frac{1}{t}\sup_X {r_t}=\tau^+_{\hat r} =\sup_{v \in X^{\textup{div}}_\mathbb Q} r^{\NA}(v)\,.
\]
The first equality already follows from Lemma~\ref{lem: suplinear} and the correspondence in Theorem~\ref{thm: max_test_curve_ray_duality}\Rom{1}. The last equality is pointed out in \cite[Lemma~4.3]{BBJ21}. In particular, the above constant(s) can be recovered using only the non-Archimedean data $r^\NA$.

Now we introduce the non-Archimedean analogue of Donaldson's $\mathcal{L}$-functionals. For each $k\geq 1$ and $\{r_t\}_t$ sublinear subgeodesic ray we define
\begin{flalign}\label{eq:LNADef}
\nonumber \Lk^{\NA}\{r_t\}:=& \frac{1}{V}\hTL\cdot \tau^+_{\hat r} +
\frac{1}{V}\int_{-\infty}^{\tau^+_{\hat r}} \left(h^0(X,T\otimes L^k\otimes \mathcal{I}(k\hat{r}_{\tau}))-h^0(X,T\otimes L^k)\right)\,\mathrm{d}\tau\\
=& -\frac{1}{V}\int_{-\infty}^{\infty}\tau\,\mathrm{d} h^0(X,T\otimes L^k\otimes \mathcal{I}(k\hat{r}_{\tau}))\,,
\end{flalign}
where we integrated by parts for Riemann--Stieltjes integrals on some interval $[\tau_0, \tau^+_{\hat r}+\epsilon]$, with $h^0(X,T\otimes L^k\otimes \mathcal{I}(k\hat{r}_{\tau_0}))=\hTL$ and $\epsilon \searrow 0$ \cite[Theorem~7.6]{Ap74}. Indeed, such $\tau_0 \in (-\infty, \tau^+_{\hat r})$ exists due to the openness theorem of Guan--Zhou \cite{GZ15}.

\begin{lemma}\Rom{1} \label{lem:LNAnormalization}
Let $\{r_t\}_t$ be a sublinear subgeodesic ray and $r'_t:=r_t + t c$ for some $c \in \mathbb R$. Then
\begin{equation}\label{eq:LNAnormalization}
\Lk^{\NA}\{r'_t\}=\Lk^{\NA}\{r_t\}+\frac{1}{V}\hTL\cdot c\,.
\end{equation}
\Rom{2} If $\{u_t\}_t$ and $\{v_t\}_t$ are sublinear subgeodesics such that $u_t \leq v_t$, then  $\Lk^{\NA}\{u_t\} \leq \Lk^{\NA}\{v_t\}$. 
\end{lemma}
\begin{proof}
\Rom{1} is obvious. Let us argue \Rom{2}. One can see that
\begin{flalign*}
\Lk^{\NA}\{v_t\}-\Lk^{\NA}\{u_t\}  =  \frac{1}{V}\int_{-\infty}^{\tau^+_{\hat u}} \left(h^0(X,T\otimes L^k\otimes \mathcal{I}(k\hat{v}_{\tau}))-h^0(X,T\otimes L^k\otimes \mathcal{I}(k\hat{u}_{\tau}))\right)\,\mathrm{d}\tau\\
+\frac{1}{V}\hTL\cdot (\tau^+_{\hat v}-\tau^+_{\hat u}) +  \frac{1}{V}\int_{\tau^+_{\hat u}}^{\tau^+_{\hat v}} \left(h^0(X,T\otimes L^k\otimes \mathcal{I}(k\hat{v}_{\tau}))-h^0(X,T\otimes L^k)\right)\,\mathrm{d}\tau\,.
\end{flalign*}
To conclude, one observes that both the first and second lines are positive quantities.
\end{proof}

Next we provide an important estimate for the radial Monge--Amp\`ere energy of approximable rays, in terms of $\Lk^{\NA}$.

\begin{prop}\label{prop:upperbound}
Let $\{r_t\}_t \in \mathcal R^1$ be an approximable ray, i.e., $\{\hat r_\tau\}_\tau$ is $\veq$-maximal. Then
\begin{equation}
\Irad[r_t]\geq \varlimsup_{k\to\infty} \frac{n!}{k^n} \Lk^{\NA}\{r_t\}\,.
\end{equation}
\end{prop}
\begin{proof}
By Lemma~\ref{lem:LNAnormalization}, we may assume that $\tau^+_{\hat r}=0$. By Lemma~\ref{lma:contimass}, we have $\int_X \omega_{\hat{r}_{\tau}}^n>0$ for $\tau<0$.
Moreover, $\hat{r}_{\tau}$ is $\veq$-model for all $\tau \in \mathbb R$ by Theorem~\ref{thm: bijection_I_max_test_curves}. We can calculate
\begin{equation}\label{eq:Icalculation}
\begin{split}
\Irad[r_t]=&\frac{1}{V}\int_{-\infty}^{0} \Big(\int_X \omega_{\hat{r}_{\tau}}^n-V \Big) \,\mathrm{d}\tau =\int_{-\infty}^{0} \Big(\frac{\int_X \omega_{\hat{r}_{\tau}}^n}{\int_X \omega^n}-1 \Big) \,\mathrm{d}\tau\\
\geq  & \int_{-\infty}^{0}  \varlimsup_{k\to\infty}\bigg(\frac{h^0(X,T\otimes L^k\otimes \mathcal{I}(k \hat{r}_{\tau}))}{\hTL}-1 \bigg) \,\mathrm{d}\tau\\
\geq &\varlimsup_{k\to\infty}\int_{-\infty}^{0} \bigg(\frac{h^0(X,T\otimes L^k\otimes \mathcal{I}(k \hat{r}_{\tau}))}{\hTL}-1\bigg) \,\mathrm{d}\tau= \varlimsup_{k\to\infty}\frac{n!}{k^n} \Lk^{\NA}\{r_t\}\,,\\
\end{split}
\end{equation}
where the first line we used \eqref{eq: I_RWN_form}, in the second line we used the Riemann--Roch theorem together with Theorem~\ref{thm:genBon}, and in the third line we used Fatou's lemma.
\end{proof}

Using the results of Section~\ref{subsec:expfilt}, in the next lemma we provide a formula that will be an important technical ingredient (closely related to \cite[Theorem~1.1]{Bern17}). Recall the definition of the Hilbert map from \eqref{eq:Hilbkdef}:

\begin{lemma}\label{lem: quantum_Hilbert_slope} Let $\{r_t\}_t$ be a sublinear subgeodesic ray with such that $r_t \leq 0$ for all $t \geq 0$. Let
\[
\lambda_H(s):=\varlimsup_{t\to\infty}t^{-1}\log\Hilb_k(r_t)(s,s)\,,\quad s\in \HTL\,.
\]
Then for any $s\in \HTL$,
\begin{equation}\label{eq:lambdaHcalc_lemma}
\lambda_H(s)=-k\sup\left\{\,\lambda<0:s\in H^0(X,T\otimes L^k\otimes \mathcal{I}(k\hat{r}_{\lambda}))\,\right\}  < \infty\,.
\end{equation}
\end{lemma}
\begin{proof} 
Let $\lambda<\sup\{\tau<0:s\in H^0(X,T\otimes L^k\otimes \mathcal{I}(k\hat{r}_{\tau})\}$. Let $C:=\int_X h^k(s,s) e^{-k\hat{r}_{\lambda}}\,\omega^n<\infty$. By definition, for any $t\geq 0$ we have $\hat{r}_{\lambda}\leq r_t-t\lambda$, so $C\geq \int_X h^k(s,s) e^{-k(r_t-t\lambda)}\,\omega^n$. As a result, $\lambda_H(s)\leq -k\lambda$, hence
\[
\lambda_H(s)\leq-k\sup\left\{\,\lambda<0:s\in H^0(X,T\otimes L^k\otimes \mathcal{I}(k\hat{r}_{\lambda}))\,\right\}\,.
\]
Now we prove the reverse inequality. We fix $p>\lambda_H(s)$ and $\epsilon >0$ satisfying $p-\epsilon>\lambda_H(s)$. We can find $t_0>0$ such that
\[
 \int_X h^k(s,s) e^{-kr_t}\,\omega^n<e^{(p-\epsilon)t}\,,\quad t \geq t_0\,.
\]
Hence $\int_0^{\infty}e^{-pt}\int_X h^k(s,s) e^{-kr_t}\,\omega^n\,\mathrm{d}t<\infty$.
By Tonelli's theorem, this is equivalent to
\begin{equation}\label{eq:finiteness}
\int_X h^k(s,s)\left(\int_0^{\infty}e^{-pt} e^{-kr_t}\,\mathrm{d}t\right)\,\omega^n<\infty\,.
\end{equation}

Before proceeding further, we show that $\lambda_H(s) \geq - k\tau^+_r = - k\lim_t \frac{\sup_X r_t}{t}$. Indeed, we get this after letting $t\to \infty$ in the following inequality:
\[
\frac{1}{t}\log{\Hilb_k(r_t)}(s,s) = -k \frac{\sup_X r_t}{t} + \frac{1}{t}\log{\Hilb_k(r_t - \sup_X r_t)}(s,s) \geq -k \frac{\sup_X r_t}{t} + \frac{1}{t}\log{\Hilb_k(0)}(s,s)\,.
\]
As a result, $\displaystyle -p/k < -k^{-1} \lambda_H(s) \leq  \lim_{t \to \infty} t^{-1}\sup_X r_t$, giving that $\hat r_{-p/k}$ is not identically equal to $-\infty$.

Next, for any $x \in X$ such that $\hat r_{-p/k}(x)$ is finite, we claim that
\begin{equation}\label{eq:uniflower}
\int_0^{\infty}e^{-pt} e^{-kr_t(x)}\,\mathrm{d}t\geq e^{-p-k}e^{-k\hat{r}_{-p/k}(x)}\,.
\end{equation}

By definition of $\hat{r}_{\tau}$, we can find $t_0>0$ so that $\hat r_{-p/k}(x)+1 \geq  r_{t_0}(x)+pk^{-1} t_0$.
Since $t \mapsto r_t(x)$ is decreasing,  we have $\hat{r}_{-p/k}(x)+pk^{-1}+1\geq r_{t}(x)+pk^{-1} t$ for $t\in [t_0,t_0+1]$. 
Hence
\[
\int_0^{\infty}e^{-pt} e^{-kr_t(x)}\,\mathrm{d}t\geq \int_{t_0}^{t_0 + 1}e^{-pt} e^{-kr_t(x)}\,\mathrm{d}t\geq \int_{t_0}^{t_0+1}e^{-k\hat{r}_{-{p}/{k}}(x)}e^{-p-k}\,\mathrm{d}t\geq e^{-p-k}e^{-k\hat{r}_{-{p}/{k}}(x)}\,.
\]
This proves the claim \eqref{eq:uniflower}.
So by \eqref{eq:finiteness} and the claim, $\int_X h^k(s,s) e^{-k\hat{r}_{-p/k}}\,\omega^n<\infty$, 
hence $p\geq -k\sup\{\,\lambda<0:s\in H^0(X,T\otimes L^k\otimes \mathcal{I}(k\hat{r}_{\lambda}))\,\}$, concluding the proof.
\end{proof}

Next we link the non-Archimedean functional $\mathcal L_k^\NA$  to the classical functional $\mathcal L_k$ for sufficiently positive subgeodesic rays:
\begin{prop}\label{prop:Lkarctononarc}
Let $\{r_t\}_t$ be a sublinear subgeodesic ray and $\delta>0$ such that  $r_t \in \mathcal{E}^{1}, $ and $\omega_{r_t}\geq \delta\omega$ for all $t\geq 0$. Then there exists $k_0(\delta) > 0$ such that $t \to \Lk(r_t)$ is convex, moreover
\begin{equation}
    \Lk^{\NA}\{r_t\}=\lim_{t\to\infty} \frac{1}{t}\Lk(r_t)\,, \quad k \geq k_0\,.
\end{equation}
\end{prop}

As it will be clear from the proof below, in case $T = K_X$ and $h_T$ is dual to $\omega^n$, one can omit the condition $\omega_{r_t}\geq \delta\omega$ from the assumptions.

\begin{proof}
By Lemma~\ref{lem:LNAnormalization}, we may assume that $\sup_{X} r_t= 0$ for any $t\geq 0$. 

By Lemma~\ref{lem: quantum_Hilbert_slope} for $f\in H^0(X,T\otimes L^k)$ we have that
\begin{equation}\label{eq:lambdaHcalc}
\lambda_{\Hilb_k}(f)=\varlimsup_{t\to\infty}t^{-1}\log{\Hilb_k(r_t)}(f,f)=-k\sup\left\{\,\lambda<0 : f\in H^0(X,T\otimes L^k\otimes \mathcal{I}(k\hat{r}_{\lambda}))\,\right\}\,.
\end{equation}
In particular, for $\lambda\geq 0$,
\begin{equation}\label{eq: filtration_identity}
\mathcal F^{\Hilb_k}_\lambda := \left\{\,f\in H^0(X,T\otimes L^k):\lambda_{\Hilb_k}(f)\leq \lambda\,\right\}=H^0(X,T\otimes L^k\otimes \mathcal{I}_{-}(k\hat{r}_{-\lambda/k}))
\end{equation}
where $\mathcal{I}_{-}(k\hat{r}_{\tau}):=\bigcap_{\lambda<\tau}\mathcal{I}(k\hat{r}_{\lambda})$, and $\mathcal F^{\Hilb_k}_\lambda$ is the filtration associated to $\lambda_{\Hilb_k}$, defined in \eqref{eq: Filt_def}.

As ${\Hilb_k(r_s)}$ is increasing in $s$, ${\Hilb_k(r_s)}^*$ is decreasing in $s$, hence the exponent $\lambda_{\Hilb_k^*}$ of ${\Hilb_k(r_s)}^*$ on $H^0(X,T\otimes L^k)^*$ is bounded above. Moreover, the family $(\Hilb_k(r_t))_{t\geq 0}$ is positive when $k \geq k_0(\delta)$ by Theorem~\ref{thm:Bernconvex}. As a result, $t \to \Lk(r_t)$ is convex (see the comments after Lemma~\ref{lma:geoderay}) and the conditions of Theorem~\ref{thm: detequality} are satisfied to imply that
\[
\begin{split}
\lim_{t\to\infty} \frac{1}{t}\log\left(\frac{\det{\Hilb_k(r_t)}}{\det{\Hilb_k(r_0)}}\right)=& \int_{0}^{\infty}\lambda\,\mathrm{d}h^0(X,T\otimes L^k\otimes \mathcal{I}_{-}(k\hat{r}_{-\lambda/k}))\\
=& k\int_{0}^{\infty}\lambda\,\mathrm{d}h^0(X,T\otimes L^k\otimes \mathcal{I}_{-}(k\hat{r}_{-\lambda}))
\end{split}
\]
for $k \geq k_0(\delta)$, where in the first line we also used \eqref{eq: filtration_identity}. As 
$\mathcal{I}(k\hat{r}_{\tau})\subseteq \mathcal{I}_{-}(k\hat{r}_{\tau})\subseteq \mathcal{I}(k\hat{r}_{\tau-\epsilon})$
for any $\epsilon>0$, and $\mathcal F^{\Hilb_k}_\lambda$ can only have finitely many jumping numbers, we get
\[
    \lim_{t\to\infty} \frac{1}{t}\Lk(r_t) = - \frac{1}{kV} \lim_{t\to\infty} \frac{1}{t}\log\left(\frac{\det{\Hilb_k(r_t)}}{\det{\Hilb_k(r_0)}}\right)=-\frac{1}{V}\int_{-\infty}^{0}\lambda\,\mathrm{d}h^0(X,T\otimes L^k\otimes \mathcal{I}(k\hat{r}_{\lambda}))\,.
\]
Comparing with \eqref{eq:LNADef}, the proof is finished.
\end{proof}

Before proceeding, we recall the following basic lemma:

\begin{lemma}\label{lem: convex_convergence_slope} Let $I \subseteq \mathbb{R}$ be an open interval. Let $f_j,f: I \to \mathbb{R}$ ($j\geq 1$) be convex functions such that $f_j \to f$ pointwise. Then for all $x \in I$, we have that
\[
f'_-(x) \leq \varliminf_{j\to\infty} {f'_j}_-(x) \leq \varlimsup_{j\to\infty} {f'_j}_+(x)\leq  f'_+(x)\,.
\]
\end{lemma}
\begin{proof}
Due to convexity, for $h>0$ small enough, we have that
$(f_j(x-h) - f_j(x))/(-h)\leq {f_j'}_-(x)$.
Letting $j \to \infty$, we arrive at $(f(x-h) - f(x))/(-h) \leq \varliminf_{j\to\infty} {f_j'}_-(x)$.
Now $h \to 0$ gives the first inequality. The other inequality follows similarly.
\end{proof}

\begin{theorem}\label{thm:lowerboundmisbyI}
Let $\{r_t\}_t \in \mathcal{R}^{1}$. Then
 \begin{equation}\label{eq:generalray}
 \varliminf_{k\to\infty} \frac{n!}{k^n} \Lk^{\NA}\{r_t\} \geq  \Irad[r_t]\,.
 \end{equation}
\end{theorem}
\begin{proof}
First, we consider $\{v_t\}_t$ a sublinear subgeodesic such that $v_t \in \mathcal E^1$,  $v_t\leq 0$, and  $\omega_{v_t}\geq \delta\omega$ for all $t\geq 0$ and some $\delta>0$. 

By Theorem~\ref{thm:QuantI}, we have $\lim_{k\to\infty} \frac{n!}{k^n}\Lk(v_t)=\mathrm{I}(v_t)$ for $t\geq 0$.
So by Lemma~\ref{lem: convex_convergence_slope} above and Proposition~\ref{prop:Lkarctononarc},
\begin{equation}\label{eq:convexslopeell}
\Irad[v_t] = \lim_{t \to \infty} \frac{\mathrm{I}(v_t)}{t}\leq \varliminf_{k\to\infty} \frac{n!}{ k^{n}}\left(\lim_{t\to\infty} \frac{1}{t}\Lk(v_t)\right)=\varliminf_{k\to\infty} \frac{n!}{k^{n}}\Lk^{\NA}\{v_t\}\,.
\end{equation}
By Lemma~\ref{lem:LNAnormalization} it is enough to prove the theorem for $\{r_t\}_t\in \mathcal{R}^{1}$ with $\sup_X r_t =0$. By Theorem~\ref{thm: approx_ray_below}, we can find sublinear subgeodesics $\{v^k_t\}_t$ such that $v^k_t \in \mathcal E^1$, $v^k_t \nearrow r_t$ and $\omega_{v^k_t}\geq \delta_k\omega$ for all $t\geq 0$ and some $\delta_k \searrow 0$. Moreover,  $\Irad[v^k_t] \to \Irad[r_t]$. By monotonicity of $\Lk^{\NA}\{\cdot\}$ (Lemma~\ref{lem:LNAnormalization}) we have
\[
\varliminf_{k\to\infty} \frac{n!}{k^{n}}\Lk^{\NA}\{r_t\}\geq \varliminf_{k\to\infty} \frac{n!}{k^{n}}\Lk^{\NA}\{v^k_t\}\geq  \Irad[v^k_t]\,.
\]
Letting $k\to\infty$, we conclude \eqref{eq:generalray}.
\end{proof}

\begin{theorem}\label{thm: closure_of_test_conf}
Let $\{r_t\}_t \in \mathcal{R}^1$ with $\sup_X r_t = 0$. Then the following are equivalent:\vspace{0.1cm}\\
\vspace{0.1cm}\Rom{1} $\{r_t\}_t \in \overline{\mathcal{T}}$.\\
\vspace{0.1cm}\Rom{2} $\{r_t\}_t \in \overline{\mathcal{F}}$.\\
\vspace{0.1cm}\Rom{3} $\PrIv[\hat r_\tau]=\hat r_\tau$, for all $\tau \leq 0$. \\
\Rom{4} $\displaystyle \lim_{k\to\infty} \int_{-\infty}^0 \bigg( \frac{h^0(X,T \otimes L^k \otimes \mathcal{I}(k {\hat r}_\tau))}{h^0(X,T \otimes L^k)}  - 1\bigg)\,\mathrm{d}\tau =   \Irad[r_t] $.
\end{theorem}
This theorem proves most of Theorem~\ref{mainthm: on_rays}. The remaining point will be completed in Corollary~\ref{cor: Thm1.1cor}.
\begin{proof} First we show $\Rom{1} \implies \Rom{3} \implies \Rom{4} \implies \Rom{1}$. Then we show $\Rom{1} \implies \Rom{2} \implies \Rom{1}$.

Due to \cite[Example 3.3]{Xi19} or (Theorem \ref{thm: Pi_cont}), \Rom{1} implies that $\{r_t\}_t$ is approximable. This in turn is equivalent with \Rom{3} due to Theorem~\ref{thm: bijection_I_max_test_curves}. However \Rom{3} implies \Rom{4} due to Proposition~\ref{prop:upperbound} and Theorem~\ref{thm:lowerboundmisbyI}. 

Now we show that  \Rom{4}  implies  \Rom{1} . For this, let us consider the approximable ray $\{\Pi(r)_t\}_t \in \mathcal R^1$ (Theorem~\ref{thm:BBJ21projnaequal}). From the same result we know that $\widehat{\Pi(r)}_\tau = P[\hat r_\tau]$ for $\tau < 0$. In particular,
\[
h^0(X,T\otimes L^k \otimes \mathcal{I}(k {\widehat{\Pi(r)}}_\tau)) = h^0(X,T \otimes L^k \otimes \mathcal{I}(k {\hat r}_\tau))\,.
\]
From the direction  \Rom{3}  implies  \Rom{4}  already proved, we obtain that  \Rom{4}  holds for $\{{\Pi(r)}_t\}_t$ in the following manner:
\begin{equation} \label{eq: I_Pi_slope}
\Irad[\Pi(r)_t] = \lim_{k\to\infty} \int_{-\infty}^0 \bigg( \frac{h^0(X,T \otimes L^k \otimes \mathcal{I}(k {\hat r}_\tau))}{h^0(X,T \otimes L^k)}  - 1\bigg)\,\mathrm{d}\tau\,.
\end{equation}
Condition~\Rom{4} now gives $\Irad[{\Pi(r)}_t]=\Irad[r_t]$. 
Since $r_t \leq \Pi(r)_t$, this gives $r_t = \Pi(r)_t$ for $t \geq 0$. Since $\{\Pi(r)_t\}_t$ is approximable due to Theorem~\ref{thm:BBJ21projnaequal}, so is $\{r_t\}_t$, concluding  \Rom{1} .

Finally, since $\mathcal T \subseteq \mathcal F$, we obtain that  \Rom{1}  implies  \Rom{2} . For the other direction, it is enough to show that elements of $\mathcal F$ are approximable. However the rays of $\mathcal F$ are all $\veq$-maximal, due to Theorem~\ref{thm: filtration_I_max_testcurve}, so they are approximable due to Theorem~\ref{thm: bijection_I_max_test_curves}, proving  \Rom{1} .
\end{proof}

\begin{theorem}\label{thm: on_rays}
Let $\{r_t\}_t \in \mathcal{R}^1$ with $\sup_X r_t = 0$  for any $t\geq 0$. Then $\lim_{k\to\infty} \frac{n!}{k^n} \Lk^{\NA}\{r_t\}$ exists and can be estimated the following way
\begin{equation}\label{eq: I_NA_ineq}
\lim_{k\to\infty} \frac{n!}{k^n} \Lk^{\NA}\{r_t\}=\lim_{k\to\infty} \int_{-\infty}^0 \bigg( \frac{h^0(X,T\otimes L^k \otimes \mathcal{I}(k {\hat r}_\tau))}{h^0(X,T\otimes L^k)}  - 1\bigg)\,\mathrm{d}\tau \geq \Irad[r_t]\,.
\end{equation}
\end{theorem}

\begin{proof} Consider the approximable ray $\{\Pi(r)_t\}_t \in \mathcal R^1$. In the argument  \Rom{3}  implies  \Rom{4}   of the previous theorem, we actually showed that the limit on the left hand side of \eqref{eq: I_NA_ineq} exists and is equal to $
\Irad[\Pi(r)_t]$. The inequality now readily follows from the fact that $r_t \leq \Pi(r)_t$, implying $\Irad[r_t] \leq \Irad[{\Pi(r)}_t]$.
\end{proof}
For $\{r_t\}_t\in \mathcal{R}^1$, it is possible to introduce the non-Archimedean Monge--Amp\`ere energy in the following manner:
\begin{equation}
\mathrm{I}^{\NA}\{r_t\}:=\Irad[\Pi(r)_t]\,.
\end{equation}
In particular, when $\{r_t\}_t \in \overline{\mathcal T}$  we have  
$\mathrm{I}^{\NA}\{r_t\}=\Irad[r_t]$. Comparing with \eqref{eq: I_Pi_slope}, we obtain a new interpretation for the non-Archimedean Monge--Ampère energy:

\begin{coro}\label{thm: Non-Archimedean} For $\{r_t\}_t \in \mathcal{R}^1$ we have
\begin{equation}
\mathrm{I}^{\NA}\{r_t\} = \lim_{k\to\infty} \frac{n!}{k^n} \Lk^{\NA}\{r_t\}\,.
 \end{equation}
In particular, if $\sup_X r_t = 0$ for any $t\geq 0$, then
\[
\mathrm{I}^{\NA}\{r_t\} = \lim_{k\to\infty} \frac{n!}{Vk^n}\int_{-\infty}^0 \big( h^0(X,T \otimes L^k \otimes \mathcal{I}(k\hat{r}_\tau)-h^0(X,T\otimes L^k)\big)\,\mathrm{d}\tau\,.
\]
\end{coro}
This proves the second part of Theorem~\ref{main_thm: Non-Archimedean}.
\section{The closure of algebraic singularity types}\label{sec:cloalgsingtype}

We start with the following result about approximable bounded geodesic rays.
\begin{prop}\label{prop: volume_id_along_bounded_ray}
Let $\{r_t\}_t\in \overline{\mathcal T}\subseteq \mathcal{R}^1$ be a ray of bounded potentials. Then for all $\tau \in (\tau_{\hat{r}}^-,\tau_{\hat{r}}^+)$ we have
\begin{equation}\label{eq: volume_eq_ray}
\int_X \omega_{\hat{r}_{\tau}}^n=\lim_{k\to\infty}  \frac{n!}{k^n} h^0(X,T\otimes L^k \otimes \mathcal{I}(k \hat{r}_\tau)) \,.
\end{equation}
In particular the limit on the right hand side exists.
\end{prop}
\begin{proof}
Without loss of generality, we may assume that $\{r_t\}_t$ is sup-normalized, i.e., $\tau_{\hat{r}}^+=0$.

Using \eqref{eq: I_RWN_form}, Theorem~\ref{thm: closure_of_test_conf}(iv), and Fatou's lemma, we have the following estimate
\begin{flalign}\label{eq:weakerint}
\int_{\tau^-_{\hat r}}^0 \bigg(\frac{\int_X  \omega_{\hat{r}_{\tau}}^n}{V} -  1\bigg) \,\mathrm{d}\tau &= \Irad[r_t] = \lim_{k\to\infty} \int_{\tau^-_{\hat r}}^0  \bigg(\frac{h^0(X,T\otimes L^k \otimes \mathcal{I}(k \hat{r}_\tau))}{\hTL} - 1 \bigg) \,\mathrm{d}\tau\\
&\leq \int_{\tau^-_{\hat r}}^0 \varlimsup_{k\to\infty} \bigg(\frac{h^0(X,T\otimes L^k \otimes \mathcal{I}(k \hat{r}_\tau))}{\hTL} - 1 \bigg) \,\mathrm{d}\tau \,. \nonumber
\end{flalign}
Comparing with \eqref{eq:Icalculation} we arrive at
\begin{equation}\label{eq:weakerint2}
\int_{\tau^-_{\hat r}}^0 \bigg(\frac{\int_X  \omega_{\hat{r}_{\tau}}^n}{V} - 1 \bigg) \,\mathrm{d}\tau = \int_{\tau^-_{\hat r}}^0 \varlimsup_{k\to\infty} \bigg( \frac{h^0(X,T\otimes L^k \otimes \mathcal{I}(k \hat{r}_\tau))}{h^0(X,T\otimes L^k)} - 1 \bigg) \,\mathrm{d}\tau\,.
\end{equation}
Since each $\hat r_\tau$ is $\mathcal I$-model, by Theorem~\ref{thm:genBon} the integrand of the left hand side is greater or equal to the integrand on the right hand side, so
for almost every $\tau \in (\tau^-_{\hat r},0)$ we have 
\begin{equation}\label{eq: volume_ineq_ray}
\frac{1}{n!}\int_X \omega_{\hat{r}_{\tau}}^n=\varlimsup_{k\to\infty}  \frac{1}{k^n} h^0(X,T\otimes L^k \otimes \mathcal{I}(k \hat{r}_\tau)) \,.
\end{equation}
Due to \eqref{eq:weakerint} and \eqref{eq:weakerint2}, the conditions of Lemma~\ref{lma:convintimppwconv} below are satisfied for $I = (\tau^-_{\hat r},0)$, $I' = (-2,1)$ and $f_k$ being the integrand on the right hand side of \eqref{eq:weakerint}. 
Due to Lemma~\ref{lma:contimass}, the function $(\tau^-_{\hat r},0) \ni \tau \mapsto \int_X \omega_{\hat r_\tau}^n$ is continuous and decreasing. We conclude that the limsup in \eqref{eq:weakerint} is a limit for all $\tau\in(\tau_{\hat{r}}^-,0)$ and \eqref{eq: volume_eq_ray} holds as desired.
\end{proof}

\begin{lemma}\label{lma:convintimppwconv}
Let $I,I'\subseteq \mathbb{R}$ be two bounded open intervals, and  $f_k:I\rightarrow I'$ for $k\in \mathbb{N}$ be a sequence of decreasing functions. Suppose that
\begin{equation}\label{eq:intclimint}
\lim_{k\to\infty}\int_{I} f_k\,\mathrm{d}\lambda=\int_I \varlimsup_{k\to\infty} f_k\,\mathrm{d}\lambda\,,
\end{equation}
where $\mathrm{d}\lambda$ is the Lebesgue measure. Denote $f:=\varlimsup_{k\to\infty} f_k$. Then the limit $\lim_{k\to\infty}f_k$ exists at each point of right continuity of $f$. In particular, $f(x)=\lim_{k\to\infty}f_k(x)$ for a.e. $x \in I$.
\end{lemma}
The proof of this lemma is due to Fan Zheng.
\begin{proof}
Without loss of generality, we may assume that $I=(0,1)$, $I'=(0,1)$. 
Let $x\in (0,1)$, such that 
\[
a:=f(x)-\varliminf_{k\to\infty} f_k(x)>0\,.
\]
We assume that $f$ is right continuous at $x$, to obtain a contradiction. There exists $\delta>0$, so that on $[x,x+\delta]$, $f>f(x)-a/2$.

We may take a subsequence $g_k$ of $f_k$ so that $g_k(x)\to f(x)-a$.
We automatically have
\begin{equation}\label{eq: g_k_id}
\lim_{k\to\infty}\int_0^1 g_k\,\mathrm{d}\lambda=\int_0^1 f\,\mathrm{d}\lambda\,,\quad \varlimsup_{k\to\infty} g_k\leq f\,.
\end{equation}
We deduce the estimates
\[
\varlimsup_{k\to\infty}\int_x^{x+\delta}g_k\,\mathrm{d}\lambda\leq \varlimsup_{k\to\infty} \delta g_k(x)\leq \delta f(x)-\delta a\leq \int_x^{x+\delta}f\,\mathrm{d}\lambda-\frac{\delta a}{2}\,.
\]
By Fatou's lemma, on the complement $S:=(0,1) \setminus[x,x+\delta]$ we have
\[
\varlimsup_{k\to\infty}\int_S g_k\,\mathrm{d}\lambda\leq \int_S\varlimsup_{k\to\infty} g_k\,\mathrm{d}\lambda\leq \int_S f\,\mathrm{d}\lambda\,. 
\]
Adding these estimates, we get $\displaystyle\varlimsup_{k\to\infty}\int_0^1 g_k\,\mathrm{d}\lambda\leq \int_0^1 f\,\mathrm{d}\lambda-\frac{\delta a}{2}$,
 contradicting \eqref{eq: g_k_id}.
 \end{proof}

Next compare the arithmetic and non-pluripolar volumes of arbitrary $\omega$-psh functions:

\begin{prop}\label{prop: analvol_arithvol}
For $\varphi\in \PSH(X,\omega)$, the limit $\displaystyle\lim_{k\to\infty}  k^{-n} \hTLm$ always exists. Moreover,
\begin{equation}\label{eq: PSH_lim_eq}
\frac{1}{n!}\int_X \omega_{\varphi}^n\leq \frac{1}{n!}\int_X \omega_{\PrIv[\varphi]}^n  = \lim_{k\to\infty}  \frac{1}{k^n} \hTLm\,.
\end{equation}
\end{prop}

\begin{proof} We note that $\hTLm = h^0(X,T \otimes L^k \otimes \mathcal I(k \PrIv[\varphi]))$, hence we can assume that $\varphi$ is $\veq$-model by replacing $\varphi$ with $\PrIv[\varphi]$. Further, due to Theorem~\ref{thm:genBon}, we can also assume that $\int_X \omega_\varphi^n = \int_X \omega_{\PrIv[\varphi]}^n>0$.

By \cite[Lemma~4.3]{DDNL5}, $\tau\mapsto P(\tau \varphi)$ is well defined for $\tau\in [0,1+\delta)$, where $\delta>0$ is a small constant depending on $\int_X \omega^n_\varphi>0$. We consider the bounded $\veq$-maximal test curve $\{\psi_{\tau}\}_\tau$ with $\tau^-_\psi=-1-\delta$ and $\tau^+_\psi=0$ (for notations, see Theorem~\ref{thm: max_test_curve_ray_duality}\Rom{3} and Definition~\ref{def: test_curves}) such that 
\[
\psi_{\tau}:=\PrIv[P((1+\delta+\tau)\varphi)]\,, \quad \tau\in [-1-\delta,0)
\]
and $\psi_0 := \lim_{\tau \nearrow 0} \psi_\tau$. Since $\varphi$ is $\mathcal I$-model, we conclude that $\psi_{-\tau} = \varphi$.
By Proposition~\ref{prop: volume_id_along_bounded_ray}, for $\tau\in [-1-\delta,0)$, we have
\begin{equation}
\frac{1}{n!}\int_X \omega_{\psi_\tau}^n= \lim_{k\to\infty}  k^{-n} h^0(X,T\otimes L^k \otimes \mathcal{I}(k \psi_\tau))\,. 
\end{equation}
Since $\psi_{-\delta} = \varphi$, plugging $\tau = -\delta$ in the above formula, we conclude that the limit on the right hand side of \eqref{eq: PSH_lim_eq} exists. Moreover this limit is equal to $\int_X \omega_{\PrIv[\varphi]}^n$.
\end{proof}

Next we characterize equality in \eqref{eq: PSH_lim_eq}, in the presence of positive mass.

\begin{prop}\label{prop: eq_Bonavero}Let $\varphi \in \PSH(X,\omega)$ with $\int_X \omega_\varphi^n >0$. Then $P[\varphi]=\PrIv[\varphi]$ if and only if 
\begin{equation}\label{eq: gen_Bonavero}
\lim_{k\to\infty} \frac{1}{k^n} \hTLm = \frac{1}{n!} \int_X \omega_{\varphi}^n.
\end{equation}
\end{prop}

\begin{proof} Since $\int_X \omega_{P[\varphi]}^n=\int_X \omega_{\varphi}^n>0 $ and $\mathcal I(k\varphi)= \mathcal I(kP[\varphi])$, we can assume that $\varphi$ is model, i.e., $\varphi = P[\varphi]$. If $\varphi$ is $\veq$-model, i.e., $\PrIv[\varphi]=\varphi$, then \eqref{eq: gen_Bonavero} follows from \eqref{eq: PSH_lim_eq}.

If \eqref{eq: gen_Bonavero} holds, then \eqref{eq: PSH_lim_eq} implies that $\int_X \omega^n_\varphi = \int_X \omega^n_{\PrIv[\varphi]}>0$.
Since $\varphi \leq \PrIv[\varphi]$ and $\varphi$ is model, \cite[Theorem~3.12]{DDNL2} gives $\PrIv[\varphi]=\varphi$, as desired.
\end{proof}

Finally, we state our last main result, collecting many of our previous findings:

\begin{theorem}\label{thm: arith_plurip_vol_deficit} For $u \in \PSH(X,\omega)$ we have 
\begin{equation}\label{eq: difference_arith_anal_volume}
\lim_{k\to\infty} \frac{h^0(X,T \otimes L^k \otimes \mathcal{I}(ku))}{k^n}=\frac{1}{n!}\int_X \omega_{\PrIv[u]}^n \geq \frac{1}{n!}\int_X \omega_u^n\,.
\end{equation}
Moreover, when $\int_X \omega_u^n>0$, we have equality in the above estimate if and only if one of the the following equivalent conditions hold:\vspace{0.1cm}\\
\Rom{1} $\displaystyle\lim_{k\to\infty} \frac{h^0(X,T \otimes L^k \otimes \mathcal{I}(ku))}{k^n} = \frac{1}{n!}\int_X \omega_u^n$. \vspace{0.1cm}\\
\Rom{2} $P[u]=\PrIv[u]$.\vspace{0.1cm}\\
\Rom{3} $[u]$ is $d_\mathcal{S}$-approximable by the quasi-equisingular sequence $[u_j]$ (see \eqref{eq: quasi_eq_sing_def}).\vspace{0.1cm}\\
\Rom{4} $[u] \in \overline{\mathcal Z}$.\vspace{0.1cm}\\
\Rom{5} $[u] \in \overline{\mathcal A}$.
\end{theorem}
In particular, when $\int_X\omega_{\PrIv[u]}^n>0$ (i.e. $\mathrm{nd}(L,u)=n$ in the terminology of \cite{Cao14}), we have $\displaystyle\lim_{k\to\infty} k^{-n} h^0(X,T \otimes L^k \otimes \mathcal{I}(ku))>0$. This reproves \cite[Proposition~3.6]{Cao14} in the case of ample line bundles.

This theorem corresponds to Theorem~\ref{main_thm: arith_plurip_vol_deficit} in the introduction. Moreover, we can now complete the proof of Theorem 1.1 as well.
\begin{coro}\label{cor: Thm1.1cor} Let $\{r_t\}_t \in \mathcal R^1$ with $\sup_X r_t =0$. Then $\hat r_\tau = P[\hat r_\tau]$ and $\int_X \omega_{\hat r_\tau}^n >0$ for any  $\tau \in (-\infty,0)$. 
Moreover, condition $\Rom{3}$ of Theorem~\ref{thm: closure_of_test_conf} is equivalent to  
\[
\Rom{5}\,\lim_{k\to\infty} \frac{h^0(X,T \otimes L^k \otimes \mathcal{I}(k\hat r_\tau))}{k^n} =  \frac{1}{n!}\int_X \omega_{\hat r_\tau}^n\,,\quad \tau \in (-\infty, 0)\,.
\]
\end{coro}

\begin{proof}[Proof of Theorem~\ref{thm: arith_plurip_vol_deficit}] The inequality \eqref{eq: difference_arith_anal_volume} was proved in Proposition~\ref{prop: analvol_arithvol}. The equivalence of  \Rom{1}  and  \Rom{2}  was proved in Proposition~\ref{prop: eq_Bonavero}.  

That $[u]$ is $d_\mathcal{S}$-approximable by its quasi-equisingular sequence is equivalent to $[P[u]]$ being $d_\mathcal{S}$-approximable by its quasi-equisingular sequence (Indeed, $d_\mathcal S(u,P[u])=0$. Also, we have $V^k_u = V^k_{P[u]}$ in the language of Theorem~\ref{thm:qeqsingapp}, hence $[u_k]=[P[u]_k]$ for the corresponding quasi-equisingular approximations). As a result, Theorem~\ref{thm:anaappvmodel} immediately gives the equivalence between  \Rom{2}  and  \Rom{3} .

Trivially, $\Rom{3} \implies \Rom{4} \implies \Rom{5}$. To finish, it is enough to argue that  \Rom{5}  implies  \Rom{2} . As before, we can assume that $u$ is model, i.e., $P[u]=u$.

Let $[u_j] \in \mathcal A$ be such that $d_\mathcal S([u],[u_j]) \to 0$. Since each $[u_j]$ is analytic,  \Rom{2} holds for each $u_j$ (Proposition~\ref{prop: anal_sing_type_envelope}). Since  \Rom{2}  is equivalent to  \Rom{3}, we can replace each $u_j$ with a potential of the type \eqref{eq: quasi_eq_sing_def}, that is algebraic.

Further, after passing to a subsequence, we can assume that $d_\mathcal S([u_j],[u_{j+1}]) \leq C^{-2j}$, where $C>1$ is the constant of \cite[Proposition~3.5]{DDNL5}. Let $v^l_j :=\max(u_j,u_{j_1}, \ldots u_{j+l})$.  
Using the triangle inequality and \cite[Proposition~3.5]{DDNL5} we have
\begin{flalign*}
d_\mathcal S([u_j],[v_j^l]) &= d_\mathcal S([u_j],[\max(u_j,v^{l-1}_{j+1})]) \leq C d_\mathcal S([u_j],[v^{l-1}_{j+1}])\\ 
& \leq C\left(d_\mathcal S([u_j],[u_{j+1}]) + d_\mathcal S([u_{j+1}],[v^{l-1}_{j+1}])\right)\,.
\end{flalign*}
After iterating the above inequality $l$ times and observing that $d_\mathcal S([u_{j+l}], [v_{j+l}^0])=0$, we conclude that  
\begin{equation}\label{eq: v_j_u_j_C}
d_\mathcal S([u_j],[v_j^l]) \leq \sum_{k=j}^{j+l-1} C^{k+1-j} d_\mathcal S([u_k],[u_{k+1}]) = \sum_{k=j}^{j+l-1}  \frac{C^{k+1-j}}{C^{2k}}  \leq \frac{1}{C^{j-2}(C-1)}\,.
\end{equation}

Now let $w^l_j := \PrIv[v^l_j]$ and $\displaystyle w_j:= \lim_{l\to\infty} w^l_j$.  By Lemma~\ref{lem: max_of_equi_sing} below and Proposition~\ref{prop: anal_sing_type_envelope}  we get that $w^l_j$ is $\veq$-model and has the same singularity type as $v^l_j$. Moreover, by Lemma~\ref{lma:decvmodel} \Rom{3},  we get that $w_j$ is $\veq$-model.

Comparing with \eqref{eq: v_j_u_j_C}, we obtain that $d_\mathcal S([u_j],[w^l_j])=d_\mathcal S([u_j],[v_j^l]) \leq C^{2-j}(C-1)^{-1}$. Letting $l \to \infty$, and using \cite[Lemma~4.1]{DDNL5}, we arrive at $d_\mathcal S([u],[w_j]) \leq C^{2-j}(C-1)^{-1}$, i.e., $d_\mathcal S([u],[w_j])\to 0$ as $j \to \infty$. Since $\int_X \omega_u^n>0$, each $w_j$ and $u$ is model, we obtain that $\displaystyle u = \lim_{j\to\infty} w_j$ (\cite[Theorem~3.12]{DDNL2}).

Since $\{w_j\}_j$ is decreasing, by Lemma~\ref{lma:decvmodel} \Rom{1}  we obtain that $u = \PrIv[u]$. Since $u = P[u]$ by assumption, \Rom{2} follows.
\end{proof}

\begin{lemma} \label{lem: max_of_equi_sing}Let $u\in \PSH(X,\omega)$.
Suppose that 
$u_1,\ldots, u_l \in \PSH(X,\omega)$ are the potentials arising from the construction in \eqref{eq: quasi_eq_sing_def}. Then $\max(u_1,\ldots, u_l)\in \PSH(X,\omega)$ has analytic singularity type.
\end{lemma}
\begin{proof} Examining the expression \eqref{eq: quasi_eq_sing_def} one notices that for each $x \in X$ we can find (a common denominator) $m \in \mathbb N$ and an open neighborhood $x \in U_x \subseteq X$ such that  
$u_k - \frac{1}{m} \max_j \log |f^k_j|^2$ is locally bounded on $U_x$ for a finite number of holomorphic functions $f^k_j \in \mathcal O(U_x)$. Then $\max(u_1,\ldots, u_l) - \frac{1}{m} \max_{j,k} \log |f^k_j|^2$ is locally bounded on $U_x$ as well.
\end{proof}

\begingroup
\setstretch{1.1}
\setlength\bibitemsep{0pt}
\setlength\biblabelsep{0pt}
\printbibliography
\endgroup
\bigskip

\footnotesize
\noindent {\sc Department of Mathematics, University of Maryland}\\
{\tt tdarvas@umd.edu}\vspace{0.1in}\\
\noindent \textsc{Department of Mathematics, Chalmers Tekniska Högskola, G\"oteborg}\par\nopagebreak
\noindent   \texttt{xiam@chalmers.se}\par\nopagebreak
\end{document}